\documentclass[10pt]{article}
\usepackage{amssymb}
\usepackage{mathrsfs}
\usepackage{amsfonts}
\usepackage{dsfont}
\usepackage{bbding}
\usepackage{bm}
\usepackage{indentfirst}
\usepackage{colortbl,dcolumn}
\usepackage{color}
\usepackage{caption}
\usepackage{amsmath}
\usepackage{psfrag}
\usepackage{booktabs}
\usepackage{cite}
\usepackage[hidelinks]{hyperref}
\usepackage{cleveref}
\usepackage[affil-it]{authblk}
\usepackage[english]{babel}
\usepackage{blindtext}
\usepackage{lipsum}
\usepackage{enumitem}  
\usepackage{multirow}
\usepackage{nomencl} 
\makenomenclature
\makeatletter
\newcommand\tabcaption{\def\@captype{table}\caption}
\newcommand\figcaption{\def\@captype{figure}\caption}
\makeatother
\usepackage{tikz,pgf}
\usepackage{pgfplots}
\usetikzlibrary{calc}
\usetikzlibrary{shapes.misc,shapes.geometric,arrows.meta,positioning,shadows.blur,chains,scopes}

\usepackage{graphicx} 
\graphicspath{{figs/}}
\usepackage{amsthm}
\allowdisplaybreaks
\numberwithin{equation}{section}
\usepackage[top=1in, bottom=1in, left=0.8in, right=0.8in]{geometry}
\makeatletter
\newcommand{\mylabel}[2]{#2\def\@currentlabel{#2}\label{#1}}
\makeatother

\newcommand{\R}{\mathbb{R}}

\newcommand{\E}{\mathbb{E}}
\renewcommand{\P}{\mathbb{P}}

\newcommand{\dd}{\text{d}}
\newcommand{\abs}[1]{\left\vert#1\right\vert}
\newcommand{\mf}[1]{\mathcal{T}_1({f}( #1),h)}
\newcommand{\mg}[1]{\mathcal{T}_2({g^r}( #1), h)}

\newcommand{\mean}[1]{\mathbb{E}\lbrack #1\rbrack}
\theoremstyle{definition}
\newtheorem{thm}{Theorem}[section]
\newtheorem{defn}[thm]{Definition}
\newtheorem{lem}[thm]{Lemma}
\newtheorem{exm}[thm]{Example}
\newtheorem{prop}[thm]{Proposition}
\newtheorem{cor}[thm]{Corollary}
\newtheorem{rem}[thm]{Remark}

\newtheorem{assumption}[thm]{Assumption}

\allowdisplaybreaks
\begin{document}
\title{Weak error analysis for strong approximation schemes of SDEs with super-linear coefficients II:
finite moments and higher-order schemes
\footnotemark[2] \footnotetext[2]{
Wang was supported by the Natural Science Foundation of China (No.12471394, 12071488, 12371417). The Postdoctoral Fellowship Program of CPSF supported Zhao under Grant Number GZC2024205.
                \\
                E-mail addresses:
                zhaoyuying78@gmail.com,
                x.j.wang7@csu.edu.cn,
                zzhang7@wpi.edu
                }}
\author[a]{Yuying Zhao}

\author[a]{Xiaojie Wang}

\author[b]{Zhongqiang Zhang
\thanks{Corresponding author}}
\affil[a]{
    School of Mathematics and Statistics, HNP-LAMA, Central South University, Changsha, Hunan, P. R. China}
\affil[b]{
    Department of Mathematical Sciences, Worcester Polytechnic Institute, Worcester, MA 01609 USA}
\maketitle	
\begin{abstract}
This paper is the second in a series of works on weak convergence of one-step schemes for solving stochastic differential equations (SDEs) with one-sided Lipschitz conditions. 
It is known that the super-linear coefficients may lead to a blowup of moments of solutions and numerical solutions and thus affect the convergence of numerical methods. 
Wang et al. 
 (2023, IMA J. Numer. Anal.) have analyzed  
 weak convergence of one-step numerical schemes when solutions to SDEs have \textit{all finite moments}. Therein some modified Euler schemes have been discussed about their weak convergence orders.
In this work, we explore the effects of limited orders of moments on the weak convergence of a family of explicit schemes. 
The schemes are based on approximations/modifications of terms in the Ito-Talyor expansion.
We provide a systematic but simple way to establish weak convergence orders for these schemes. 
We present several numerical examples of these schemes and show their weak convergence orders. \\


	
\end{abstract}

\section{Introduction}\label{sec:intro}
In this work, we investigate weak convergence of first- and second-order explicit schemes for stochastic differential equations (SDEs) with non-globally Lipschitz coefficients of superlinear growth, following discussions in  \cite{Milstein2005numerical,wang2021weak}.
In \cite{wang2021weak}, an approximation theorem is established on the weak convergence orders of one-step numerical schemes with \textit{infinity} moments. 
The weak convergence of first-order numerical schemes has also been analyzed for SDEs with global Lipschitz coefficients \cite{Milstein2005numerical} and globally monotone coefficients \cite{wang2021weak,zhao2024weak}.
However,  several limitations in \cite{wang2021weak} 
should be addressed. The limitations include but are not limited to 
1)
requiring infinitely many moments of exact solutions and numerical solutions,  2) schemes of at most first-order weak convergence, and 3) not structure-preserving from the schemes. 

We will address the first two limitations by discussing 
1) relaxing requirements of all finite moments, which necessitate the moments of derivatives of solutions with respect to the initial condition; 
2) second-order schemes obtained from modifying classical second-order schemes for SDEs with Lipschitz coefficients. 
We will not focus on 3) while we discuss some possibilities in Remark \ref{rem:structure-preserving}. For 1), we extend the weak convergence theorem for one-step numerical schemes in \cite{wang2021weak}, where the authors assume the solutions to the SDEs have infinitely many moments 
before developing a weak error analysis with non-globally Lipschitz coefficients of superlinear growth. We will reduce the requirements of all finite moments of both exact and numerical solutions from the infinite order to a finite order for a $q$-th order numerical scheme ($0<q\leq2$). The reduction allows the design of numerical schemes of weak convergence when solutions have only a few moments. Also, we show the required moments of finite orders are critical for the desired convergence order. 
For 2), since only first-order schemes are considered in \cite{wang2021weak,chak2024regularity}, we discuss numerical schemes with second-order weak convergence, which haven't been explored for SDEs with non-globally Lipschitz coefficients. We modify schemes of weak order two for SDEs with Lipschitz coefficients and show that the weak convergence order of the modified schemes may degenerate to one for nonlinear SDEs with only limited moments. 
See Example \ref{exm:nonlin-diff-equ} for a representative example.

We omit a thorough review of schemes of weak convergence for SDEs with non-globally Lipschitz coefficients and refer interested readers to \cite{wang2021weak}. 



The main novelty and contributions of the work are summarized as follows.
\begin{itemize}
    \item 
\textit{Relaxed assumptions} (compared to  \cite{wang2021weak}). 
We relax one of the key assumptions in the fundamental theorem from \cite{wang2021weak}, which requires all moments for solutions to SDEs and their numerical approximations. The relaxation leads to wider applications of the fundamental theorem to practical numerical schemes. To establish a similar theorem with the relaxation, we investigate the solutions and their derivatives with respect to the initial data. After analyzing the moments of these derivatives, we provide a simpler proof (see Appendix \ref{app:gen-SDE-well-pose} and \ref{appen:deri-wrt-ini-val}) 
with a more relaxed assumption than in  \cite{wang2021weak}, which is inherited from  
\cite{Cerrai2001second,wang2021weak}.

\item \textit{Schemes of weak convergence of high orders and effects of moments of solutions}. 
Based on classical second-order schemes, we present several modified schemes (Examples \ref{exm:truncation}-\ref{exm:modified}) and discuss their weak convergence orders for nonlinear SDEs. We show in Examples 
\ref{exm:nonlin-diff-equ} (theoretically) and 
\ref{ex:numerical_non_dri_diff_critical_case} (numerically) that  the convergence order of the schemes of 
second-order weak convergence
become one when the moment conditions are not satisfied. 
\end{itemize}
The article is structured as follows. We introduce key notations and assumptions in Section \ref{sec:assumption-preliminary}. We also present an approximation theorem for SDEs with weaker assumptions than those in \cite{wang2021weak}. The proof for the improvement is found in Appendix \ref{append-proof-weak-conver-theorem}.
We introduce second-order weak convergence schemes in Section \ref{sec:second-order-scheme-prensented} and outline the necessary technical assumptions to achieve this. 
We present some numerical results in Section \ref{sec:numer-results}.
The essential proofs of the main results are presented
in  Section \ref{sec:weak-con-order-proof} while more details can be found in the appendix.
%
\section{Notations and Preliminaries}
\label{sec:assumption-preliminary}
We will use similar notations as in \cite{wang2021weak} and present them below for completeness.

We denote by  $|\cdot|$ the $l^2$-norm of vectors and matrices. The space of $k$-th continuously differentiable functions from $\R^d $ to $\R^m $ is denoted by
$\mathbf{C}^{k}(\R^d,\R^m)$.
For a multi-index $\alpha$, we denote the partial derivatives of $\nu \colon \R^d \rightarrow \R^l $ by
	$D^{\alpha} v(x)
	:=
	\frac{\partial^{\abs{\alpha}_1} 
		v(x)}{\partial x_{1}^{\alpha_{1}} 
		\cdots \partial x_{d}^{\alpha_{d}}}
	=\partial_{x_{1}}^{\alpha_{1}} 
	\cdots \partial_{x_{d}}^{\alpha_{d}} v (x), \abs{\alpha}_1=\sum_{i=1}^d\alpha_i.
 $
For a nonnegative integer $k$, we define
	$
	D^{k} \nu ( x ) := \left\{D^{\alpha} \nu ( x ) 
	: | \alpha | =k \right\}
	$  
	as the set of all partial derivatives of order $k$.  We also introduce $
	\left|D^{k}  \nu ( x ) \right|
	=\Big( \sum_{|\alpha|=k}\left|D^{\alpha} 
	\nu (x)\right|^{2} \Big)^{1 / 2}$.
 We also define $x \vee y := \max\{x, y\}$ and $ x \wedge y := \min\{x, y\}$.
 
For $ r \in \mathbb{N} $, 
The $ L^r $ norm of a $ \R^{d \times m} $-valued variable $Z \in L^r(\Omega; \R^{d \times m})$, with $ \E $ the expectation, will be denoted by  
$ \| Z \|_{L^r(\Omega;\R^{d \times m})} 
=(\E [|Z|^r])^{\frac{1}{r}} < \infty $.
Consider the  following It\^{o} stochastic differential equation:
\begin{equation}
	\label{eq:Problem_SDE}
		\dd X(t) = f ( X(t) ) \, \dd t 
		+  g(X(t)) \, 
           \dd W(t),
		 \quad t \in (0, T],
	\quad 
		X(0)  = x_{0} \in \mathbb{R}^d,
\end{equation}
where $f \colon \mathbb{R}^d \rightarrow \mathbb{R}^d$ is the drift coefficient, and
$ g \colon \mathbb{R}^d \rightarrow \mathbb{R}^{d\times m} $ 
is the diffusion coefficient. The process $W(t)$ is an $ m $-dimensional Wiener process on a complete probability space $ ( \Omega, \mathcal{ F }, \P ) $ with a filtration $ \mathcal{ F}_t $ and the initial data $x_0$
is deterministic for simplicity. 
\subsection{Conditions on coefficients of SDEs}
 We now outline the key assumption for the drift and diffusion coefficients.
Most items of the assumption are similar to those in 
\cite{wang2021weak} but we remove the item \eqref{eq:moment-all-cond} that is used in \cite{wang2021weak}. The removal leads to 
limited moments of solutions instead of all finite moments in \cite{wang2021weak}.
%
\begin{assumption}
\label{ass:coefficient_function_assumption}
\mylabel{ass:a1}{(A1)}
The drift coefficient  
$ f \in \mathbf{C}^{2q+2}( \mathbb{R}^{d};\mathbb{R}^{d})$ meets the following condition for $ r \geq 0 $, 
\begin{equation}
\label{ass:drift-f}
\sup_{x \in \mathbb{R}^{d}}
\frac{ | D^{\alpha} f(x) |}{1+ |x|^{(2r+1-j)\vee 0}} < \infty , \quad \abs{\alpha}_1=j,
\,\,\,  j \in \{ 0, 1, \ldots , 2q+2\};
\end{equation}
\mylabel{ass:a2}{(A2)}
The diffusion coefficient $ g \in \mathbf{C}^{2q+2} (\mathbb{R}^{d};  \mathbb{R}^{d \times m} ) $ satisfies the condition that there exists $ \rho \leq r+1 $  such that
\begin{equation}
\label{ass:diff-g}
\sup_{x \in \mathbb{R}^{d}}
\frac{| D^{\alpha} g (x) |}{1 +
 | x |^{(\rho-j)\vee0}} < \infty ,
 \quad  \abs{\alpha}_1 =j,
 \,\,\,  j \in \{ 0, 1, \ldots , 2q+2\};
\end{equation}
\mylabel{ass:a3}{(A3)}
Let $ p_0 \geq 2 $ be a sufficiently large number and there exists 
$ c_{p_0} \in \R $
such that
\begin{equation}
\label{eq:assum-f-g}
       \langle D  f(x) y , y  \rangle 
           + \frac{p_0 -1}{2}
           | D g(x) y|^{2}
            \leq c_{ p_0 } |y|^{2}, 
            \quad  
              x,y \in \mathbb{R}^{d}.
\end{equation}
\end{assumption}
\begin{rem}\label{rem:assumption-difference}
Compared to \cite{wang2021weak},  we relax several conditions as follows.
\begin{enumerate}
    \item 
 We relax condition $\rho\leq r$ to 
$\rho\leq r+1$ in \ref{ass:a2};
\item 
 In \ref{ass:a3}, we only require  one existing large $p_0\geq 2$ instead of all;
\item
We remove the following condition (the condition (A4) in  \cite{wang2021weak}), which is cited from \cite{Cerrai2001second}: 
there exist $ a_1 > 0 $ and $ \gamma  , c_1 \geq 0 $ such that for any $ x , y \in \mathbb{R}^{d}$ it holds 
\begin{equation}\label{eq:moment-all-cond}
\langle f(x+y)-f(x),y \rangle \leq
 - a_1 | y |^{2r+2} + c_1 (| x |^{\gamma} + 1).
\end{equation}
In the proof, we don't require $a_1>0$, which is required when $t\to\infty$ in \cite{Cerrai2001second}.   
\end{enumerate}
\end{rem}

We also assume  the following one-sided Lipschitz condition (globally monotone condition), e.g. in \cite{tretyakov2013fundamental,sabanis2016euler},

\noindent
\mylabel{assu:a3'-mono-condi}{(A3')}
$
\langle x - y , f(x) - f(y) \rangle
+ \frac{p_0' -1}{2}
|g(x)-g(y)|^{2} \leq c_{p_0'} | x - y |^{2},
\quad
\forall x,y \in \mathbb{R}^{d}$, $p_0'>1$. 
%
Young's inequality leads to  
\begin{equation}
 \label{eq:co_mono_condi} 
	\langle x, f(x) \rangle
	+ 
	\frac{p_{0}'-1-\varepsilon}{2}
       | g(x) |^{2} 
	\leq 
  c_0 + c'_{p_0'} |x|^{2} ,
	\quad
	\forall x \in \R^d,
\end{equation} 
where $\varepsilon>0$ is arbitrarily small  if $g(0)\neq 0$ and $ c_0 = \frac{|f(0)|^2}{2} +  \frac{(p_0' -1)(p_0'-1-\varepsilon )}
{2 \varepsilon} |g(0)|^2 $; and
$ c'_{p_0} = c_{p_0'} + \frac{1}{2} $. When $g(0)=0$, we have $\varepsilon=0$ and $c_0=\frac{\abs{f(0)}^2}{2}$. 
According to It\^{o}'s formula, the inequality \eqref{eq:co_mono_condi} ensures bounded moments \cite{mao2008stochastic,khasminskii2011stochastic}. Specifically, there exists a constant $ C > 0 $ such that
\begin{equation}
 \label{eq:esti_sol_SDE_wrt_initial}
  \E [|X^{0,x_0}(t)|^{p_1} ] 
      \leq C (1+|x_0|^{p_1}),
    \quad 
    2 \leq p_1 \leq p_{0}'-\varepsilon, 
    \quad t \in [0, T].
\end{equation}
We remark that assumption \ref{ass:a3} in Assumption \ref{ass:coefficient_function_assumption} implies this condition with $p_0'\geq p_0$. 
The inequality may be proved by using Taylor's expansion and applying the inequality \eqref{eq:assum-f-g}: 
\[ \langle x - y , \int_{0}^{1} D f \big (x+ s(y-x)\big) (x-y) \, \dd t \rangle
+ \frac{p_0 -1}{2}
|\int_{0}^{1} D g \big(x+ s (y-x)\big) (x-y) \, \dd s|^{2} \leq c_{p_0} | x - y |^{2}, \]
where $p_0$ can be arbitrarily large, e.g. in \cite{Cerrai2001second,wang2021weak},
$p_0'=p_0<\infty$. However, it may be true that  $p_0'>p_0$, which is the case in   Example 
\ref{exm:nonlin-diff-equ}. Although it is redundant when $p_0'=p_0$, we assume \ref{assu:a3'-mono-condi} 
in addition to Assumption \ref{ass:coefficient_function_assumption}.
 
%
Assumption \ref{ass:a1} gives us
\begin{equation}\label{contronal:f}
  |f(x) - f(y) | \leq C (1 + |x|^{2r} + |y|^{2r})
  |x - y|, \quad  x,y \in \R^d,
\end{equation}
which implies the polynomial growth 
\begin{equation}
\label{eq:f-growth}
  |f(x)| \leq C (1 + |x|^{2r+1}),
  \quad
   x \in \R^d . 
\end{equation}
From \eqref{eq:co_mono_condi} and \eqref{contronal:f}, we derive 
\begin{equation}
   \label{eq:g-polynomial-growth-con}
  | g(x) | \leq C ( 1+|x|^{r+1} ).
  \end{equation}
  Under the inequalities \eqref{contronal:f},  
\eqref{ass:diff-g}, and 
\eqref{eq:co_mono_condi}, we conclude that $r+1 \geq \rho$.
If $f(x)$ is a scalar polynomial that satisfies Assumption 
\ref{ass:coefficient_function_assumption}, it can be expressed as:
$f(x) = -c_{2k+1} x^{2k+1} + \sum_{i=0}^{2k} c_i x^i $,
where $c_{2k+1}>0$, $c_i\in \mathbb{R}$, $0\leq i\leq 2k$, and $r\geq k$ and $k \in \mathbb{N}$ (natural number).

\subsection{The approximation theorem}
We begin by defining $ Y(t,x;t+h) $ as the one-step approximation for the solution $ X(t,x;t+h)$, with an initial value $x$ on a uniform mesh on $ [0,T] $ and a step size given by $ h = \frac{T}{N} $, where $ N \in \mathbb{N} $. On the uniform mesh grid 
$ \{t_n = nh, n = 0, 1, \cdots, N\} $, we recursively construct the numerical approximations  $ \{Y_n\}_{0 \leq n \leq N} $ using the relation $ Y_{n+1} = Y(t_n,Y_n;t_{n+1}) $. To simplify notation, $C$ denotes a positive constant independent of the step size $h$.

We consider the convergence of $\mean{\varphi(Y_{n})}$ to $\mean{\varphi (X_{t_n})}$ where $\varphi$ and its derivatives have at most a polynomial growth at infinity. 
\begin{assumption}\label{assm:test-functions}
Let $ \varphi \colon \mathbb{R}^{d} \rightarrow \mathbb{R}^m $ ($m\geq 1$).
Assume that $\varphi \in \mathbb{G}^{2q+2}$, i.e., there exist constants $L>0$ and $\kappa \geq 1$ such that 
\begin{equation}
\label{eq:varphi-bound}
  |D^j \varphi (x)| \leq L (1 + |x|^{\kappa} ),
    \,
    j \in \{j=0, 1, 2, ..., 2q+2\}.
\end{equation}
\end{assumption}
The following regularity condition is needed for our main theorem.
\begin{lem}(c.f.\cite[Theorem 1.3.6]{Cerrai2001second})
\label{lem:deri-mom-bound}
    Let Assumption \ref{ass:coefficient_function_assumption} and  inequality \ref{assu:a3'-mono-condi}  
be fulfilled.
Then, the solution to SDE \eqref{eq:Problem_SDE}  $ X(t, x; s), 0 \leq t \leq s \leq T $ is $ 2q+2 $ times differentiable
with respect to the initial data $ x \in \R^d$ and for any 
$ q \geq 1 $
it holds that
\begin{align}
\label{eq:deri-esti-mom-bound}
 \sup_
 { x \in \mathbb{R}^d } \E [ |D X ( t, x; s ) |^{\mathds{P}} ]
 & \leq  C(T,\mathds{P}),
    \quad 
    1 \leq \mathds{P} \leq p_0, \\
 \sup_
 { x \in \mathbb{R}^d }
 \E [ |D^j X ( t, x; s ) |^{\mathds{P}} ]
 & \leq  C(T,\mathds{P},j),
\quad \mathds{P} \in \big[1,\frac{p_0}{((2r+1-j)\vee 0)+j}\big], \quad j = 2,\ldots,2q+2,
\end{align}
where $ C(T,\mathds{P},j) $ is a 
constant that depends on $ T,\mathds{P},j $. Here $D^j$'s refer to derivatives in $x$.
\end{lem}
This lemma is proved in Appendix \ref{appen:deri-wrt-ini-val}.

The following theorem 
has been known if the coefficients are Lipschtiz continuous \cite{Milstein1985Weak,Milstein1978method}
or  $p_0$ being arbitrarily large in \ref{ass:a3}  \cite{wang2021weak}. 
Here we relax the requirements of \cite{wang2021weak} as in  Remark \ref{rem:assumption-difference}.

\begin{thm}[The approximation theorem, c.f. \cite{wang2021weak}]
\label{thm:fundamental_weak_convergence}
Let Assumption \ref{assm:test-functions} hold.
Suppose 
\begin{enumerate}
[label=\textnormal{(\arabic*)}]
    \item[\mylabel{ass:i}{(i)}] 
(Coefficients of SDEs)
 Assumptions \ref{ass:coefficient_function_assumption} and  \ref{assu:a3'-mono-condi} hold;    
 \item[\mylabel{ass:ii}{(ii)}] 
(Local approximation error)
For some constants $C > 0$ and $\varkappa \geq 0$, and for any $ i_j \in \{ 1, 2, \cdots, d \} $ and 
$ x \in \R^{d} $, the one-step approximation $ Y(t,x;t+h) $ has the following orders of accuracy: 
   \begin{align}  \label{eq:error_of_general_one_step_approximation}
         \Big| \E \Big[ 
            \prod_{j=1}^s 
               \big( \delta_{X,x} \big)^{i_j}
                        \Big]
                 -  \E \Big[
                   \prod_{j=1}^s \big( \delta_{Y,x} \big)^{i_j} 
                         \Big]  \Big|
              &  \leq 
                C ( 1 + |x|^{\varkappa} ) h^{q+1}, 
          \,\,\,
            s=1, \ldots, 2q+1, 
          \\
       \label{eq:estimate_general_one_step_ref_solution}
       \Big \|  \prod_{j=1}^{2q+2}  
       \big( \delta_{X,x} \big)^{i_j}
       \Big \|_{ L^{2} ( \Omega;\mathbb{R} ) }
        & \leq
            C ( 1 + |x|^{\varkappa} )
            h^{ q + 1 }, 
      \\
\label{eq:estimate_general_one_step_approximate_solution}
          \Big \|  \prod_{j=1}^{2q+2}  
         \big( \delta_{Y,x} \big)^{i_j}
         \Big \|_{ L^{2} ( \Omega;\mathbb{R} ) }
         & \leq
         C ( 1 + |x|^{\varkappa} )
         h^{ q + 1 }, 
    \end{align}
where  we denote
      $ ( \delta_{X,x} )^{i_j} := 
                                X^{i_j}(t,x;t+h) -  x^{i_j},   \quad
       ( \delta_{Y,x} )^{i_j} := 
                                Y^{i_j}(t,x;t+h) -  x^{i_j}.
                                $ 
  \item[\mylabel{ass:iii}{(iii)}] 
(Moment bounds) 
There exist constants $ \beta \geq 1 $ and $ C > 0 $ such that for $ p \geq 2 \kappa  \vee (\beta \kappa + \varkappa) $, it holds 
   \begin{equation}
\label{eq:ass_approximation_moment_bound}
         \sup_{ N \in \mathbb{N} }
           \sup_{ 0 \leq n \leq N }
              \E [ |Y_{n}|^{p} ] 
                \leq  
                  C ( 1 + |Y_0|^{\beta p} ).
   \end{equation}  
\end{enumerate}
\noindent
Then we obtain a global weak convergence order of order $ q $, i.e., for some $C>0$ independent of $h$, 
   \begin{equation}
     \label{thm:convergence_order}
        \big | \E \big [
           \varphi \big ( X(t_0,X_0;T) \big ) 
                      \big ]
         - \E \big[
              \varphi \big ( Y(t_0,Y_0;t_N) \big )
                      \big ] \big |
                 \leq 
     C \big(1+|X_0|^{\beta(\beta \kappa + \varkappa)}\big)
              h^{q}.
   \end{equation} 	
\end{thm}
The proof of this theorem is presented in Appendix \ref{append-proof-weak-conver-theorem}.
\section{Schemes of weak convergence
based on the It\^{o}-Taylor expansion}
\label{sec:second-order-scheme-prensented}
Here we recall the Milstein–Talay method, a one-step approximation based on  the  It\^{o}-Taylor expansion \cite{Talay1984efficient}:
\begin{align}
\label{eq:one-step-Milstein–Talay}
X_{n+1} &= Y_{MT}(t_n,X_{n};t_n+h), \nonumber\\
Y_{MT}(t,x;t+h)
  & = 
  x + f(x) h
    + \sum_{r=1}^{m}
     \int_{t}^{t+h}
      g^r(x) 
    \, \dd W_r(s) \nonumber \\
   & +
    \sum_{r=1}^{m} \sum_{r_1=1}^{m} \int_{t}^{t+h} \int_{t}^{s} 
    \Lambda_{r_1} g^r(x) \, \dd W_{r_1}(s_1) \dd W_r(s) 
  +
   \sum_{r=1}^{m} \int_{t}^{t+h}
  \mathcal{L} g^r(x) h \, \dd W_r(s) 
  \nonumber \\
   & + 
   \sum_{r=1}^{m} 
   \int_{t}^{t+h} \int_{t}^{s} 
\big( \Lambda_{r} f (x) - \mathcal{L} g^r(x) \big)  \, \dd W_r(s_1) \dd s  
  + \frac{h^2}{2}
 \mathcal{L} f (x).
\end{align}
Here $     \Lambda_{r}  =  \sum_{i=1}^{d} g^{i, r} \frac{\partial}{\partial x_{i}}  
    \text{ and }
    \mathcal{L} = f^\top \frac{\partial}{\partial x} + \frac{1}{2} \sum_{r=1}^{m} \sum_{i, j=1}^{d} g^{i, r} g^{j, r} \frac{\partial^{2}}{\partial x_{i} \partial x_{j}}.$
The method is shown to have weak convergence of order two when 
$f$ and $g$ are Lipschitz continuous, see e.g. \cite{Milstein2004stochastic}.  
When Lipschitz continuity conditions are violated and coefficients have polynomial growth instead of linear growth, the Milstein–Talay method can explode at a certain time as in the Euler scheme \cite{Hutzenthaler2011Strong}.
Inspired by modified Euler schemes in the literature, we  propose the following scheme 
\begin{eqnarray}
\label{eq:second-order-new-scheme}
Y_{n+1}
  & = & 
  Y_{n} + \mf{ Y_n } h
    + \sum_{r=1}^{m} 
     \mg{ Y_n } 
    \, \Delta W_r(n) \nonumber \\
 &&  +
    \sum_{r=1}^{m} \sum_{r_1=1}^{m} \int_{t_n}^{t_{n+1}} \int_{t_n}^{s} 
   \mathcal{T}_3 \big(\Lambda_{r_1} g^r( Y_n), h \big) \, \dd W_{r_1}(s_1) \dd W_r(s)
  +
   \sum_{r=1}^{m} \int_{t_n}^{t_{n+1}}
   \mathcal{T}_4 \big( \mathcal{L} g^r(Y_n), h \big) h \, \dd W_r(s)   \nonumber \\
  &&  + 
   \sum_{r=1}^{m} 
   \int_{t_n}^{t_{n+1}} \int_{t_n}^{s} 
\mathcal{T}_5 \big( ( \Lambda_{r} f ( Y_n) - \mathcal{L} g^r(Y_n) ), h \big)  \, \dd W_r(s_1) \dd s
  + \int_{t_n}^{t_{n+1}} 
  \mathcal{T}_6 \big( \mathcal{L} f (Y_n), h \big) \frac{h}{2} \, \dd s,
\end{eqnarray}
where the maps $\mathcal{T}_i(\cdot)$ ($i = 1,\ldots,6 $) are   approximations of `$\cdot$' or zero and will be specified in Section \ref{sec:assumption-tame-function} and 
%
  $   \Delta W_{r}(n)
   = 
        W_r(t_{n+1})- W_r(t_{n}), \,
        n=0,1,2,\ldots,N-1.
$
 
\subsection{Assumptions on the maps for bounded moments}\label{sec:assumption-tame-function}
To obtain moment boundedness of solutions to the scheme  \eqref{eq:second-order-new-scheme}, we assume the following conditions:
\begin{assumption}\label{ass:tame_function_assumption}
\mylabel{ass:tame-h1}{(H1)} (Growth conditions on the $ \mathcal{T}_1 $ - $\mathcal{T}_2$) There exist positive constants $\gamma_1 $, $\gamma_2 $ and $ C $,  such that
 \begin{align}
  \label{con:balanced-EM-sublinear-fghx}
       |\mathcal T_{i}(z,h) |  
       & \leq
     Ch^{-\gamma_i} \wedge |z|, \quad  i=1,2,\quad  z \in \R^d. 
  \end{align} 
\mylabel{ass:tame-h2}{(H2)} (Approximation condition) There exist $ C,\,\tau, \,l_1 >0 $ such that for all $z \in \R^d,$
 \begin{equation}
 \label{con:balanced-EM-difference-f}
      | \mathcal T_{1}(z,h) - z| 
    \leq C h^{\tau} |z|^{l_1}.
 \end{equation}

%
\noindent
\mylabel{ass:tame-h3}{(H3)} (Growth conditions on the $ \mathcal{T}_3 $ - $\mathcal{T}_6$) There exist $\gamma_3 > \frac12$, $\gamma_4 > 1$ and $ C >0 $  such that for all $z \in \R^d,$
 \begin{align}
  \label{con:balanced-EM-sub-T3-5}
       |\mathcal{T}_3(z,h) |  
       & \leq
         Ch^{-\gamma_1} \wedge |z|, \quad
      |\mathcal{T}_4(z,h)|+ |\mathcal{T}_5(z,h) | 
        \leq
         Ch^{-\gamma_3} \wedge |z|,\quad   |\mathcal{T}_6(z,h)|
        \leq
          Ch^{-\gamma_4} \wedge |z| ,
  \end{align} 
where $\gamma_1$ is from \eqref{con:balanced-EM-sublinear-fghx}.
    
\end{assumption}

The condition \ref{ass:tame-h2} in Assumption \ref{ass:tame_function_assumption} may be replaced by the following condition:

\noindent
    \mylabel{ass:tame-h2p}{(H2')} (One-sided Lipschitz condition) For a sufficiently large $ {p}_{\mathcal{T}} \geq 2 $, there is a constant $ C > 0 $ such that,
 \begin{equation}
  \label{con:one-sided-lip-con-balanced}
       \langle 
	  x, 
	  \mf{x}
	   \rangle 	  
    +  \frac{ p_{\mathcal{T}}-1}{2}
     |\sum_{r=1}^{m}\mg{x}|^{2} \leq C (1+\abs{x}^2),  \quad 
     x \in \R^d.
 \end{equation}
It is shown in Section \ref{sec:moment-bound-second-order-scheme} that 
 Assumption \ref{ass:tame_function_assumption} is sufficient to obtain moment bounds of the scheme \eqref{eq:second-order-new-scheme}.

\subsection{Modified Euler schemes and the maps $\mathcal{T}_i$} \label{sec:discussion-modified-sch}

In this section, we 
revisit schemes of at most first-order weak convergence when $\mathcal{T}_i=0$, $i=3,4,5,6$ in \eqref{eq:second-order-new-scheme}. 
These choices lead to the modified Euler scheme:
\begin{equation}
\label{eq:Euler-scheme}
Y_{n+1} = 
  Y_{n}   
  + \mf{ Y_n } h
    + \sum_{r=1}^{m} 
      \mg{ Y_n }  
    \, \Delta W_r(n),
\end{equation}
where
 $ \Delta W_{r}(n):= W_r(t_{n+1})- W_r(t_{n}), $ 
 $ n \in \{ 0,1,2,\ldots,N-1 \}. $
Moreover, the maps $\mathcal T_{1}, \mathcal T_{2}$ satisfy
$ \mathcal{T}_1 \colon \R^d \times (0,1) \rightarrow \R^d $ and
$ \mathcal{T}_2 \colon \R^d \times (0,1) \rightarrow \R^d$. 
Recall that the maps $\mathcal{T}_i(z,h)$, $i=1, 2$ represent an approximation of $z$ when values of $\abs{z}$ are not excessively large. Moreover, $\mathcal{T}_i(z,h)$ is bounded in $z$ and is bounded above by a negative power of the time step size $h$. 

In Figure \ref{fig:tame-function-sketch}, we sketch some popular choices of the map $\mathcal{T}_i(z,h)$ for modified Euler schemes which have at most first-order weak convergence \cite{wang2021weak}.

\begin{center}
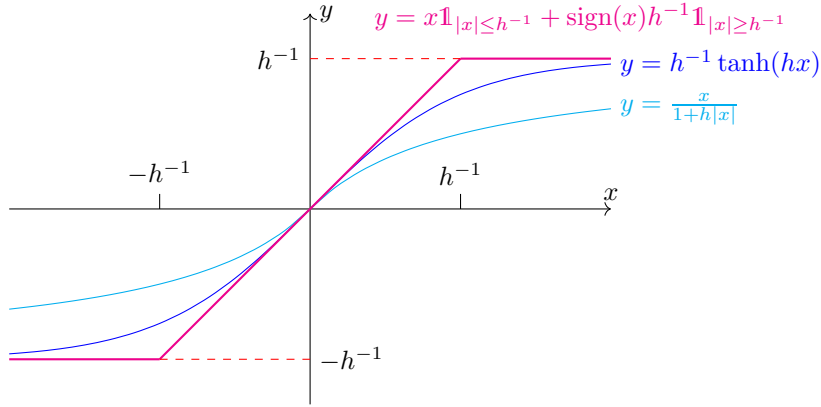

\begin{tikzpicture}[scale=0.18]
 \draw [->] (-20,0) -- (20,0) node [above] {$x$};
\draw [->] (0,-13) -- (0,13) node [right] {$y$};
\draw [ color = blue , smooth , domain =-20:20]
  plot (\x, {10*(exp(\x/10)- exp(-\x/10))/(exp(\x/10)+ exp(-\x/10))})   node[right] {$y=h^{-1}\tanh(h x)$};
  
  \draw [ color = cyan, smooth , domain =-20:20]
  plot (\x, {\x/(1 + abs(\x)/10)})   node[right] {$y=\frac{x}{1+h\abs{x} }$};
  
\draw [color= magenta]  (18,11) node [above] {$y=x \mathds{1}_{\abs{x}\leq h^{-1}}+ \text{sign}(x)h^{-1}\mathds{1}_{\abs{x}\geq h^{-1}}$};
  \draw [ color =magenta, smooth, thick , domain =-10:10]
  plot (\x, {\x}); 
  
  \draw [ color =red , dashed , domain =0:10]  plot (\x, 10) ;
  
   \draw [ color =magenta , solid,thick , domain =10:20]
  plot (\x, 10) ;

    \draw [ color =magenta, solid,thick , domain =-20:-10]
  plot (\x, -10) ;

    \draw [ color =red, dashed , domain =-10:0]
  plot (\x, -10) ;

  \draw [solid] (-10,0.0) -- (-10,1)  node[above] {$-h^{-1}$};
  \draw [solid](10,0) -- (10,1) node[above] {$h^{-1}$};
  
  \draw (0,-10) node[right] {$-h^{-1}$};
  \draw (0,10) node[left] {$h^{-1}$};
\end{tikzpicture}
\captionof{figure}{Illustration of approximations of truncation functions
$x \mathds{1}_{\abs{x}\leq h^{-1}}+\text{sign}(x)h^{-1}\mathds{1}_{\abs{x}\geq h^{-1}}$,  $   \frac{x}{1+h\abs{x}}$, and 
$h^{-1}\tanh(h x)$.}
\label{fig:tame-function-sketch}
\end{center}

%
We assume that 
$\mathcal T_{1}(z,h)$ and $\mathcal T_{2}(z,h)$ are deterministic{\footnote{It is possible to apply randomized functions but we only consider deterministic functions for brevity.}} and we assume either \ref{ass:tame-h1}, \ref{ass:tame-h2} in Assumption \ref{ass:tame_function_assumption} or \ref{ass:tame-h1}, \ref{ass:tame-h2p} hold, depending on the specific choices of $\mathcal T_{1}$ and $\mathcal T_{2}$.  
\begin{lem}
\label{lem:Euler-scheme-moment-bound}
Suppose Assumption \ref{ass:coefficient_function_assumption}, \ref{assu:a3'-mono-condi}, and 
 \ref{ass:tame-h1} and \ref{ass:tame-h2} in Assumption \ref{ass:tame_function_assumption} hold. 
 Consider the scheme \eqref{eq:Euler-scheme} and let $p_0'$ in \eqref{eq:co_mono_condi} be sufficiently large. Then there exist
$\beta_1 \geq 1 $ and $ C > 0 $ independent of $ h $ such that 
\begin{eqnarray} 
\label{eq:balanced-euler-result-moment-bound}
\E  \big[ 
\left|Y_{n}\right|^{p}
\big] 
& \leq  &
C 
\big( 1 + | X_{0} | ^{ {p} \beta_1} 
\big),
\quad   
\text{for all \,} p \in [ 1, \mathds{B}_1],\quad n = 1,2 \cdots N,
\end{eqnarray}
where $\beta_1
 = 1+ \frac{(p\gamma_1+1)\mathds{G}_1}{p} \wedge \frac{(1+\frac12 p+ p\gamma_2)\mathds{G}_1}{p}$ and $\mathds{B}_1 = \frac{ p_0'-\varepsilon -\mathds{G}_1}{1+ \gamma_1 \mathds{G}_1} \wedge \frac{p_0'-\varepsilon-\mathds{G}_1}{1+ (\frac12 + \gamma_2) \mathds{G}_1}$  
with $\gamma_1$, $\gamma_2$ from \ref{ass:tame-h1} in Assumption \ref{ass:tame_function_assumption}.
Here $\mathds{G}_1 = 6r \vee \frac{(2r+1)l_1-1}{\tau}$ with $ r \geq 0 $ from \ref{ass:a1} and $ \tau, l_1 > 0 $ from \ref{ass:tame-h2}.

If \ref{ass:tame-h2} is replaced by \ref{ass:tame-h2p}, then 
$\mathds{G}_1 = 6r$ and 
$\mathds{B}_1 = \frac{ p_{\mathcal{T}} -\mathds{G}_1}{1+ \gamma_1 \mathds{G}_1} \wedge \frac{p_{\mathcal{T}}-\mathds{G}_1}{1+ (\frac12 + \gamma_2) \mathds{G}_1}$.
\end{lem}
This lemma can be proved by extending the results in \cite{Zhang2017preserving,chen2019mean}.
See Appendix \ref{append-modified-moment-bound} for details. 

Note that $\beta_1$ plays the role of $\beta$ in Theorem \ref{thm:fundamental_weak_convergence} \ref{ass:iii}.
 Applying Theorem \ref{thm:fundamental_weak_convergence}, with the moment bounds for the modified Euler scheme \eqref{eq:Euler-scheme} in Lemma \ref{lem:Euler-scheme-moment-bound},
 we obtain the following theorem.  

\begin{thm}[Weak convergence order for the modified Euler scheme \eqref{eq:Euler-scheme}]
 \label{thm:MEM-global-conver-order}
Let Assumption \ref{assm:test-functions} on the test function $\varphi$ hold. 
    Let Assumptions \ref{ass:coefficient_function_assumption} and  \ref{assu:a3'-mono-condi} be satisfied. 
    Additionally, assume that 
\ref{ass:tame-h1}, \ref{ass:tame-h2} (or \ref{ass:tame-h2p}) 
 in Assumption \ref{ass:tame_function_assumption} hold. 
Assume further that there exist $ C > 0 $, $q_0>0$ and 
$\eta_{q_0} \geq 0$ such that
\begin{eqnarray}
\label{eq:ass-euler-difference}
    |z- \mathcal{T}_{i}(z,h)| & \leq & C (1+|z|^{\eta_{q_0}}) h^{q_0}, \quad   z \in \R^d,\quad i = 1,2.  
\end{eqnarray}
 If $\mathds{P}$ in Lemma \ref{lem:deri-mom-bound} is no less than $4q+4$ and  $\mathds{B}_1$ in Lemma \ref{lem:Euler-scheme-moment-bound} 
 is no less than $2 \kappa  \vee (\beta_1 \kappa + \varkappa)$,
 the scheme \eqref{eq:Euler-scheme} has a global weak convergence of order $ q \wedge 1 \leq q_0 $:
\begin{equation*}
	\big |\E \big [
	 \varphi \big (
	  X(t_0,X_0;T)
	   \big ) \big ]
	  -\E \big[
	 \varphi \big (
	  Y( t_0,Y_0;T )
	  \big )\big ] \big |
	 \leq 
	   C   ( 1 + | X_0 |^{\beta_1(\beta_1 \kappa+ \varkappa)} )
	    h^{q\wedge 1},
 \end{equation*}
where 
 $\beta_1$ is from \eqref{eq:balanced-euler-result-moment-bound},
 $\varkappa=(2q+2)(2r+1)(1 \vee \eta_{q_0})$ and   $r$ comes from \eqref{ass:drift-f}.
\end{thm}
The proof is similar to that in \cite{wang2021weak} and we omit the proof.
Here we use $\varkappa$ from \cite{wang2021weak}, which can be verified as in Lemmas \ref{lem:norm-esti-one-step-new-MT} and \ref{lem:balanced-one-step-error}. 
The parameters in Table \ref{table:para-essential-euler} are essential for weak convergence order of numerical schemes \eqref{eq:Euler-scheme}.
\begin{table}[htb]
\centering
\scalebox{0.8}{
\begin{tabular}{c|c|c|c|c|c} \hline 
 $\mathds{P}$ &  $\gamma_1$, $\gamma_2$  & $\mathds{G}_1$, 
    $\mathds{B}_1$,  
    $\beta_1$ 
   & $\eta_{q_0}$, $q_0$
   & $\kappa$
   & $\varkappa$
    \\ \hline 
derivatives & %
approximation & moment bounds & approximation & test function &minimum moments \\\hline
Lemma \ref{lem:deri-mom-bound} &
Assumption \ref{ass:tame_function_assumption} &
Lemma \ref{lem:Euler-scheme-moment-bound}  
& Theorem \ref{thm:MEM-global-conver-order} & Assumption \ref{assm:test-functions} & Theorem \ref{thm:MEM-global-conver-order} \\\hline 
\end{tabular}}
\caption{Essential parameters for weak convergence orders of the modified Euler scheme \eqref{eq:Euler-scheme}}
\label{table:para-essential-euler}
\end{table}
   \begin{rem}
   We may extend the condition \eqref{eq:ass-euler-difference} to general modified schemes. For example,
   for parameterized maps $\mathcal{T}_{i}(z,h;w)$, one may derive similar results as in Theorem \ref{thm:MEM-global-conver-order}, if $\abs{z-\mathcal{T}_i(z,h;w)} \leq C(1+|z|^{\eta_{q_0}}+|w|^{\eta_{q_0}})h^{q_0}$.
\end{rem}

Below we list some examples of modified Euler scheme
 \eqref{eq:Euler-scheme}, where $\mathcal T_{1}$ and $\mathcal T_{2}$ are explicitly given. We also summarize  the key parameters to determine the weak convergence orders of these schemes in
 Table \ref{table:para-cal-mom-bound}.
\begin{itemize}
\item 
fully tamed Euler (c.f. \cite{sabanis2016euler}), for any $ z,\,w \in \R^d$ and $0< \alpha_1 \leq 1$,
\begin{equation}
\label{exam:eq-fully-tamed-Euler}
  \mathcal T_{1}(z,h;w)
  =  \frac{z}{1+h^{\alpha_1}\abs{z} +h^{\alpha_1} \abs{w}},  \quad 
\mathcal T_{2}(w,h;z) =\frac{w}{1+h^{\alpha_1}\abs{z} +h^{\alpha_1} \abs{w}}.
\end{equation}
\item  balanced scheme \cite{Zhang2017preserving}
\begin{equation}
\label{eq:balanced-scheme-first-order}
    \mathcal T_{i}(z,h) = h^{-1} \tanh(h z),\quad  
    i=1,2,  \quad  z \in \R^d.
 \end{equation}
\item tamed Euler scheme \cite{chen2019mean}
\begin{equation}
\label{eq:tamed-scheme-first-order}
   \mathcal T_{i}(z,h) = \frac{z}
                   {1+h|z|}, \quad i=1,2, \quad  z \in \R^d.
 \end{equation}
\item modified Euler scheme  \cite{wang2021weak}
   \begin{equation}  
\label{eq:first-order-scheme-MES}
\mathcal T_{1}(z,h) = \frac{z}
                   {1+h|z|^2}, \quad
\mathcal T_{2}(w,h;z) = \frac{w}
                   {1+h|z|^2}, \quad z, w \in \R^d.
    \end{equation}
\item  truncation  scheme
\begin{equation}
\label{eq:scheme-truncation}
    \mathcal T_i(z,h) = z\mathds{1}_{\abs{z}\leq h^{-\alpha}} + \vartheta {\rm sgn}(z) h^{-\alpha}\mathds{1}_{\abs{z}> h^{-\alpha}}, \quad \vartheta =0 
    \text{ or } 1, \quad i=1,2.
\end{equation}
\end{itemize}

\begin{table}[!tb]
\centering
\scalebox{0.8}{
\begin{tabular}{c|c|c|c|c|c|c|c} \hline 
   scheme  & $\gamma_1$  & $\gamma_2$  & $\mathds{G}_1$    
   & $\mathds{B}_1$  
   & $\varkappa$ 
   & $\eta_{q_0}$
   & $q_0$
    \\ \hline 
 fully tamed  \eqref{exam:eq-fully-tamed-Euler}    &  $\alpha_1$     & $\alpha_1$ &     $6r$  
 & $  \frac{p_{\mathcal{T}}-\mathds{G}_1}{1+(\frac{1+\alpha_1}{2})\mathds{G}_1} $    
 & $(2q+2)(2r+1)\eta_{q_0}$
& $1+\varsigma$  &  $\varsigma\alpha_1$ \\  \hline 
 balanced 
\eqref{eq:balanced-scheme-first-order}    &    $1$   & $1$      & 
$6r \vee \frac{(3-2\varsigma)(2r+1)-1}{2-2\varsigma}$  
&   $ \frac{p_0'-\varepsilon-\mathds{G}_1}{1+\frac{3}{2} \mathds{G}_1}$  
& $ (2q+2)(2r+1)\eta_{q_0}$ & $3-2\varsigma$ 
& $2-2\varsigma$\\ \hline 
tamed Euler  
\eqref{eq:tamed-scheme-first-order}  &  $1$ & $1$      & $ 6r \vee \frac{(2r+1)(1+\varsigma)-1}{\varsigma}$ 
& $  \frac{p_0'-\varepsilon-\mathds{G}_1}{1+\frac{3}{2} \mathds{G}_1}$    
& $ (2q+2)(2r+1) \eta_{q_0} $ & $1+\varsigma$ & $\varsigma$ \\ \hline 
truncation \eqref{eq:scheme-truncation} & $\alpha$ &
$\alpha$ & $6r \vee \frac{(2r+1)(1+\frac{\epsilon}{\alpha})-1}{\epsilon}$   
& $  \frac{p_0'-\varepsilon-\mathds{G}_1}{1+(\frac12+\alpha) \mathds{G}_1}$ 
&  $ (2q+2)(2r+1)\eta_{q_0}$ & $1+\frac{\epsilon}{\alpha}$
& $\epsilon$
\\ \hline
\end{tabular}}
\caption{Essential parameters for weak convergence orders of the scheme \eqref{eq:Euler-scheme}. Here $\varsigma\in [0,1]$ and $\epsilon>0$.}
\label{table:para-cal-mom-bound}
\end{table}
The fully tamed Euler \eqref{exam:eq-fully-tamed-Euler} satisfies condition \ref{ass:tame-h1} from Assumption \ref{ass:tame_function_assumption} with $\gamma_1=\alpha_1$ and $\gamma_2 =\frac{\alpha_1}{2}$. Due to \eqref{eq:co_mono_condi}, $\mathcal{T}_1(f(x),h;g(x))$ and $\mathcal{T}_2(g(x),h;f(x))$ satisfy \ref{ass:tame-h2p}, by Lemma \ref{lem:Euler-scheme-moment-bound}, $\mathds{G}_1= 6r$. Furthermore, due to 
$ \abs{z-\mathcal{T}_i(z,h;w)}=h^{\alpha_1}\abs{z}\frac{\abs{z}+ \abs{w}}{1+h^{\alpha_1}\abs{z}+h^{\alpha_1}\abs{w}}\leq h^{\varsigma \alpha_1} \abs{z}(\abs{z}^{\varsigma}+\abs{w}^{\varsigma}),$ condition \eqref{eq:ass-euler-difference} holds with $q_0 =\varsigma \alpha_1$ and $\eta_{q_0}=1+\varsigma$.

For the balanced scheme \eqref{eq:balanced-scheme-first-order}, we need to verify the condition \ref{ass:tame-h2} from Assumption \ref{ass:tame_function_assumption}.
In fact, by \eqref{eq:tanh-function-property}, 
we have  
\begin{align}
 \abs{h^{-1}\tanh(hz) -z}\leq 
h^{-1}\abs{\tanh(hz) -hz}\leq  h^{-1}\abs{hz}^{3-2\varsigma} =  h^{2-2\varsigma} \abs{z}^{3-2\varsigma},  \quad 0\leq \varsigma\leq 1,   
\end{align}
i.e.,
$\tau=2-2\varsigma$ and $l_1=3-2\varsigma$. 
Then,
$\mathds{G}_1= 6r \vee \frac{(3-2\varsigma)(2r+1)-1}{2-2\varsigma}$, $\eta_{q_0} = 3-2\varsigma$ and $q_0=2-2\varsigma$.

For the tamed Euler \eqref{eq:tamed-scheme-first-order}, condition \ref{ass:tame-h1} is satisfied with $\gamma_1 = \gamma_2=1$. Furthermore,
we have 
$\abs{z-\mathcal{T}_1(z,h)}= h\abs{z}\frac{\abs{z}}{1+h^2\abs{z}}\leq h^{\varsigma} \abs{z}^{1+\varsigma},
\, 0 \leq \varsigma \leq 1.$
Then, condition \ref{ass:tame-h2} from Assumption \ref{ass:tame_function_assumption} hold with $\tau = \varsigma$ and $l_1 = 1+\varsigma$. So, we get $\mathds{G}_1= 6r \vee \frac{(2r+1)(1+\varsigma)-1}{\varsigma}$, $\eta_{q_0}=1+\varsigma$ and $q_0 = \varsigma$.

In the modified Euler scheme \eqref{eq:first-order-scheme-MES}, 
$\mathcal{T}_2(g(x),h;f(x))$ 
does not satisfy \ref{ass:tame-h1} in Assumption \ref{ass:tame_function_assumption}. We refer interested readers to  \cite{wang2021weak} for the proof of bounded moments when $p_0$ is arbitrarily large.

 It's straightforward to check that for the truncation scheme \eqref{eq:scheme-truncation}, $\abs{\mathcal T_i(z,h)} \leq h^{-\alpha} \wedge \abs{z}$, and thus \ref{ass:tame-h1} in Assumption \ref{ass:tame_function_assumption} hold with $\alpha= \gamma_1 = \gamma_2 $. For any $\epsilon > 0$, we also have
\begin{equation}
\label{eq:verify-h2-trun-scheme}
\abs{\mathcal{T}_1 (z,h) - z}
 =  \abs{\mathcal {T}_1(z,h) -z}\mathds{1}_{\abs{z}> h^{-\alpha}} \leq  \abs{z-\vartheta {\rm sgn}(z) h^{-\alpha}}\mathds{1}_{\abs{z}> h^{-\alpha}}  \leq (1 + \vartheta) \abs{z}^{1+\frac{\epsilon}{\alpha}}h^{\epsilon}. 
\end{equation}
Subsequently, \ref{ass:tame-h2} in Assumption \ref{ass:tame_function_assumption} is satisfied with $\tau = \epsilon$ and $l_1 =1+\frac{\epsilon}{\alpha}$. One can show $\mathds{G}_1= 6r \vee \frac{(2r+1)(1+\frac{\epsilon}{\alpha})-1}{\epsilon}$, $\eta_{q_0} = 1+\frac{\epsilon}{\alpha}$ and $q_0 = \epsilon$.
\begin{rem}The moment bounds in Table \ref{table:para-cal-mom-bound} are consistent with those in the literature. The balanced scheme \eqref{eq:balanced-scheme-first-order} and the tamed Euler scheme \eqref{eq:tamed-scheme-first-order}, as discussed in \cite{Zhang2017preserving} and \cite{chen2019mean},
exhibit the same moment bound for $1 \leq p \leq \mathds{B}_1$. 
\end{rem}
\subsection{Weak convergence order for the scheme 
\eqref{eq:second-order-new-scheme}}\label{sec:moment-bound-second-order-scheme}
Next, we consider the weak convergence order of the scheme 
\eqref{eq:second-order-new-scheme}. According to Theorem \ref{thm:fundamental_weak_convergence}, we need to show the moment bounds of numerical solutions.
\begin{lem}[\bf Moment bounds of numerical solutions \eqref{eq:second-order-new-scheme}]
\label{lem:new-scheme-moment-bound} 
Let Assumptions
\ref{ass:coefficient_function_assumption},
\ref{assu:a3'-mono-condi}, and \ref{ass:tame_function_assumption}  hold. 
Let $p_0'$ in \eqref{eq:co_mono_condi} be sufficiently large. 
Then there exist
$ \beta_2 \geq 1 $ and $ C > 0 $ independent of $ h $ such that 
\begin{eqnarray} 
  \label{eq:second-order-new-scheme-moment-bound-result}
   \E  \big[ 
   \left|Y_{n}\right|^{p}
   \big] 
  & \leq & 
   C 
\big( 1 + | X_{0} | ^{ \beta_2 p} 
 \big),
 \quad   
 \text{for all \,} p \in [ 1, \mathds{B}_2],\quad  n=1,2, \ldots, N,
\end{eqnarray}
where { 
$\beta_2 = 1+  (\beta_1-1) \wedge \frac{(1-\frac12 p+ p\gamma_3) \mathds{G}_1}{p} \wedge \frac{(1- p+ p\gamma_4) \mathds{G}_1}{p}$}, 
and $ \mathds{B}_2 = 
\mathds{B}_1\wedge \frac{ p_0'-\varepsilon -\mathds{G}_1}{1+(\gamma_3-\frac12)\mathds{G}_1} \wedge \frac{ p_0'-\varepsilon-\mathds{G}_1}{1+(\gamma_4 -1 )\mathds{G}_1}$. 
Also, $\gamma_3$, $\gamma_4$ are from Assumption \ref{ass:tame_function_assumption}. Here $\mathds{G}_1$, $\beta_1$ and $\mathds{B}_1$  come from Theorem \ref{thm:MEM-global-conver-order}.

If \ref{ass:tame-h2} in Assumption \ref{ass:tame_function_assumption} is replaced by \ref{ass:tame-h2p}, then 
$\mathds{G}_1=6r$, and $\mathds{B}_2=\mathds{B}_1 \wedge \frac{ p_{\mathcal{T}} -\mathds{G}_1}{1+(\gamma_3-\frac12)\mathds{G}_1} \wedge \frac{ p_{\mathcal{T}}-\mathds{G}_1}{1+(\gamma_4 -1 )\mathds{G}_1}$.
  \end{lem}

Note that $\beta_2$ plays the role of $\beta$ in Theorem \ref{thm:fundamental_weak_convergence} \ref{ass:iii}. The proof for the above lemma is presented in Section \ref{subsec:moment-bound-second-order}.

By Theorem \ref{thm:fundamental_weak_convergence}, we need the one-step approximation of the scheme 
\eqref{eq:second-order-new-scheme} to be of certain order.
To this end, we assume the following on the maps $\mathcal{T}_i's$.
\begin{assumption}
\label{ass:diff-T1-T6-new-scheme-MT}
For $ \mathcal{T}_i \colon  \R^d \times (0,1) \rightarrow \R^d $, there exist $ C > 0 $, $q_0>0$ and $\eta_{q_0} \geq 0$ such that for all $  z \in \R^d $,
\begin{eqnarray}
    |z- \mathcal{T}_{i}(z,h)| & \leq & C (1+|z|^{\eta_{q_0}}) h^{q_0}, \quad i = 1,2,3, \nonumber  \\
     |z- \mathcal{T}_{j}(z,h)| & \leq & C (1+|z|^{\eta_{q_0}}) h^{q_0-1}, \quad j = 4,5,6.
\end{eqnarray}
\end{assumption}
Under the above assumptions, by Theorem \ref{thm:fundamental_weak_convergence}, our main result is summarized as follows.

\begin{thm}[Weak convergence order for the scheme \eqref{eq:second-order-new-scheme}]
 \label{thm:MMT-global-conver-order}
Let Assumption \ref{assm:test-functions} hold. 
    Let Assumptions \ref{ass:coefficient_function_assumption} and  \ref{assu:a3'-mono-condi} be satisfied. 
    Additionally, assume that 
\ref{ass:tame-h1}, \ref{ass:tame-h2} (or \ref{ass:tame-h2p}), \ref{ass:tame-h3}
 in Assumption \ref{ass:tame_function_assumption}, and Assumption \ref{ass:diff-T1-T6-new-scheme-MT} hold. If $\mathds{P}$ in Lemma \ref{lem:deri-mom-bound} is no less than $4q+4$ and $\mathds{B}_2 \geq 2\kappa \vee (\beta_2\kappa +\varkappa)$, 
 the scheme \eqref{eq:second-order-new-scheme} has a global weak convergence order $q \wedge 2\leq q_0$:
\begin{equation*}
	\big |\E \big [
	 \varphi \big (
	  X(t_0,X_0;T)
	   \big ) \big ]
	  -\E \big[
	 \varphi \big (
	  Y( t_0,Y_0;T )
	  \big )\big ] \big |
	 \leq 
	   C ( 1 + | X_0 |^{\beta_2(\beta_2 \kappa+\varkappa)} )
	  h^{q\wedge 2},
 \end{equation*}
where  {
 $\beta_2$ is from \eqref{eq:second-order-new-scheme-moment-bound-result}, 
 $ \varkappa  = 
    (2q+2)r_1  \vee r_2 
    $ and $r_1=(4r+1)\eta_{q_0},\,r_2= (2q+1)(6r+1)$ come from Lemmas \ref{lem:norm-esti-one-step-new-MT} and \ref{lem:balanced-one-step-error}. Also, 
  $r$ is from \eqref{ass:drift-f}.}\end{thm}

We summarize the key parameters in Table \ref{table:para-essential-truncation} for the weak convergence order of the scheme \eqref{eq:second-order-new-scheme}.
\begin{table}[!ht]
\centering
\scalebox{0.7}{
\begin{tabular}{c|c|c|c|c|c} \hline 
$\mathds{P}$ & $\gamma_i$'s 
  & $\mathds{G}_1$, $\mathds{B}_2$,  
   $\beta_2$ 
   & $\eta_{q_0}$, $q_0$
   & $\kappa$
   & $\varkappa$
    \\ \hline 
derivatives &    %
approximation & moment bounds & approximation &  test function &minimum moments \\\hline
Lemma \ref{lem:deri-mom-bound} & 
Assumption \ref{ass:tame_function_assumption} &
Lemma \ref{lem:new-scheme-moment-bound}  &  
Assumption 
\ref{ass:diff-T1-T6-new-scheme-MT} & Assumption \ref{assm:test-functions} & Theorem \ref{thm:MMT-global-conver-order} \\ \hline
\end{tabular}}
\caption{Essential parameters for the weak convergence order of the scheme \eqref{eq:second-order-new-scheme}.}
\label{table:para-essential-truncation}
\end{table}
\begin{rem}
When the maps are parameterized, e.g., $\mathcal{T}_{i}(z,h;w)$, we may derive similar results as in Theorem \ref{thm:MMT-global-conver-order} if we assume $\abs{z-\mathcal{T}_i(z,h;w)} \leq C(1+|z|^{\eta_{q_0}}+|w|^{\eta_{q_0}})h^{q_0}$. 
\end{rem}
\subsubsection{Examples of the  maps $\mathcal{T}_i$'s}
\label{subsec:ex-sec-order-sche}
We next present three examples with specified $\mathcal T_i$'s and the key parameters for weak convergence orders. 
\begin{exm}[Truncation scheme]
\label{exm:truncation}
In the scheme \eqref{eq:second-order-new-scheme}, we take 
\begin{equation}
\label{eq:scheme-proposed-truncation}
    \mathcal T_i(z,h) = z\mathds{1}_{\abs{z}\leq h^{-\alpha}} + \vartheta {\rm sgn}(z) h^{-\alpha}\mathds{1}_{\abs{z}> h^{-\alpha}}, \quad \vartheta =0 
    \text{ or } 1, \quad i=1,2,3,4,5,6.
\end{equation}
\end{exm}
 It's straightforward to check that $\abs{\mathcal T_i(z,h)} \leq h^{-\alpha} \wedge \abs{z}$, and thus \ref{ass:tame-h1} and
 \ref{ass:tame-h3} in Assumption \ref{ass:tame_function_assumption} hold with $\alpha= \gamma_1 = \gamma_2 = \gamma_3= \gamma_4 $. As shown in \eqref{eq:verify-h2-trun-scheme}, 
for any $\epsilon > 0$, 
\ref{ass:tame-h2} in Assumption \ref{ass:tame_function_assumption} is satisfied with $\tau = \epsilon$ and $l_1 =1+\frac{\epsilon}{\alpha}$. 
Additionally, Assumption \ref{ass:diff-T1-T6-new-scheme-MT} is fulfilled with $\eta_{q_0}=1+\frac{\epsilon}{\alpha}$ and $q_0=\epsilon$, by 
\eqref{eq:verify-h2-trun-scheme}. 
Thus, by Lemmas \ref{lem:new-scheme-moment-bound} and \ref{lem:balanced-one-step-error}, we can show that $\mathds{G}_1 = 6r \vee \frac{(2r+1)(1+\frac{\epsilon}{\alpha})-1}{\epsilon}$, $\beta_2
 = 1+ \frac{(1-p+ p \alpha) \mathds{G}_1}{p}$, $p \in [ 1, \mathds{B}_2]$, $\mathds{B}_2 = \frac{p_0'- \varepsilon-\mathds{G}_1}{1+(\frac12+\alpha)\mathds{G}_1}$, and $\varkappa = r_2 \vee 
(2q+2)r_1, 
   r_1=(4r+1) \epsilon, r_2= (2q+1)(6r+1)$.
Let $\epsilon=\alpha = 2$. If $\mathds{P} \geq 4q+4 $ and $\mathds{B}_2 \geq 2\kappa \vee (\beta_2\kappa +\varkappa)$,
the weak convergence order is $q \leq q_0 = \alpha =2$, by Theorem \ref{thm:MMT-global-conver-order}. 
The second-order weak convergence is confirmed numerically in Examples \ref{ex:numerical_ex_linear_diff},
where enough high-order moments of numerical solutions exists.   

 \begin{rem}[Structure preserving schemes]\label{rem:structure-preserving}
 Taking $\vartheta=0$ in 
\eqref{eq:scheme-proposed-truncation} and applying structure-preserving schemes for SDEs with Lipschitz continuous coefficients \footnote{Examples are some explicit schemes for systems with separable Hamiltonians.}  may lead to 
structure-preserving schemes for SDEs with non-Lipschitz coefficients. In this case,  we may also stop the trajectories where truncation is needed and thus we have an acceptance-rejection strategy similar to that in 
\cite{Milstein2005numerical}. 
Therein, the authors discard the approximate trajectories that leave a sufficiently large ball   $ S_R:=\{x:|x|<R\} $. This strategy has been tested on quasi-symplectic schemes for nonlinear stochastic Hamiltonian systems. 
A similar truncation strategy has been considered in \cite{Liu2013strong}, which may be structure-preserving while the scheme requires small time sizes for stability. We will not discuss this research direction for brevity as verifying 
that the resulting schemes admit enough high-order moments for convergence order is necessary. 
\end{rem}
\begin{exm}
\label{exm:balanced}
In the scheme \eqref{eq:second-order-new-scheme}, we propose the balanced scheme as follows{\footnote{For convenience all $\mathcal T_{i}$ are the same. Other choices of these maps are possible, e.g., $\mathcal T_{2}(z,h) =h^{-\frac{3}{2}} \tanh(h^{-\frac{3}{2}} z)$.}} 
: for all $z \in \R^d$, 
\begin{equation}
\label{eq:scheme-proposed-balanced}
    \mathcal T_{i}(z,h) = h^{-2} \tanh(h^{2} z),\quad i=1,2,3,4,5,6.
 \end{equation}
It's clear that \ref{ass:tame-h1} and \ref{ass:tame-h3} from Assumption \ref{ass:tame_function_assumption} are fulfilled when $\gamma_1 = \gamma_2 =\gamma_3 = \gamma_4=2 $.
Also, based on \eqref{eq:tanh-function-property}, it holds for any $\varsigma\in [0,1]$ that
\begin{align}
\abs{\mathcal T_{i}(z,h)-z} \leq h^{-2} (h^2|z|)^{3-2\varsigma}=h^{4-4\varsigma}\abs{z}^{3-2\varsigma}, \quad i=1,2,3,4,5,6.
\end{align}
This implies that \ref{ass:tame-h2} in Assumption \ref{ass:tame_function_assumption} is valid with 
$\tau =4-4\varsigma$ and $l_1=3-2\varsigma$, and Assumption \ref{ass:diff-T1-T6-new-scheme-MT} holds with $q_0=4-4 \varsigma$ and $\eta_{q_0} = 3-2 \varsigma$.
By Lemmas \ref{lem:new-scheme-moment-bound} and \ref{lem:balanced-one-step-error}, $\mathds{G}_1 = 6r \vee \frac{(2r+1)(3-2\varsigma)-1}{4-4 \varsigma}$, $\mathds{B}_2 = \frac{p_0'- \varepsilon-\mathds{G}_1}{1+\frac52\mathds{G}_1}$,  $\beta_2
 = 1+ \frac{(p  +1)\mathds{G}_1}{p}$, $p \in [ 1, \mathds{B}_2]$, and $\varkappa = r_2 \vee 
(2q+2)r_1, r_1=(4r+1)\eta_{q_0}, r_2= (2q+1)(6r+1)$.
By Theorem \ref{thm:MMT-global-conver-order}, the weak convergence order is $q\wedge 2 \leq q_0=2$, if $\mathds{P} \geq 4q+4 $ and $\mathds{B}_2 \geq 2\kappa \vee (\beta_2\kappa +\varkappa)$.
We also present numerical results (Example \ref{ex:numerical_ex_linear_diff}) 
and verify the second-order weak convergence.
\end{exm}

\begin{rem}
The choices of the maps can be made more general as long as the maps satisfy Assumption \ref{ass:diff-T1-T6-new-scheme-MT}. 
We use the hyperbolic tangent function since it has the following properties: 
\begin{align}
\label{eq:tanh-function-property}
\abs{\tanh(y)} &\leq 1, \, \nonumber\abs{\tanh(y)} \leq \abs{y}, \nonumber  \\
\abs{y - \tanh(y)} &\leq \tanh^2(\theta y) \abs{y}, \, \, \hbox{for some} \, \, 0 \leq \theta \leq 1, \nonumber \\
 \abs{y-\tanh y} &\leq 
 \abs{y}^{3-2 \varsigma}, \text{ for any } 
0\leq \varsigma \leq 1.
\end{align}
\end{rem}

Under some certain conditions, Assumption \ref{ass:diff-T1-T6-new-scheme-MT} can be extended when the mappings $\mathcal{T}_{i}$ depend on some $w \in \R^d$, written as $\mathcal{T}_{i} = \mathcal{T}_{i} (z,h;w)$.

\begin{exm}\label{exm:modified}
In the scheme \eqref{eq:second-order-new-scheme}, we apply the following modification maps, where 
\begin{equation}
    \label{eq:scheme-proposed-modified}
    \mathcal{T}_i=
    \mathcal{T}_i (z,h;w) =\frac{z}{1+h^2\abs{z}+h^2\abs{w}}, \quad i=1,2; 
    \quad 
    \mathcal{T}_j =\mathcal{T}_j(z,h) =\frac{z}{1+h^2\abs{z}},\quad j=3,4,5,6.
\end{equation}
Evidently, \ref{ass:tame-h1} and \ref{ass:tame-h3} in Assumption \ref{ass:tame_function_assumption} are fulfilled with $ \gamma_1 = \gamma_2 = \gamma_3 = \gamma_4 = 2. $ Moreover, for any $\varsigma \in [0,1]$, we have 
    \[\abs{z-\mathcal{T}_i(z,h;w)}=
 h^2\abs{z}\frac{\abs{z}+ \abs{w}}{1+h^2\abs{z}+h^2\abs{w}}
 \leq h^{2\varsigma} \abs{z}(\abs{z}^{\varsigma}+\abs{w}^{\varsigma}), \quad i=1,2, \] and 
     \[\abs{z-\mathcal{T}_j(z,h)}=
    h^2\abs{z}\frac{\abs{z}}{1+h^2\abs{z}}\leq h^{2\varsigma} \abs{z}^{1+\varsigma}, \quad j=3,4,5,6.\]  
i.e. \ref{ass:tame-h2} in Assumption \ref{ass:tame_function_assumption} and Assumption \ref{ass:diff-T1-T6-new-scheme-MT} are satisfied with 
$\tau =2 \varsigma$, $l_1=1+\varsigma$, $\eta_{q_0}=1+\varsigma$, and $q_0=2\varsigma$.

It can  be readily verified that \ref{ass:tame-h2p} is satisfied with $p_{\mathcal{T}}=p_0'-\varepsilon$ if we take
\begin{equation}
\label{ex:scheme-satisfy-one-sided-con}
\mathcal{T}_1(f(x),h) = \mathcal{T}_1(f(x),h;g(x)),\quad \mathcal{T}_2(g(x),h)=\mathcal{T}_2(g(x),h;f(x)).
\end{equation}
Subsequently, applying Lemmas \ref{lem:new-scheme-moment-bound} and \ref{lem:balanced-one-step-error} yields that 
$\mathds{G}_1 = 6r$, 
$\mathds{B}_2=\frac{p_{\mathcal{T}}-\mathds{G}_1}{1+\frac52\mathds{G}_1}$, 
$\beta_2
 = 1+ \frac{(p  +1)\mathds{G}_1}{p}$, $p \in [ 1, 
 \mathds{B}_2]$, and $\varkappa = r_2 \vee 
(2q+2)r_1, r_1=(4r+1)(1+\varsigma), r_2= (2q+1)(6r+1)$.
The weak convergence order becomes $q \leq q_0 =2$ when $\mathds{P} \geq 4q+4$ and $\mathds{B}_2 \geq 2\kappa \vee (\beta_2\kappa +\varkappa)$, according to Theorem \ref{thm:MMT-global-conver-order}. We numerically verify this scheme \eqref{ex:scheme-satisfy-one-sided-con} in Example \ref{ex:numerical_ex_linear_diff},
confirming the second-order weak convergence if enough moments of numerical solutions exist.
\end{exm} 
\subsection{Examples of weak convergence order depending on moment bounds}\label{sec:exm-order-calculuation}
\begin{exm}
\label{ex:ex_linear_diff}
Consider the following SDE with Lipschitz diffusion coefficient
\begin{equation}
\label{eq:linear_diffusion_model}
			\dd X(t)  = \big(X(t)- X^3(t) \big) \, \dd t + X(t)  \, \dd W(t),
		\quad t \in (0, T],
		\quad
		X(0)  = x.
\end{equation}
\end{exm}
Assumption \ref{ass:coefficient_function_assumption} is fulfilled with $r=1$, $\rho=1$ while $p_0<+\infty$ and $c_{p_0}=\frac{p_0+1}{2}$ in \eqref{eq:assum-f-g}.  
Also, the inequality \ref{assu:a3'-mono-condi} holds with $p_0'=p_0<+\infty$. Based on Lemma \ref{lem:deri-mom-bound}, we can obtain $\mathds{P} < +\infty$ to ensure estimates of exact solutions and derivatives in $x$.
{As $p_0'$ can be arbitrarily large, $ \mathds{B_1}$ and $ \mathds{B}_2 $ are arbitrarily large, and thus both are larger than $ 2\kappa \vee (\beta_2 \kappa+\varkappa)$. Recall that 
$\varkappa=(2q+2)(4r+1)\eta_{q_0}\vee (2q+1)(6r+1)$ for the scheme 
\eqref{eq:second-order-new-scheme} and 
$\varkappa 
=(2q+2)(2r+1)(1\vee \eta_{q_0})$ for the scheme \eqref{eq:Euler-scheme}}.
By Theorems \ref{thm:MEM-global-conver-order} and \ref{thm:MMT-global-conver-order}, the weak convergence orders are $q_0\wedge 2$ and $q_0 \wedge 1$.

Applying Theorem \ref{thm:MMT-global-conver-order}, we obtain the second-order weak convergence of the truncation scheme \eqref{eq:scheme-proposed-truncation} with $\alpha=2$, the balanced scheme \eqref{eq:scheme-proposed-balanced}, and the modified scheme \eqref{ex:scheme-satisfy-one-sided-con} with $\varphi(x) = \cos x$, $x^2$. 
\begin{table}[htb]
\centering
\scalebox{0.8}{
\begin{tabular}{c|c|c|c|c|c|c|c} \hline 
scheme \eqref{eq:second-order-new-scheme}  
& $\mathds{G}_1$   
& $\mathds{B}_2$  
   & $\varkappa$
   & $2\kappa \vee (\beta \kappa+\varkappa) $
   & $\eta_{q_0}$
   & $q_0$
   & $q$ (order) \\ \hline 
 truncation \eqref{eq:scheme-proposed-truncation} ($\alpha=2$) 
 &  $6$  
 & $  \frac{p_0'-\epsilon-6}{16}$  & $60$  
 &   $60$ 
 & $2$
 & $=\epsilon=2$  
 & $2$ \\  \hline 
 balanced 
  \eqref{eq:scheme-proposed-balanced}  
  &     $6$  
&$  \frac{p_0'-\epsilon-6}{16} $
& $ 60 $  
& $60$   & $2$ & $2$ 
& $2$ \\ \hline 
modified   \eqref{eq:scheme-proposed-modified}   
&$6$  
& $   \frac{p_0'-\epsilon-6}{16}$ 
& $60$     & $60$  & $2$
& $2$  
& $2$ \\ \hline \hline 
scheme \eqref{eq:Euler-scheme} 
   & $\mathds{G}_1$    
   & $\mathds{B}_1$  
   & $\varkappa$  
   & $2\kappa \vee (\beta \kappa+\varkappa)$
   & $\eta_{q_0}$
   & $q_0$ 
   & $q$ (order) \\ \hline 
 fully tamed  \eqref{exam:eq-fully-tamed-Euler} ($\alpha_1=1$)
 & $6$  
 & $ \frac{p_0'-\varepsilon-6}{7} $   
 & $18$ 
  &  $18$  
  & $2$
  & $\alpha_1$
  & $=\alpha_1=1$
  \\  \hline 
 balanced 
\eqref{eq:balanced-scheme-first-order}  
& $6$  
& $ \frac{p_0'-\varepsilon-6}{10}$  
& $24$  &  $24$   & $2$ & $1$ 
& $1$ \\ \hline 
tamed Euler  
\eqref{eq:tamed-scheme-first-order}       
& $6$ 
& $   \frac{p_0'-\varepsilon-6}{10}$    
& $ 24$ &  $24$ & $2$ & $1$ 
& $1$ \\ \hline 
truncation \eqref{eq:scheme-truncation} 
& $6$   
& $   \frac{p_0'-\varepsilon-6}{10}$ 
&  $24$ & $24$ & $2$&
$=\alpha=1$ 
& $1$
\\ \hline
\end{tabular}}
\caption{Essential parameters for weak convergence order of the schemes in Example \ref{ex:ex_linear_diff},  $\varphi(x)=\cos x$. Here $\mathds{P}<+\infty$.}
\label{table:exm-para-cal-first-weak-conv}
\end{table}
By Theorem \ref{thm:MEM-global-conver-order}, we may show first-order weak convergence of the fully tamed Euler \eqref{exam:eq-fully-tamed-Euler} with $\alpha_1=1$, the balanced scheme \eqref{eq:balanced-scheme-first-order}, the tamed Euler scheme \eqref{eq:tamed-scheme-first-order}, the modified Euler scheme \eqref{eq:first-order-scheme-MES}, and the truncation scheme \eqref{eq:scheme-proposed-truncation} with $\alpha = 1 $. 
In Table \ref{table:exm-para-cal-first-weak-conv}, we present the essential parameters in case of $\varphi(x)=\cos x$, where $\kappa = 0$.
The weak convergence orders are numerically verified in Figure \ref{fig:linear_diff_cosx_x_2} in Section \ref{sec:numer-results}.

The next example shows that the weak convergence order deteriorates if the numerical solutions don't have enough high-order moments. 
\begin{exm}
\label{exm:nonlin-diff-equ}
Consider the following stochastic differential equations 
\begin{equation}
\label{eq:sde-critical-exm}
		\dd X(t)  = \big(X(t)-X^3(t) \big) \, \dd t + \sigma X^2(t)  \, \dd W(t),
		\quad t \in (0, T],
		\quad
		X(0)  = x_0.
\end{equation}
\end{exm}
Assumption \ref{ass:coefficient_function_assumption} is fulfilled with  $ r= 1$ and $\rho=2$
while $p_0 =  3/(2\sigma^2) +1$ and $c_{p_0}=1$ 
in \eqref{eq:assum-f-g}. 
Also, we may check that  for all 
$p_0' \leq \frac{2}{\sigma^2} +1$, it holds that 
$    x (x-x^3) + \frac{p_0'-1}{2}\sigma^2 x^4 \leq x^2$.
Take $\varphi(x)=\cos x$ and two sets of model parameters:
\textbf{Case I}:  
$\sigma = 0.1$,   
\textbf{Case II}:  
$\sigma =0.5$. In the calculations, we may observe that $q$ is not an integer while the order $q$ in some previous theorems is required to be an integer. For simplicity, we keep $q$ as is for the discussion while $q$ in assumptions and lemmas should be replaced by 
the rounding of $q$ to its nearest integer. 

 When $\sigma=0.1$, $p_0'\leq 201$, we would like to mention that $\mathds{P}$ in Lemma \ref{lem:deri-mom-bound} is large enough to ensure estimates of exact solutions and derivatives in $x$. For the balanced scheme \eqref{eq:scheme-proposed-balanced}, $\mathds{B}_2=  \frac{p_0'-\epsilon-6}{16} \approx 12.2$, $q_0=4-4\varsigma$, $\eta_{q_0}=3-2\varsigma$, $\varsigma \in [0,1]$. According to Theorem \ref{thm:MMT-global-conver-order}, we need 
$\mathds{B}_2\geq 2\kappa \vee (\beta_2 \kappa+\varkappa)=\varkappa$, i.e. 
$12.2  \geq \varkappa
   = (2q+2)(4r+1)\eta_{q_0}  \vee (2q+1)(6r+1) = 10(q+1)\eta_{q_0}\vee 7(2q+1)
 $. 
Since $q\leq q_0 = 4-4\varsigma$, we take 
$\varsigma = 0.964970$ and thus 
$\eta_{q_0}\approx 1$ and $q \leq 0.14.$
For the scheme 
\eqref{eq:balanced-scheme-first-order}, $\mathds{B}_1 =  \frac{p_0'-\epsilon-6}{10} \approx 19.5$, $q_0=2-2\varsigma$, $\eta_{q_0}=3-2\varsigma$, $\varsigma \in [0,1]$. By Theorem \ref{thm:MEM-global-conver-order}, we need 
$\mathds{B}_1 \geq 2\kappa \vee (\beta_1 \kappa+\varkappa)=\varkappa$, i.e.
$19.5\geq \varkappa=(2q+2)(2r+1)(1 \vee \eta_{q_0}) =
   6(q+1) (1 \vee(3-2\varsigma))$.
As $q\leq q_0 = 2-2\varsigma$, we may take 
$\varsigma \approx 0.965$ and thus 
$q \leq q_0 = 2-2\varsigma
\approx 0.803 $.
For the rest of the schemes, calculations are similar and we present results in Table \ref{tbl:parameter-example-nonlinear-diff-modified-critical}.

Observe that 
when $\sigma=0.5$,  $p_0'\leq 9$ while Theorem \ref{thm:MMT-global-conver-order} requires 
$\varkappa=(2q+2)(4r+1)\eta_{q_0}\vee 
(2q+1)(6r+1)$-th moments of 
numerical solutions. In this case, we cannot draw any conclusion on the convergence order. 
For the scheme \eqref{eq:scheme-proposed-truncation} with $\alpha =2$, the schemes \eqref{eq:scheme-proposed-balanced} and \eqref{ex:scheme-satisfy-one-sided-con}, 
we present numerical results in Example \ref{ex:numerical_non_dri_diff_critical_case}, where we observe a weak convergence order around one. 
%
\begin{table}[!htb]
\centering
\scalebox{0.7}{
\begin{tabular}{c|c|c|c|c|c} \hline 
 scheme \eqref{eq:second-order-new-scheme}  
  &  $\mathds{G}_1$ 
   & $\mathds{B}_2$ 
   &  {$\varkappa
   = 10(q+1)\eta_{q_0}  \vee 7(2q+1)$     }
   & $q_0$  & $q$(order)  \\ \hline  {Truncation \eqref{eq:scheme-proposed-truncation}($\alpha =2$)}   
   & $6$ 
&  $  \frac{p_0'-\varepsilon-6}{16} \approx 12.2 $ 
 & {$10(q+1)(1+\frac{\epsilon}{\alpha}) \vee 7(2q+1)$}  
 & $\epsilon$
 &  {$\approx 0.14$}
 \\  \hline 
Balanced 
  \eqref{eq:scheme-proposed-balanced}  
  &$6$  
& $  \frac{p_0'-\varepsilon-6}{16} \approx 12.2$ 
& $10(q+1)(3-2\varsigma) \vee 7(2q+1)$ & $4-4\varsigma$ & $\approx 0.14$  \\ \hline 
 Modified   \eqref{eq:scheme-proposed-modified}  & $6$  
& $  \frac{p_0'-\varepsilon-6}{16} \approx 12.2 $  & $10(q+1)(1+\varsigma) \vee 7(2q+1)$
   & $2 \varsigma$  
   & $\approx 0.14$
 \\ \hline \hline
 scheme \eqref{eq:Euler-scheme} 
     & $\mathds{G}_1$  & $\mathds{B}_1$ & $\varkappa=(2q+2)(2r+1) \eta_{q_0}$    & $q_0$ & $q$ (order)
     \\ \hline 
 fully tamed  \eqref{exam:eq-fully-tamed-Euler} ($\alpha_1=1$)
    &  $6$   & $ \frac{p_0'-\varepsilon-6}{7} \approx 27.86 $   
  &  $3(2q+2)(1+\varsigma)$ 
  & $=\varsigma=1$
  & $=\varsigma=1$
  \\  \hline 
 balanced 
\eqref{eq:balanced-scheme-first-order}  
& $6$
&  $ \frac{p_0'-\varepsilon-6}{10} \approx 19.5$ 
& $3(2q+2)(3-2\varsigma)$    & $2-2\varsigma$ 
&  {$\approx 0.803$}
 \\ \hline 
tamed Euler  
\eqref{eq:tamed-scheme-first-order}       
&  $6$
&  $ \frac{p_0'-\varepsilon-6}{10} \approx 19.5$
& $3(2q+2)(1+\varsigma)$   & $\varsigma$ &  {$\approx 0.803$} \\ \hline 
truncation \eqref{eq:scheme-truncation} ($\alpha=1$)
&  $6$
& $ \frac{p_0'-\varepsilon-6}{10} \approx 19.5$
& $3(2q+2)(1+\frac{\epsilon}{\alpha})$ 
 & $\epsilon$ & {$ \approx 0.803$}
\\ \hline
\end{tabular}}
\caption{\textbf{Case I}: $\sigma=0.1$. Essential parameters for checking the weak convergence orders of schemes for \eqref{eq:sde-critical-exm} with $\varphi(x)=\cos{x}$.
Here   $2\kappa \vee (\beta \kappa+\varkappa)=\varkappa$,
$\epsilon>0$, 
and $\varsigma\in [0,1]$.}
\label{tbl:parameter-example-nonlinear-diff-modified-critical}
\end{table}
\section{Numerical Experiments}
\label{sec:numer-results}
In this section, we will test the scheme \eqref{eq:second-order-new-scheme} for three examples with
different maps $\mathcal{T}_i$'s: the truncation scheme \eqref{eq:scheme-proposed-truncation} with $\alpha =2$ (denoted by \textbf{TS2}); the balanced scheme \eqref{eq:scheme-proposed-balanced} (\textbf{BS2}); the scheme \ref{eq:second-order-new-scheme} with 
  \eqref{ex:scheme-satisfy-one-sided-con} and $\mathcal{T}_3-\mathcal{T}_6$ from  \eqref{eq:scheme-proposed-modified} (\textbf{MS2}). 
In Example \ref{ex:numerical_non_dri_diff_critical_case}, we also compare these schemes with those of at most first-order weak convergence from Section \ref{sec:discussion-modified-sch}:  the balanced scheme \eqref{eq:balanced-scheme-first-order} (\textbf{BS1}), the modified Euler scheme \eqref{eq:first-order-scheme-MES} (\textbf{MS1}) and the truncation scheme \eqref{eq:scheme-truncation} with $\alpha = 1 $ (\textbf{TS1}).

We present numerical results for schemes in three examples: 
 the first has one-sided Lipschitz drift and Lipschitz diffusion; the second has locally Lipschitz drift and diffusion, where the solution has limited moments; the third is two-dimensional. 
 
 In the experiments, we compare their weak error and computational costs. We consider different test functions $\varphi(x) = x, x^2,\cos x$ in Assumption \ref{assm:test-functions} and we use a large number of  Monte Carlo samples with  a small enough time step size $ h_{ref} $ to compute references: 
 $    \E \big[ \varphi(X_T)\big]
     \thickapprox
    \frac{1}{M} \sum_{i=1}^{M}
    \varphi\big( X_T(\omega_i,h_{ref} )\big).
 $
 Here $M$ is sufficiently large to have negligible statistical errors with the $95\%$ confidence interval.  
 In simulations, we use the Mersenne twister algorithm \cite{matsumoto1998mersenne} to generate pseudorandom numbers.
  The experiments are performed using Matlab R2022a on a desktop (31.7 GB RAM, 12th Gen Intel(R) Core(TM) i9-12900H CPU at 2.50 GHz) with 64-bit Windows 11 operating system.  

\begin{exm}
\label{ex:numerical_ex_linear_diff}
Consider the following SDE in Example \ref{ex:ex_linear_diff}
\begin{equation}
 			\dd X(t)  = \big(X(t)- X^3(t) \big) \, \dd t + X(t)  \, \dd W(t),
		\quad t \in (0, T],
		\quad
		X(0)  = x.
\end{equation}
\end{exm}

In this example, we fix $ T=2 $, the initial condition $ x = 0.5$, and take $h_{ref}=10^{-3}$ and 
$ M = 2 \times 10^6 $. 
In Figure \ref{fig:linear_diff_cosx_x_2}, we present  weak errors of three different numerical schemes (\textbf{TS2}, \textbf{MS2} and \textbf{BS2}) with the test functions $\varphi(x)= x^2 \hbox{,} \cos x$ against different step sizes $ 0.002,0.003,0.004$
on a log-log scale. As predicted in Example \ref{ex:ex_linear_diff}, the weak convergence orders of the proposed schemes are close to $2$.
\begin{figure}[!ht]
\includegraphics[width=0.42\linewidth, height=0.22\textheight]{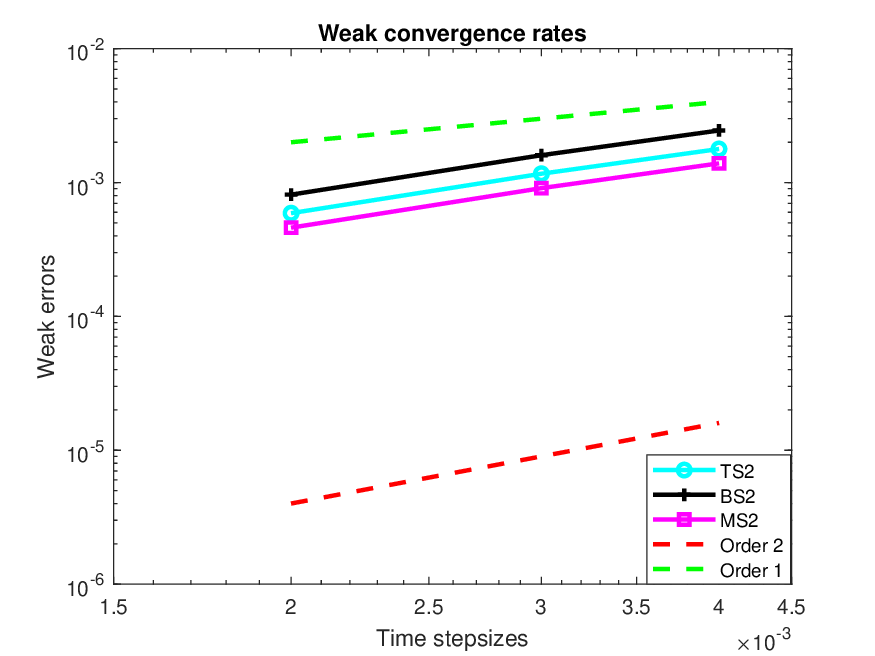} \includegraphics[width=0.42\linewidth, height=0.22\textheight]{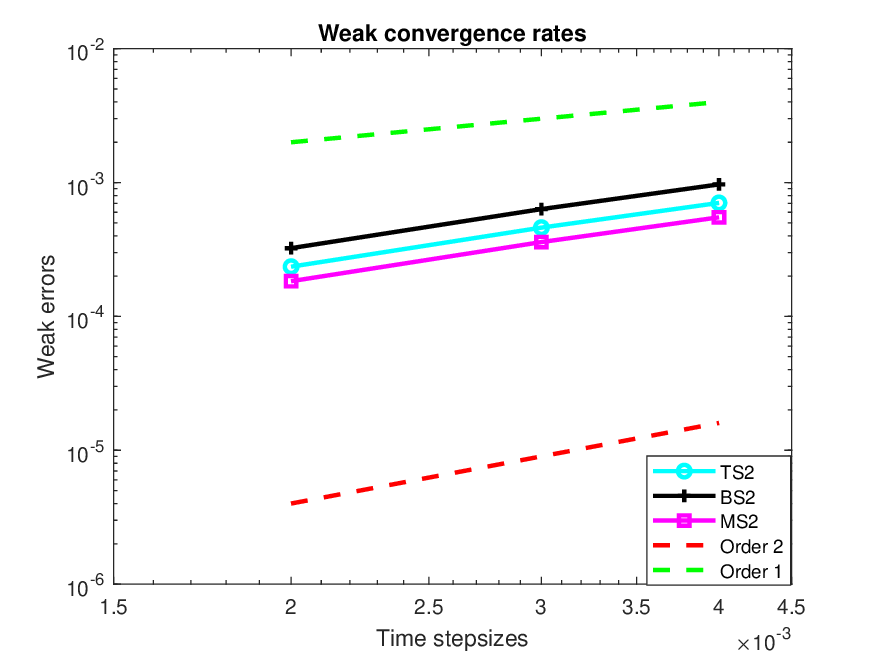}
	\caption
	    {Example \ref{ex:numerical_ex_linear_diff}: weak convergence orders 
with $ \varphi(x) = x^2 $ (Left) and $ \varphi(x) =\cos(x)$ (Right).
	     }
\label{fig:linear_diff_cosx_x_2}	
\end{figure}
\begin{exm}
\label{ex:numerical_non_dri_diff_critical_case}
Consider the equation in Example \ref{exm:nonlin-diff-equ}:
\begin{equation*}
		\dd X(t)  = \big(X(t)-   X^3(t) \big) \, \dd t + \sigma X^2(t)  \, \dd W(t),
		\quad t \in (0, T],\quad 		X(0)  = x_0=0.1. 
\end{equation*}
\end{exm}
Consider  \textbf{Case I} ($\sigma=0.1$) and \textbf{Case II} ($\sigma=0.5$) in Example \ref{exm:nonlin-diff-equ}.
We first investigate the weak convergence order when $\sigma=0.1$. 
For the reference solution, we take $ h_{ref} = 2^{-13}$ and $ M = 3 \times 10^6 $.  
We present in
Table \ref{table:non-dri-diff-one-dim} weak errors of the selected schemes for  different step sizes 
$ h = 2^{-2},2^{-3},2^{-4},2^{-5},2^{-6},2^{-7}$ where we use the test function  $\varphi(x) = \cos x$.
The number of trajectories $M = 3 \times 10^6 $ yields statistical errors that are $10$ times smaller than the reported weak errors. We omit the statistical errors for brevity.  
We underline errors at the order of $ 2 \times 10^{-4} \sim  4 \times 10^{-4}$   
in Table \ref{table:non-dri-diff-one-dim} and their corresponding computational time at the bottom of the table. 
We observe that the performance for the underlined accuracy is comparable among all schemes in the table.
The empirical convergence orders in Table \ref{table:non-dri-diff-one-dim} are higher than the theoretical predictions.
According to Theorems \ref{thm:MMT-global-conver-order} and 
\ref{thm:MEM-global-conver-order}, the weak convergence order with $\varphi(x) = \cos x$ should be less than $1$  for the 
TS2, BS2, MS2, and approximately $1$ for TS1 BS1, MS1 schemes. Here we conjecture that the step sizes are not small enough to apply Theorems 
\ref{thm:MEM-global-conver-order}
and \ref{thm:MMT-global-conver-order}.
To verify our conjecture, we further take $h_{ref} = 10^{-6}$ and $M =10^8$ and present results in Table \ref{table:non-dri-diff-one-dim-small-step-sizes} weak errors of the schemes with $\varphi(x)=\cos x$ for smaller step-size $h=5 \times 10^{-4}$, $10^{-3}$, $2 \times 10^{-3}$, $4 \times 10^{-3}$. We observe that the convergence orders of scheme \eqref{eq:second-order-new-scheme} (\textbf{TS2}, \textbf{BS2}, \textbf{MS2}) decrease to almost one as time step-sizes become smaller and smaller. We could expect that the convergence order should be less than one if we take even smaller time step sizes.
We also observe that the schemes of 
\textbf{TS2}, \textbf{BS2} and \textbf{MS2}
have better accuracy than schemes of 
\textbf{TS1}, \textbf{BS1} and \textbf{MS1}.
\begin{table}[!ht]
\centering
\scalebox{0.9}{
\begin{tabular}
 { c|cc|cc| 
 cc| cc|
 cc| cc}
		\hline
h 
&\textbf{TS2}  &\hspace{-0.4cm} rate & \textbf{BS2} & \hspace{-0.4cm} rate 
&  \textbf{MS2} & \hspace{-0.4cm} rate &  \textbf{TS1} &\hspace{-0.4cm}  rate 
&  \textbf{BS1} & \hspace{-0.4cm} rate & \textbf{MS1}  & \hspace{-0.4cm} rate \\
\hline
$ 2^{-7} $ & 8.11e-7   &\hspace{-0.4cm} -- & 8.11e-7&\hspace{-0.4cm}  --
&1.46e-6 &\hspace{-0.4cm}  -- & \underline{2.21e-4} & \hspace{-0.4cm} --
& \underline{2.22e-4}   &\hspace{-0.4cm}  --  & \underline{3.00e-4} & \hspace{-0.4cm} --   \\
\hline
$ 2^{-6} $ & 2.45e-6 &\hspace{-0.4cm} 1.59  &2.45e-6  & \hspace{-0.4cm} 1.59 
& 5.02e-6 &\hspace{-0.4cm}  1.78  & 4.43e-4    & \hspace{-0.4cm} 1.00
& 4.44e-4  &\hspace{-0.4cm} 1.00  & 5.99e-4    & \hspace{-0.4cm}  1.00  \\
\hline
$ 2^{-5} $ & 8.21e-6   &\hspace{-0.4cm} 1.74 & 8.21e-6  & \hspace{-0.4cm} 1.75 
& 1.83e-5 & \hspace{-0.4cm} 1.87 & 8.78e-4  & \hspace{-0.4cm} 0.99
& 8.79e-4  & \hspace{-0.4cm} 0.99 & 1.18e-3  & \hspace{-0.4cm} 0.98   \\
\hline
$ 2^{-4} $ & 2.93e-5    & \hspace{-0.4cm} 1.83  &2.93e-5 & \hspace{-0.4cm} 1.83
& 6.84e-5 &\hspace{-0.4cm}  1.90  & 1.71e-3  & \hspace{-0.4cm} 0.96
& 1.71e-3 &\hspace{-0.4cm} 0.96   & 2.27e-3  & \hspace{-0.4cm} 0.94  \\
\hline
$ 2^{-3} $ 
& 1.08e-4
& \hspace{-0.4cm} $1.88$  
& 1.08e-4
&\hspace{-0.4cm} 1.88
& \underline{2.54e-4}
& \hspace{-0.4cm} 1.89 
&3.25e-3   & \hspace{-0.4cm} 0.92
&3.26e-3   & \hspace{-0.4cm} 0.93
& 4.20e-3  & \hspace{-0.4cm} 0.89  \\
\hline
$ 2^{-2} $ & \underline{3.96e-4}   
& \hspace{-0.4cm} 1.88  
& \underline{3.97e-4}  & \hspace{-0.5cm} 1.88
&9.04e-4  & \hspace{-0.4cm} 1.83 
&5.89e-3   & \hspace{-0.4cm} 0.86
& 5.91e-3  &\hspace{-0.4cm}  0.86
& 7.33e-3 & \hspace{-0.4cm} 0.80 \\
 \hline
 \hline
& 7.36e+02   &      & 7.53e+02 &
    & 6.52e+02   &      & 7.48e+02 &
    &7.52e+02 &           & 7.16e+02 &  \\ \hline
 \end{tabular}}
\caption{ \textbf{Case I} in Example 
\ref{ex:numerical_non_dri_diff_critical_case}: weak errors with the test function  $\varphi(x) = \cos x$ of the selected schemes at 
$T=1$.  CPU time (in seconds) is listed at the bottom for schemes with underlined errors.}
\label{table:non-dri-diff-one-dim}
\end{table}
\begin{table}[!ht]
\centering 
\scalebox{1}{
\begin{tabular}
 {c|cc| cc| 
 cc| cc|
 cc| cc}
\hline
\centering{h} 
&\textbf{TS2}  &\hspace{-0.5cm} rate & \textbf{BS2} & \hspace{-0.5cm} rate 
&  \textbf{MS2} & \hspace{-0.5cm} rate &  \textbf{TS1} &\hspace{-0.5cm}  rate 
&  \textbf{BS1} & \hspace{-0.5cm} rate & \textbf{MS1}  & \hspace{-0.5cm} rate 
\\
\hline
$5 \times 10^{-4} $ & 2.66e-8
  &\hspace{-0.5cm} -- & 2.66e-8 &\hspace{-0.5cm}  --
& 2.93e-8 & \hspace{-0.5cm}  -- & 1.45e-5
 & \hspace{-0.5cm} --
& 1.45e-5
   &\hspace{-0.5cm}  --  
   & 1.97e-5
   & \hspace{-0.5cm} --   \\
\hline
$1 \times 10^{-3} $ 
& 5.91e-8
 &\hspace{-0.5cm} 1.15 
& 5.91e-8  & \hspace{-0.5cm} 1.15 
& 6.98e-8
 &\hspace{-0.5cm}   1.25  
 & 2.90e-5
  & \hspace{-0.5cm} 1.00 
& 2.90e-5
  &\hspace{-0.5cm}  1.00 
& 3.94e-5
    & \hspace{-0.5cm} 1.00 
   \\
\hline
$ 2 \times 10^{-3} $ & 1.32e-7
   &\hspace{-0.5cm} 1.16 
  &1.32e-7
 & \hspace{-0.5cm} 1.16 
& 1.74e-7
 & \hspace{-0.5cm} 1.32 
 & 5.79e-5
 & \hspace{-0.5cm} 1.00
&5.79e-5
  & \hspace{-0.5cm} 1.00 & 7.87e-5
  & \hspace{-0.5cm} 1.00   \\
\hline
$ 4 \times 10^{-3} $ 
& 3.13e-7
  & \hspace{-0.5cm} 1.25 
&3.13e-7
  & \hspace{-0.5cm} 1.25 
& 4.83e-7
 &\hspace{-0.5cm}  1.47  & 1.16e-4
  & \hspace{-0.5cm} 1.00
& 1.16e-4
&\hspace{-0.5cm}  1.00 
& 1.57e-4
  & \hspace{-0.5cm} 1.00   \\
\hline
 \end{tabular}
 }
\caption{ \textbf{Case I} in Example 
\ref{ex:numerical_non_dri_diff_critical_case}: weak errors with $\varphi(x) = \cos x$ for smaller step sizes of the selected schemes at 
$T=1$.}
\label{table:non-dri-diff-one-dim-small-step-sizes}
\end{table}

For \textbf{Case II}, where $\sigma=0.5$, we do not expect the numerical schemes (\textbf{TS2}, \textbf{MS2} and \textbf{BS2}) to have weak convergence order of 1 or 2 (see discussions in Example
\ref{exm:nonlin-diff-equ}).  However, we observe a first-order weak convergence in 
Figure \ref{fig:nonlinear_dri_diff_cosx_x_2}, where we present the weak errors of three different numerical schemes (\textbf{TS2}, \textbf{MS2} and \textbf{BS2}) with the test functions $\varphi(x)= x^2 \hbox{,} \cos x$.
In this case, the empirical observation does not match the theoretical prediction.
We conjecture that the mismatch may disappear when much smaller step sizes are used. 
However, we cannot perform such 
experiments due to limited computational resources. 
\begin{figure}[!ht]
\includegraphics[width=0.42\linewidth, height=0.22\textheight]{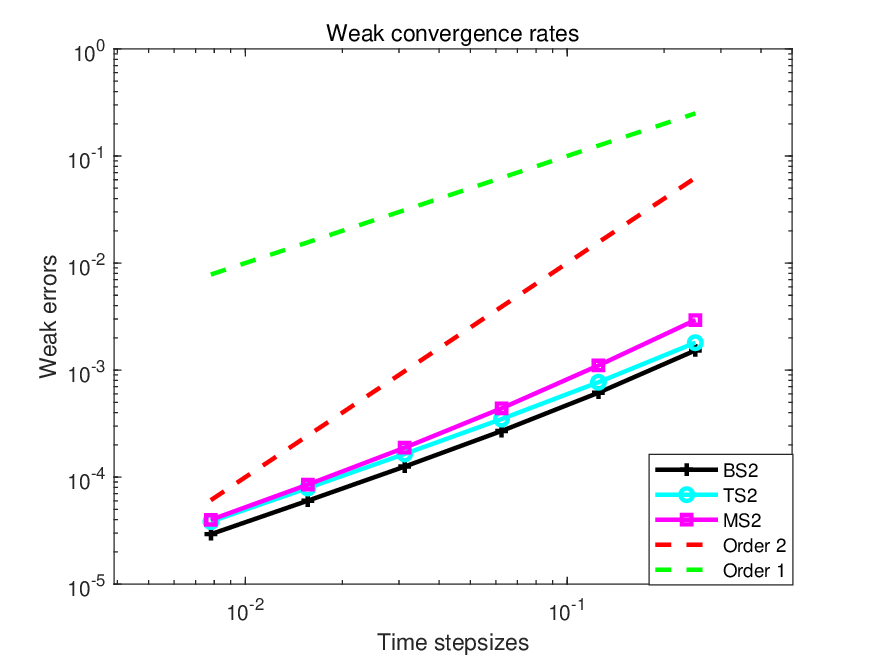} \includegraphics[width=0.42\linewidth, height=0.22\textheight]{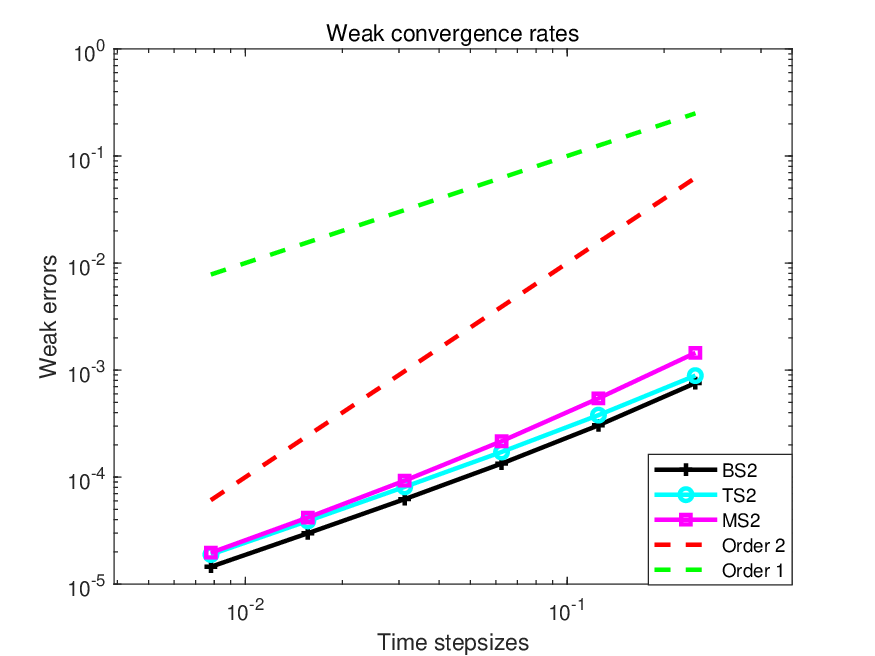}
	\caption
	    {\textbf{Case II} in  Example \ref{ex:numerical_non_dri_diff_critical_case}: weak convergence orders with $ \varphi(x) = x^2 $ (Left) 
	    	and $ \varphi(x) =\cos(x)$ (Right).
	     }
\label{fig:nonlinear_dri_diff_cosx_x_2}	
\end{figure}
\begin{exm}\label{exm:sde-twod}
Consider the stochastic
FitzHugh-Nagumo (FHN) model \cite{buckwar2022splitting} in the form of
\begin{equation}
\label{eq:stochastic_FHN_model}
	\begin{split}
		&  \left(\begin{array}{c}
			\dd X_{1}(t) \\ \dd X_{2}(t) \end{array}\right)
		  = \left(\! \begin{array}{c}
		 	X_{1}(t)-X_{1}^3(t) - X_{2}(t)
		 	\\  X_{1}(t) - X_{2}(t) + 1 
	 	\end{array} \! \right) \dd t
 	+  \left(\! \begin{array}{cc}
 	      X_1(t) + 1 & 
 		0 \\ 0 & 
 		 X_2(t) + 1
 		\end{array} \! \right)
 	   \dd  W(t),
	\end{split}
\end{equation}
for $ t \in (0,T]$ and $ X(0)  = (0.8,0.8)^\top $, with solution $ X(t) := \big( X_1(t) , X_2(t) \big)^\top $
for $ t \in [0,T] $,
where $ W(t) : = \big( W_1(t), W_2(t) \big)^\top $
is a two-dimensional Brownian motion. 
\end{exm}
\begin{figure}[!ht] 
\centering\includegraphics[width=0.42\linewidth, height=0.22\textheight]{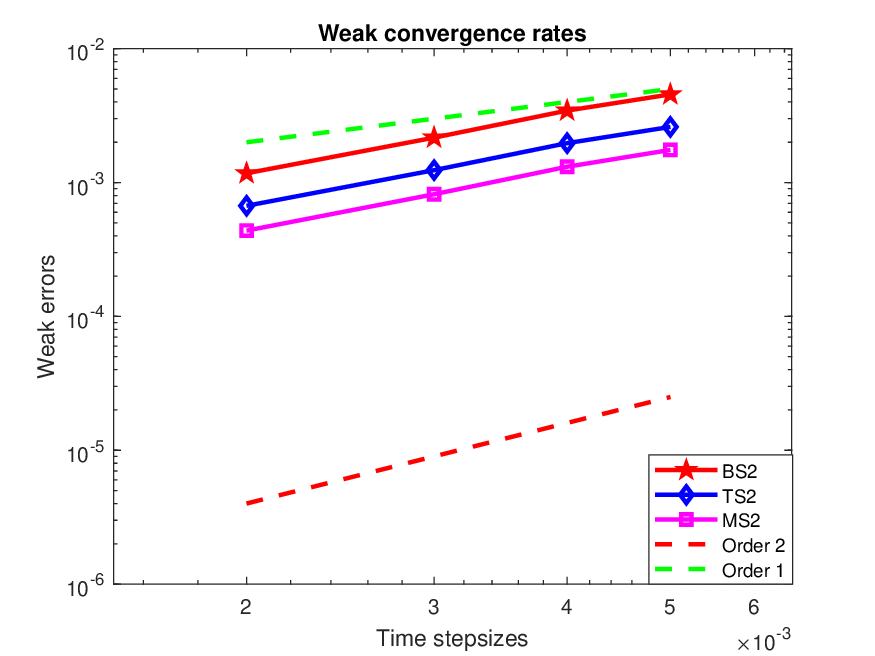}
\includegraphics[width=0.42\linewidth, height=0.22\textheight]{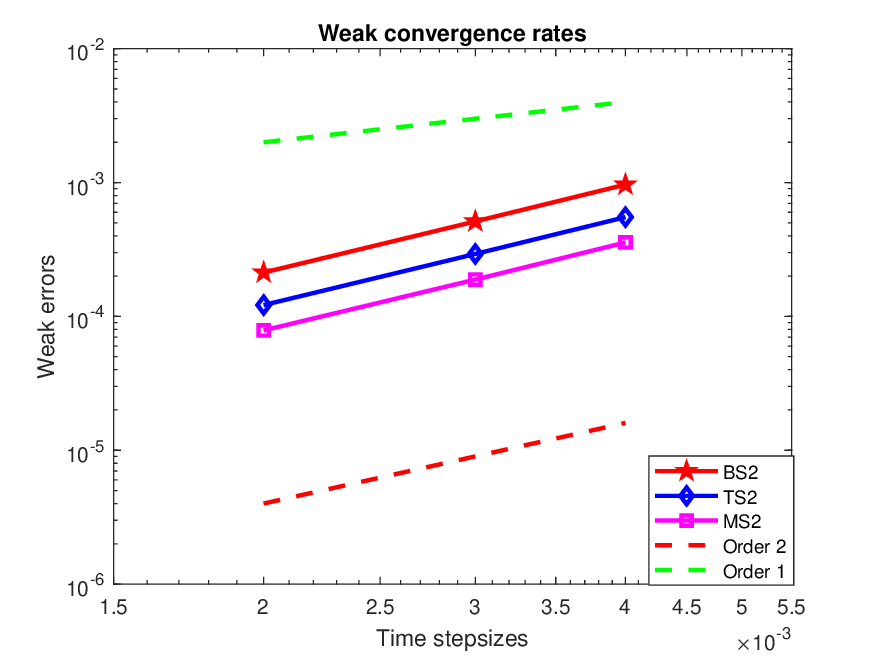}
	\caption
	    {Example \ref{exm:sde-twod}: weak convergence orders of the scheme
    \eqref{eq:second-order-new-scheme} with  $ \varphi(x) = x  $ (Left) 
	    	and $ \varphi(x) = \cos(x) $ (Right).
	     }
\label{fig:two_dimension_x_exp(x)}	
\end{figure}
We fix $ T=2 $ and take $h_{ref}=10^{-3}$
and $ M = 2 \times 10^6 $ 
to generate a reference solution. 

In this example,  Assumption \ref{ass:coefficient_function_assumption} is satisfied with   $ r= 1$ and $\rho = 1$ and $p_0<+\infty$. Also, $p_0'=p_0<\infty$ in the assumption \ref{assu:a3'-mono-condi},
according to the discussions in Section \ref{subsec:ex-sec-order-sche},
all the schemes of \textbf{TS2}, \textbf{MS2} and \textbf{BS2} have weak convergence of order two. 
In Figure \ref{fig:two_dimension_x_exp(x)}, we present the weak errors of these three numerical schemes with the test functions ($\varphi(x)= x \hbox{,} \cos x $) and step sizes $h=  0.002,0.003,0.004$. We observe that the weak convergence orders of these schemes are close to $2$ in agreement with the theoretical prediction.
\section{Weak convergence of the   scheme \eqref{eq:second-order-new-scheme}}
\label{sec:weak-con-order-proof}

In this section, we  prove the second-order weak convergence for the proposed scheme 
\eqref{eq:second-order-new-scheme} under the assumptions in Sections \ref{sec:assumption-preliminary} and \ref{sec:second-order-scheme-prensented}.

\subsection{Moment bounds}
\label{subsec:moment-bound-second-order}
We first prove the moment bounds of the scheme \eqref{eq:second-order-new-scheme}.

\begin{proof}[Proof of Lemma 
\ref{lem:new-scheme-moment-bound}]
To prove the boundedness of moments, we adopt a proof strategy similar to that of Lemma \ref{lem:Euler-scheme-moment-bound}. However, we provide only the necessary details to streamline our presentation. The key to proving the boundedness of moments is to estimate the growth of the solution under some events
\begin{equation*}
\Omega_{\mathcal{R}, n}
  := 
  \left\{
  \omega \in \Omega:
 \sup_{0 \leq i \leq n}
  \left|Y_{i}(\omega)\right| \leq \mathcal{R} \right\}, 
  \quad  n=0,1, \ldots, N, \, \quad
   N \in \mathbb{N}.
\end{equation*}
For integer $ \bar{p} \geq 1 $, We have 
%
\begin{eqnarray}
 \label{eq:decom-new-scheme-sub}
	 \E \big[ \mathds{1}_{\Omega_{ \mathcal{R}, n+1}} |Y_{n+1}|^{\bar{p}}  \big] 
  & \leq &
   \E \big[ \mathds{1}_{\Omega_{ \mathcal{R}, n}} |Y_{n}|^{\bar{p}}  \big]
  +
   \E \big[ \mathds{1}_{\Omega_{ \mathcal{R}, n}} |Y_{n}|^{\bar{p}-2} 
    \big(\bar{p} \langle Y_n, Y_{n+1}-Y_n \rangle
      + \frac{\bar{p}(\bar{p}-1)}{2}| Y_{n+1}-Y_n |^2
    \big) \big]  \nonumber \\
 && + C \sum_{l=3}^{\bar{p}}
   \E \big[ \mathds{1}_{\Omega_{ \mathcal{R}, n}} 
   |Y_{n}|^{\bar{p}-l} 
   |Y_{n+1} - Y_n |^{l}  \big] \nonumber \\
 & := & 
 \E \big[ \mathds{1}_{\Omega_{ \mathcal{R}, n}} |Y_{n}|^{\bar{p}}  \big] 
 + I_1 + I_2.
\end{eqnarray}
Compared to the proof of bounded moments for the scheme \eqref{eq:Euler-scheme}, it is essential to provide a proper upper bound for $ I_1 $. Similar to the proof of the upper bound for $ I_1 $ in Lemma \ref{lem:Euler-scheme-moment-bound}, we have
\begin{eqnarray}
 \label{eq:decom-I1-new-scheme}
    I_1 
   & \leq &
  \bar{p}  \E \big[ \mathds{1}_{\Omega_{ \mathcal{R}, n}} |Y_{n}|^{\bar{p}-2} 
      \langle Y_n, \mathcal{T}_1(f(Y_n),h) h \rangle
      \big] 
  +  \bar{p} \E \big[ \mathds{1}_{\Omega_{ \mathcal{R}, n}} 
   |Y_{n}|^{\bar{p}-2} 
   \langle Y_n, \int_{t_n}^{t_{n+1}} 
   \mathcal{T}_6 \big( \mathcal{L} f(Y_n), h \big) \frac{h}{2} \, \dd s \rangle \big]   \nonumber \\
  && +
  \frac{\bar{p} (\bar{p}-1)}{2} \E \big[ \mathds{1}_{\Omega_{ \mathcal{R}, n}} 
   |Y_{n}|^{\bar{p}-2} 
   |\mathcal{T}_1(f(Y_n),h) h|^{2}  \big]
  + \frac{\bar{p} (\bar{p}-1)}{2} 
  \E \big[ \mathds{1}_{\Omega_{ \mathcal{R}, n}} 
   |Y_{n}|^{\bar{p}-2} 
   \big|\sum_{r=1}^{m} \mg{Y_n} \Delta W_r(n)
   \big|^{2}  \big]  \nonumber \\
 && + 
   \frac{\bar{p} (\bar{p}-1)}{2} 
   \E \Big[ \mathds{1}_{\Omega_{ \mathcal{R}, n}} 
   |Y_{n}|^{\bar{p}-2} \big|
    \sum_{r=1}^{m} \sum_{r_1=1}^{m} \int_{t_n}^{t_{n+1}} \int_{t_n}^{s} 
   \mathcal{T}_3 \big(\Lambda_{r_1} g^r( Y_n), h \big) \, \dd W_{r_1}(s_1) \dd W_r(s) \big|^2 \Big] \nonumber \\
 && +
   \bar{p} (\bar{p}-1)
   \E \Big[ \mathds{1}_{\Omega_{ \mathcal{R}, n}} 
   |Y_{n}|^{\bar{p}-2} \big|
   \sum_{r=1}^{m} \int_{t_n}^{t_{n+1}}
   \mathcal{T}_4 \big( \mathcal{L} g^r(Y_n), h \big) h \, \dd W_r(s) 
   \big|^2 \Big] \nonumber \\
  &&  + 
   \bar{p} (\bar{p}-1)
   \E \Big[ \mathds{1}_{\Omega_{ \mathcal{R}, n}} 
   |Y_{n}|^{\bar{p}-2} \big|
   \sum_{r=1}^{m} 
   \int_{t_n}^{t_{n+1}} \int_{t_n}^{s} 
\mathcal{T}_5 \big( ( \Lambda_{r} f ( Y_n) - \mathcal{L} g^r(Y_n) ), h \big)  \, \dd W_r(s_1) \dd s \big|^2 \Big] 
   \nonumber \\
 && + 
   \frac{\bar{p} (\bar{p}-1)}{2} 
   \E \Big[ \mathds{1}_{\Omega_{ \mathcal{R}, n}} 
   |Y_{n}|^{\bar{p}-2} \big|
 \int_{t_n}^{t_{n+1}} 
  \mathcal{T}_6 \big( \mathcal{L} f (Y_n), h \big) \frac{h}{2} \, \dd s \big|^2 \Big]
  \nonumber \\
  && + 
    \bar{p} (\bar{p}-1)
    \E \Big[ \mathds{1}_{\Omega_{ \mathcal{R}, n}} 
   |Y_{n}|^{\bar{p}-2} \Big \langle \sum_{r=1}^{m} \mg{Y_n} \Delta W_r(n), \sum_{r=1}^{m} \int_{t_n}^{t_{n+1}}
   \mathcal{T}_4 \big( \mathcal{L} g^r(Y_n), h \big) h \, \dd W_r(s)
  \nonumber \\
  &&   \quad + 
   \sum_{r=1}^{m} 
   \int_{t_n}^{t_{n+1}} \int_{t_n}^{s} 
  \mathcal{T}_5 \big( ( \Lambda_{r} f ( Y_n) - \mathcal{L} g^r(Y_n) ), h \big)  \, \dd W_r(s_1) \dd s
   \Big \rangle \Big].
\end{eqnarray}
By Schwarz's inequality, \eqref{eq:the-result-I1}, \ref{ass:tame-h3}, \ref{ass:a1} and \ref{ass:a2}, we have
\begin{eqnarray}
    I_1    
   & \leq & 
  C h 
 + Ch
  \E \big[ \mathds{1}_{\Omega_{ \mathcal{R}(h), n}} |Y_{n}|^{\bar{p}} \big] 
 + Ch^2 \E \big[ \mathds{1}_{\Omega_{ \mathcal{R}(h), n}} |Y_{n}|^{4r+\bar{p}} \big]  
   + 
  Ch^{\tau+1} \E \big[ \mathds{1}_{\Omega_{ \mathcal{R}(h), n}} |Y_{n}|^{(2r+1)l_1+\bar{p}-1} \big] \nonumber \\
  && + 
  Ch^{3} \E \big[ \mathds{1}_{\Omega_{ \mathcal{R}(h), n}} |Y_{n}|^{4r+2\rho+\bar{p}-2} \big] 
 + Ch^{4} \E \big[ \mathds{1}_{\Omega_{ \mathcal{R}(h), n}} |Y_{n}|^{8r+\bar{p}} \big].
\end{eqnarray}
The scheme for estimating $I_2$ in \eqref{eq:decom-new-scheme-sub} is analogous to the approach used to estimate $I_2$ in \eqref{eq:esti-I2}. By leveraging Assumption \ref{ass:tame-h3} and an elementary inequality, we can show that
\begin{eqnarray}
  \label{eq:esti-I2-new-scheme}
   I_2 
   & \leq &  
   C \sum_{l=3}^{\bar{p}}
   \E \big[ \mathds{1}_{\Omega_{ \mathcal{R}, n}} 
   |Y_{n}|^{\bar{p}-l} 
   \big( h^l |\mf{Y_n}|^l + h^{\frac{l}{2}}
    | \sum_{r=1}^{m} \mg{Y_n} |^l
  + h^l |\Lambda_{r_1} g^r( Y_n)|^l + h^{\frac{3l}{2}}
    |\mathcal{L} g^r(Y_n)|^l \nonumber \\
  &&  \quad  +
  h^{2l} |\Lambda_{r} f ( Y_n) - \mathcal{L} g^r(Y_n)|^l + h^{\frac{3l}{2}}
    |\mathcal{L} f(Y_n)|^l 
  \big) \big]    \nonumber \\
& \leq & 
  Ch
   + 
   C \sum_{l=3}^{\bar{p}} h^l
   \E \big[ \mathds{1}_{\Omega_{ \mathcal{R}, n}} |Y_{n}|^{2rl+\bar{p}} \big]  
   + 
   C \sum_{l=3}^{\bar{p}} h^{\frac{l}{2}}
   \E \big[ \mathds{1}_{\Omega_{ \mathcal{R}, n}} |Y_{n}|^{rl+\bar{p}} \big]   \nonumber \\
   & & + 
   C \sum_{l=3}^{\bar{p}} h^{\frac{3l}{2}}
   \E \big[ \mathds{1}_{\Omega_{ \mathcal{R}, n}} |Y_{n}|^{3rl+\bar{p}} \big]
  + C \sum_{l=3}^{\bar{p}} h^{2l}
   \E \big[ \mathds{1}_{\Omega_{ \mathcal{R}, n}} |Y_{n}|^{4rl+\bar{p}} \big].
\end{eqnarray}
Similar to the proof of Lemma \ref{lem:Euler-scheme-moment-bound}, we set $\mathcal{R} = \mathcal{R}(h)=h^{-1/\mathds{G}}$ with 
$ \mathds{G}_1 = 6r \vee \frac{(2r+1)l_1-1}{\tau} $ to show that
\begin{equation}
  \label{eq:new-scheme-result-moment-bound-sub}
     \E \big[ \mathds{1}_{\Omega_{ \mathcal{R}, n}} |Y_n|^{\bar{p}}  \big] 
   \leq 
     C 
  \big( 1 + | Y_{0} | ^{\bar{p}} 
  \big). 
\end{equation}
It remains to estimate $ \E \big[ \mathds{1}_{\Omega_{ \mathcal{R}, n}^{c}} |Y_{n}|^{p}\big] $. It follows from \eqref{eq:second-order-new-scheme}, \eqref{eq:decom_inequality_Yn} and Assumption \ref{ass:tame-h3} that
\begin{eqnarray}
\label{eq:decom_ineq_Yn_new_scheme}
    | Y_{n+1} | 
 &  \leq &
| Y_n | + Ch^{1-\gamma_{1}} + \sum_{r=1}^{m} 
  C h^{-\gamma_2}
\big|W_r(t_{n+1})-W_r(t_n)\big|
 + \sum_{r=1}^{m} \sum_{r_1=1}^{m} 
 Ch^{-\gamma_1} \Big| \int_{t_n}^{t_{n+1}} \int_{t_n}^{s} \dd W_{r_1} (s_1) \dd W_r(s) \Big|
  \nonumber \\
 && +
  \sum_{r=1}^{m} 
  C h^{1-\gamma_3}
\big|W_r(t_{n+1})-W_r(t_n)\big| 
  + \sum_{r=1}^{m} 
  C h^{-\gamma_3}
 \Big| \int_{t_n}^{t_{n+1}} \int_{t_n}^{s} 
\dd W_{r} (s_1) \dd s \Big| +Ch^{2-\gamma_4}
  \nonumber \\
 & \leq &
   | Y_0 | 
 + (n+1)Ch^{1-\gamma_{1} }
 + \sum_{k=0}^{n}  \sum_{r=1}^{m} 
 C h^{-\gamma_2}
 \big|W_r(t_{k+1})-W_r(t_k)\big|
\nonumber \\
 && +  \sum_{k=0}^{n} \sum_{r=1}^{m} 
  C h^{-\gamma_3}
\Big| \int_{t_k}^{t_{k+1}} \int_{t_k}^{s} 
\dd W_{r} (s_1) \dd s \Big|
  + \sum_{k=0}^{n} 
 \sum_{r=1}^{m} \sum_{r_1=1}^{m} 
 Ch^{-\gamma_1} \Big| \int_{t_k}^{t_{k+1}} \int_{t_k}^{s} \dd W_{r_1} (s_1) \dd W_r(s) \Big|
  \nonumber \\
 &&
 + \sum_{k=0}^{n} 
 \sum_{r=1}^{m}
 Ch^{1-\gamma_3}
 \big|W_r(t_{k+1})-W_r(t_k)\big| + (n+1)Ch^{2-\gamma_{4} } .
\end{eqnarray}
Similar to arguments in the proof of \eqref{eq:resu_decom_Yn_subevent}-\eqref{eq:EM-result-subevent-mom-bound}, using
the H\"older inequality with 
$ \frac{1}{p'} + \frac{1}{q'} =1 $ 
for 
$ q' = \frac{\bar{p}}{(p\gamma_1+1)\mathds{G}_1} \vee \frac{\bar{p}}{(1+(\frac12 + \gamma_2)p)\mathds{G}_1}  
\vee  
\frac{\bar{p}}{(1-\frac12 p+ p\gamma_3) \mathds{G}_1}
\vee  
\frac{\bar{p}}{(1- p+ p\gamma_4) \mathds{G}_1}> 1 $
due to 
$ p \leq \frac{ \bar{p}-\mathds{G}_1}{1+\gamma_1 \mathds{G}_1} \wedge \frac{ \bar{p}-\mathds{G}_1}{1+(\frac12+\gamma_2)\mathds{G}_1} \wedge \frac{ \bar{p}- \mathds{G}_1}{1+(\gamma_3-\frac12)\mathds{G}_1} \wedge \frac{ \bar{p}- \mathds{G}_1}{1+(\gamma_4 -1 )\mathds{G}_1}$ and the Chebyshev inequality gives
\begin{equation*}
   \E 
      \big[ \mathds{1}_{\Omega_{ \mathcal{R}(h), n}^{c}}
   |Y_{n}|^{p}\big]  
  \leq 
  \big( \E \big[ |Y_{n}|^{p p'} \big]
   \big)^{\frac{1}{p'}}
   \sum_{i=0}^{n}
   \frac{\big(  \E \big[
   \mathds{1}_{\Omega_{ \mathcal{R}(h), i-1}}
    |Y_i|^{\bar{p}}
   \big] \big)^{\frac{1}{ q'}}}
   {\mathcal{R}(h)^{\bar{p}/q'}}.
\end{equation*}
Using the H\"older inequality, \eqref{eq:decom_ineq_Yn_new_scheme} and the elementary inequality implies
\begin{eqnarray} 
\label{eq:Yn-estimate-new-scheme}
    \big( \E \big[ |Y_{n}|^{p p'} \big]
   \big)^{\frac{1}{p'}} 
  & \leq & 
  C \big( |Y_0|^{\bar{p}} 
  + n^{\bar{p}}  h^{\bar{p}(1-\gamma_1)}
  + n^{\bar{p}}  h^{\bar{p}(\frac{1}{2}-\gamma_2)}
  + n^{\bar{p}}  h^{\bar{p}(\frac{3}{2}-\gamma_3)}
  + n^{\bar{p}}  h^{\bar{p}(2-\gamma_4)}
  \big)^{\frac{p}{\bar{p}}}
  \nonumber \\
%
 & \leq & 
    C h^{-p \gamma_1}
   + C h^{-\frac{p}{2} -p \gamma_2}
   + C h^{\frac{p}{2}-p \gamma_3}
   + C h^{p-p \gamma_4}
 + C \big( 1 + |X_0|^{\bar{p}} \big)
   ^{\frac{p}{\bar{p}}}.
\end{eqnarray}
Notice that arguments similar to those in the proof of Lemma \ref{lem:Euler-scheme-moment-bound} yield that
\begin{eqnarray}
    \E  \big[ 
   \left|Y_{n}\right|^{p}
   \big] 
 & = &
 \E 
      \big[ \mathds{1}_{\Omega_{ \mathcal{R}(h), n}}
   |Y_{n}|^{p}\big]  
  +
     \E 
      \big[ \mathds{1}_{\Omega_{ \mathcal{R}(h), n}^{c}}
   |Y_{n}|^{p}\big]    \nonumber \\
  & \leq &
  \big(
  \E \big[ \mathds{1}_{\Omega_{ \mathcal{R}(h), 
  n}} |Y_{n}|^{\bar{p}}  \big] 
   \big)^{\frac{p}{\bar{p}}}
 + \E 
      \big[ \mathds{1}_{\Omega_{ \mathcal{R}(h), n}^{c}}
   |Y_{n}|^{p}\big]   \nonumber \\
  & \leq &
C \big( 1+|X_0|^{p} \big)^{\beta}
\end{eqnarray}
for all integer $ p \in \big[ 1, 
 \frac{ \bar{p}-\mathds{G}_1}{1+\gamma_1 \mathds{G}_1} \wedge \frac{ \bar{p}-\mathds{G}_1}{1+(\frac12+\gamma_2)\mathds{G}_1} \wedge \frac{ \bar{p}-\mathds{G}_1}{1+(\gamma_3-\frac12)\mathds{G}_1} \wedge \frac{ \bar{p}-\mathds{G}_1}{1+(\gamma_4 -1 )\mathds{G}_1}] $. Here $\beta =1+ \frac{\bar{p}}{pq'} = 1+ \frac{(p\gamma_1+1)\mathds{G}_1}{p} \wedge \frac{(1+\frac12 p+ p\gamma_2)\mathds{G}_1}{p} \wedge \frac{(1-\frac12 p+ p\gamma_3) \mathds{G}_1}{p} \wedge \frac{(1- p+ p\gamma_4) \mathds{G}_1}{p} $.
 
Then, by Jensen's inequality, \eqref{eq:second-order-new-scheme-moment-bound-result} holds for non-integer $p$ as well.

Except for the second term $I_1$ on the right-hand side of \eqref{eq:decom-new-scheme-sub}, the proof of bounded moments for the scheme \eqref{eq:second-order-new-scheme} under Assumption \ref{ass:tame-h2p} is identical to that under Assumption \ref{ass:tame-h2}. Applying the elementary inequality and Assumption \ref{ass:tame-h2p}, we get
\begin{eqnarray*}
    I_1    
   & \leq & 
  C h 
 + Ch
  \E \big[ \mathds{1}_{\Omega_{ \mathcal{R}(h), n}} |Y_{n}|^{\bar{p}} \big] 
 + Ch^2 \E \big[ \mathds{1}_{\Omega_{ \mathcal{R}(h), n}} |Y_{n}|^{4r+\bar{p}} \big]  
    \nonumber \\
  && + 
  Ch^{3} \E \big[ \mathds{1}_{\Omega_{ \mathcal{R}(h), n}} |Y_{n}|^{4r+2\rho+\bar{p}-2} \big] 
 + Ch^{4} \E \big[ \mathds{1}_{\Omega_{ \mathcal{R}(h), n}} |Y_{n}|^{8r+\bar{p}} \big].
\end{eqnarray*}
Similar to the proof above, 
we choose $\mathcal{R} = h^{-1/\mathds{G}_1}$ and obtain a similar conclusion.
\end{proof}
\subsection{Weak convergence of the scheme \eqref{eq:second-order-new-scheme}}
To verify the one-step approximation errors \eqref{eq:error_of_general_one_step_approximation}-\eqref{eq:estimate_general_one_step_approximate_solution} in Theorem \ref{thm:fundamental_weak_convergence}, we  examine the one-step approximation of the numerical scheme \eqref{eq:second-order-new-scheme}, which is expressed as follows:
\begin{eqnarray}
\label{eq:one-step-second-order-new-scheme}
Y(t,x;t+h)
  & = & 
  x + \mf{ x } h
    + \sum_{r=1}^{m}
     \int_{t}^{t+h}
     \mg{ x } 
    \, \dd W_r(s) \nonumber \\
 &&  +
    \sum_{r=1}^{m} \sum_{r_1=1}^{m} \int_{t}^{t+h} \int_{t}^{s} 
   \mathcal{T}_3 \big(\Lambda_{r_1} g^r(x), h \big) \, \dd W_{r_1}(s_1) \dd W_r(s)
  +
   \sum_{r=1}^{m} \int_{t}^{t+h}
   \mathcal{T}_4 \big( \mathcal{L} g^r(x), h \big) h \, \dd W_r(s)   \nonumber \\
  &&  + 
   \sum_{r=1}^{m} 
   \int_{t}^{t+h} \int_{t}^{s} 
\mathcal{T}_5 \big( ( \Lambda_{r} f (x) - \mathcal{L} g^r(x) ), h \big)  \, \dd W_r(s_1) \dd s
  + \int_{t}^{t+h} 
  \mathcal{T}_6 \big( \mathcal{L} f (x), h \big) 
  \frac{h}{2} \, \dd s.
\end{eqnarray}
Based on Assumption \ref{ass:diff-T1-T6-new-scheme-MT}, we can derive accurate estimates of the strong errors of the one-step approximations, which are essential for obtaining the desired estimates of the weak errors.
\begin{lem}
 \label{lem:norm-esti-one-step-new-MT}
Let Assumptions \ref{ass:a1}, \ref{ass:a2}
and Assumption \ref{ass:diff-T1-T6-new-scheme-MT} hold, then  
for any $ \mathfrak{p} \geq 1 $ it holds that
\begin{align}
  \label{eq:norm-esti-one-step-EM-balanceed-EM}
    \| \delta_{Y_{MT},x} -  \delta_{Y,x} 
    \|_{L^{2 \mathfrak{p}}(\Omega;\mathbb{R}^d)}
     & \leq
     C (1+|x|^{r_1})h^{q_0+\frac{1}{2}} ,
\\
  \label{eq:norm-esti-one-step-banlanced-EM}
      \|
           \delta_{Y,x}
       \|_{L^{2 \mathfrak{p}}(\Omega;\mathbb{R}^d)}
    & \leq C  
(1+|x|^{r_1})
 h^{\frac{1}{2}},
\\
  \label{eq:norm-esti-one-step-EM}
      \|
           \delta_{Y_{MT},x}
       \|_{L^{2 \mathfrak{p}}(\Omega;\mathbb{R}^d)}
    & \leq C  
 (1+|x|^{4r+1})h^{\frac{1}{2}},
\end{align}
where $r_1 = (4r+1)\eta_{q_0}$, $r$ is from \eqref{eq:f-growth} and  $\eta_{q_0}$ is from Assumption  \ref{ass:diff-T1-T6-new-scheme-MT}.
\end{lem}
\begin{proof}
Using the Minkowski inequality, the moment inequality \cite[Theorem 7.1]{mao2008stochastic} and  Assumption \ref{ass:diff-T1-T6-new-scheme-MT} gives
\begin{eqnarray}
 \label{eq:norm-esti-new-scheme-MT}
   \| \delta_{Y_{MT},x} -  \delta_{Y,x} 
     \|_{L^{2 \mathfrak{p} }(\Omega;\mathbb{R}^d)}
 & \leq & 
  C h \big\| f(x)-\mf{x}
  \big\|_{L^{2 \mathfrak{p} } (\Omega, \R^d )} 
 +  C h^{\frac{1}{2}} 
 \big\| g^r(x)-\mg{x}
  \big\|_{L^{2 \mathfrak{p} } (\Omega, \R^d )}   \nonumber  \\
 & & + C h 
  \big\| \Lambda_{r_1} g^r(x)-
    \mathcal{T}_3 \big(\Lambda_{r_1} g^r(x), h \big)
  \big\|_{L^{2 \mathfrak{p} } (\Omega, \R^d )}  \nonumber \\
 && + C h^{\frac{3}{2}}
  \big\| \mathcal{L} g^r(x)-\mathcal{T}_4 \big( \mathcal{L} g^r(x), h \big)
  \big\|_{L^{2 \mathfrak{p} } (\Omega, \R^d )} 
  + C h^{2}
  \big\| \mathcal{L} f (x)-\mathcal{T}_6 \big( \mathcal{L} f(x), h \big) \|_{L^{2 \mathfrak{p} } (\Omega, \R^d )} 
  \nonumber \\
 && + C h^{\frac{3}{2}}
  \big\| \big( \Lambda_{r} f (x) - \mathcal{L} g^r(x) \big) -
    \mathcal{T}_5 \big( ( \Lambda_{r} f (x) - \mathcal{L} g^r(x) ), h \big)
  \big\|_{L^{2 \mathfrak{p} } (\Omega, \R^d )} \nonumber \\
 & \leq &
 C (1+|x|^{r_1})h^{q_0 + \frac{1}{2}} ,
\end{eqnarray}
which implies the first assertion.
Similar arguments lead to
\begin{eqnarray}
   \| \delta_{Y,x} 
     \|_{L^{2 \mathfrak{p} }(\Omega;\mathbb{R}^d)}
 & \leq & 
  C h \big\| \mf{x}
  \big\|_{L^{2 \mathfrak{p} } (\Omega, \R^d )} 
 +  C h^{\frac{1}{2}} 
  \big\| \mg{x}
  \big\|_{L^{2 \mathfrak{p} } (\Omega, \R^d )}   \nonumber  \\
 & & + C h 
  \big\|
    \mathcal{T}_3 \big(\Lambda_{r_1} g^r(x), h \big)
  \big\|_{L^{2 \mathfrak{p} } (\Omega, \R^d )}  \nonumber \\
 && + C h^{\frac{3}{2}}
  \big\| \mathcal{T}_4 \big( \mathcal{L} g^r(x), h \big)
  \big\|_{L^{2 \mathfrak{p} } (\Omega, \R^d )} 
  + C h^{2}
  \big\|\mathcal{T}_6 \big( \mathcal{L} f(x), h \big) \|_{L^{2 \mathfrak{p} } (\Omega, \R^d )} 
  \nonumber \\
 && + C h^{\frac{3}{2}}
  \big\|
    \mathcal{T}_5 \big( ( \Lambda_{r} f (x) - \mathcal{L} g^r(x) ), h \big)
  \big\|_{L^{2 \mathfrak{p} } (\Omega, \R^d )} \nonumber \\
 & \leq &
  C (1+|x|^{r_1})
 h^{\frac{1}{2}}.
\end{eqnarray}
The estimate of 
$\| \delta_{Y_{MT},x}
\|_{L^{2 \mathfrak{p} }(\Omega;\mathbb{R}^d)}$ can be done similarly.
\end{proof}
\begin{lem}[One-step weak error estimates]
\label{lem:balanced-one-step-error}
Under the same conditions of Lemma \ref{lem:norm-esti-one-step-new-MT}, 
the one-step approximation 
\eqref{eq:one-step-second-order-new-scheme}
of the numerical scheme \eqref{eq:second-order-new-scheme} satisfies
\begin{eqnarray*}
       \Big| \E \Big [ 
         \prod_{j=1}^s (\delta_{X,x})^{i_j}
       - 
        \prod_{j=1}^s (\delta_{Y,x})^{i_j} 
          \Big] \Big|
        & \leq &
         C \big( 1+|x|^{r_2}+|x|^{(2q+1)r_1}
      \big) h^{3 \wedge (q_0+1)},
         \, s=1, \ldots, 2q+1,
       \nonumber \\
%
  	\Big\| \prod_{j=1}^{2q+2}  
  	      ( \delta_{Y, x} )^{i_j} 
  	\Big\|_{L^2 (\Omega; \R)}
     & \leq &  
    C (1+|x|^{(2q+2) r_1})  h^{q+1}, \nonumber \\
%
   \Big\| \prod_{j=1}^{2q+2}  
     ( \delta_{X, x} )^{i_j} 
   \Big\|_{L^2 (\Omega; \R)}
& \leq & C (1+|x|^{(2q+2)(2r+1)}) h^{q+1},
\end{eqnarray*}
Here $r_1$ comes from Lemma \ref{lem:norm-esti-one-step-new-MT} and $r_2 = (2q+1)(6r+1).$
\end{lem}
\begin{proof}
By the triangle inequality, we have
\begin{eqnarray}
 &&
      \Big| \E \Big [ 
          \prod_{j=1}^s (\delta_{X,x})^{i_j}
         - 
          \prod_{j=1}^s (\delta_{Y,x})^{i_j} 
       \Big] \Big|   \nonumber \\
  &  \leq  &
       \Big| \E \Big [ 
          \prod_{j=1}^s (\delta_{X,x})^{i_j}
       - 
          \prod_{j=1}^s (\delta_{Y_{MT},x})^{i_j} \Big] \Big|
       + 
       \Big| \E \Big [ 
          \prod_{j=1}^s (\delta_{Y_{MT},x})^{i_j}
       - 
          \prod_{j=1}^s (\delta_{Y,x})^{i_j}
       \Big] \Big|      \nonumber  \\
   &   =:  &
            J_{1} + J_{2},
            \qquad
            s=1, \ldots, 2q+1.
\end{eqnarray}
Building on the proof presented in \cite[Lemma 1.4, Chapter 2]{Milstein2004stochastic}, we establish the inequality $ J_1 \leq C (1+|x|^{r_2}) h^3 $.
Now, our focus shifts to the estimation of $ J_2 $. Specifically, when $ s=1 $, we can estimate the second term $ J_2 $ in the following manner:
\begin{eqnarray}
  \label{eq:error-estimate-s1-J2}
       \big| \E \big [  
           ( \delta_{Y_{MT},x})^{i_1}
        -  (\delta_{Y,x})^{i_1} 
        \big] \big|      
  & \leq &
        \big | \E \big[
        \big(f(x)-\mf{x}
        \big) h
        + \frac{h^2}{2}  
        \big( 
        \mathcal{L} f (x)-\mathcal{T}_6 \big( \mathcal{L} f(x), h \big) \big) \nonumber  \\          
   & & +
   \sum_{r=1}^{m} \sum_{r_1=1}^{m} \int_{t}^{t+h} \int_{t}^{s} 
   \big( \Lambda_{r_1} g^r(x)-
    \mathcal{T}_3 \big(\Lambda_{r_1} g^r(x), h \big) \big) \, \dd W_{r_1}(s_1) \dd W_r(s) 
  \big] \big|  
   \nonumber  \\
 & \leq &
    C (1+|x|^{r_1}) h^{q_0+1}.
\end{eqnarray}
With the H\"older inequality and Lemma \ref{lem:norm-esti-one-step-new-MT} available, one can readily obtain the result for $s=2$, yielding 
\begin{eqnarray}
    \label{eq:error-estimate-s2-J2}
  && \Big| \E \big[ 
         \prod_{j=1}^2 (\delta_{Y_{MT},x})^{i_j}
     - 
          \prod_{j=1}^2 (\delta_{Y,x})^{i_j} 
     \big]  \Big|  \nonumber \\
   &  \leq &  
   \big \| \delta_{Y_{MT},x} -  \delta_{Y,x} 
    \big \|_{L^{2}(\Omega;\mathbb{R}^d)}
   \big \| \delta_{Y_{MT},x} \big \|_{L^{2}(\Omega;\mathbb{R}^d)} 
   + \big \| \delta_{Y,x} \big \|_{L^{2}  
    (\Omega;\mathbb{R}^d)}
    \big \| \delta_{Y_{MT},x} -  \delta_{Y,x} 
    \big \|_{L^{2}(\Omega;\mathbb{R}^d)} 
    \nonumber \\
   & \leq &
 C (1+|x|^{2 r_1})h^{q_0 +1}.
\end{eqnarray} 
Utilizing a technique analogous to the one demonstrated in Equation \eqref{eq:error-estimate-s2-J2}, we obtain the following result
\begin{equation}
    \label{eq:error-estimate-s3-J2}
  \Big| \E \big[ 
         \prod_{j=1}^3 (\delta_{Y_{MT},x})^{i_j}
     - 
          \prod_{j=1}^3 (\delta_{Y,x})^{i_j} 
     \big]  \Big| 
    \leq
C (1+|x|^{3 r_1}) h^{q_0 + \frac{3}{2}}.
\end{equation}
Similar to 
\eqref{eq:error-estimate-s3-J2}, we have
\begin{equation*}
       J_2 
         \leq
        C (1+|x|^{(2q+1) r_1})
           h^{q_0+1},
             \,\,
             s = 1,\ldots, 2q+1.
\end{equation*}
Using the H\"{o}lder inequality and 
\eqref{eq:norm-esti-one-step-banlanced-EM} yields
\begin{equation*}
     	\Big\| \prod_{j=1}^{2q+2}  
     	   ( \delta_{Y, x} )^{i_j} 
     	\Big\|_{L^2 (\Omega; \R)}
      \leq   
       	 \prod_{j=1}^{2q+2}
       	   \big\|   
       	   ( \delta_{Y, x} )^{i_j} 
       	   \big\|_{L^{4q+4} (\Omega; \R)}
      \leq
         C (1+|x|^{(2q+2) r_1}) h^{q+1}.
\end{equation*}
\vskip -20pt 
\end{proof}
By Lemmas \ref{lem:new-scheme-moment-bound} and  \ref{lem:balanced-one-step-error}
and Theorem \ref{thm:fundamental_weak_convergence}, we can readily derive the weak convergence order of the numerical scheme \eqref{eq:second-order-new-scheme} in Theorem \ref{thm:MMT-global-conver-order}.
\section{Conclusion and discussion}
We have discussed the weak convergence orders of modified Euler and Milstein-Talay schemes. When 
solutions and numerical solutions have finite moments, it is crucial to find the highest moments of numerical solutions to determine convergence orders. 
We show how the modification maps in the modified schemes determine the highest order of the moments. We investigate the effects of limited moments on weak convergence orders theoretically and numerically. 
However, the theoretical orders of weak convergence seem to be lower than the empirically observed orders. Specifically, 
the modified Euler schemes seem to be of first order even with limited moments while the modified Milstein-Talay schemes can degenerate in convergence orders - it may be of first order or less. 
Further exploration in this direction is needed to understand the mismatch. 

\bibliographystyle{abbrv}
\bibliography{bibfile}
\appendix
\section{Moment estimates for an SDE with random forcing}
\label{app:gen-SDE-well-pose}
In this section, we establish general moment estimates for the SDE \eqref{eq:Problem_SDE}, which will be used in the analysis in Appendix \ref{appen:deri-wrt-ini-val}. 

Let $ \mathcal{H}_1 $, $\mathcal{H}_2$ be  random functions defined on $\Omega \times [0,T]$ and be adapted to the natural filtration generated by $W(t)$. 
\begin{assumption}
\label{ass:general-fun-growth-con}
There exists a constant
$K \geq 0$ and  $q_1 \geq 1,\, q_2 \geq 2$, such that 
for all $t \in [0,T]$,
\begin{equation*}
\mean{\int_0^t
|\mathcal{H}_1(\omega,s)|^{q_1}\,ds} \leq K, \quad 
\mean{\int_0^t|\mathcal{H}_2(\omega,s)|^{q_2}\,ds} \leq K. 
\end{equation*}
\end{assumption}
\begin{lem}
\label{lem:general-SDE-form-mom-bou}
Under Assumption \ref{ass:coefficient_function_assumption} with sufficiently large $p_0 \geq 2 $
and Assumption \ref{ass:general-fun-growth-con}, the stochastic differential equation
\begin{equation}
\label{eq:general-SDE-form}
     Z(t) = Z(0)+ \int_{0}^{t}\big( Df(X(s)) Z(s) + \mathcal{H}_1(\omega,s) \big) \,\dd s
    + \int_{0}^{t}\big( Dg(X(s)) Z(s) + \mathcal{H}_2(\omega, s) \big) \,\dd W(s),
\end{equation}
admits a unique strong solution $\{Z(t) : t \in [0, T ]\}$ such that for any
$ p_2 \in [1, q_1 \wedge q_2 \wedge  \frac{p_0+1}{2}]$, 
\begin{equation}
\label{eq:general-solu-mom-bound}
\E [| Z(t)|^{p_2}] 
  \leq C_{p_2}
     \E \Big[
     (|Z(0)|^2)^{p_2}
 + \int_{0}^{t} |\mathcal{H}_1(\omega,s)|^{p_2} \, \dd s
 +\int_{0}^{t} |\mathcal{H}_2(\omega,s)|^{p_2} \, \dd s \Big]
 \exp(C_{p_2}T),
\, t \in [0,T], 
\end{equation}
where the constant $C > 0$ depends only on $K$, $p_2$ and $T$. Here $q_1$ and $q_2$ are from  Assumption \ref{ass:general-fun-growth-con}.
\end{lem}
The existence and uniqueness of the solution can be proved
by applying a truncation argument and 
Assumption \ref{ass:coefficient_function_assumption}, see e.g., in \cite[Chapter 2]{mao2008stochastic}.
Applying It\^{o}'s formula and Assumption 
\ref{ass:a3} leads to the moment bound \eqref{eq:general-solu-mom-bound}.
%

If $\mathcal{H}_1(\omega,t,\epsilon)$ and $\mathcal{H}_2(\omega,t,\epsilon)$  depend on $\epsilon$ and have limits in a certain sense when $\epsilon\to0$, then we can show that the solution to the resulting equation, denoted by $\mathcal{S}_{\epsilon}(t)$, also has a limit.

\begin{cor}
\label{cor:general-SDE-form-con-beha}
 Under the same conditions in Lemma \ref{lem:general-SDE-form-mom-bou}, if there exists $p_3\geq 1$ and $\bar{p_3} \geq 2$ such that
\begin{equation}
\label{eq:ass-limit-behavior-H-fun}
\lim_{\epsilon \rightarrow 0}
\E\big[\int_0^t
|\mathcal{H}_1(\omega,s,\epsilon)|^{p_3}\,ds \big] = 0, \quad
\lim_{\epsilon \rightarrow 0}
\E \big[\int_0^t|\mathcal{H}_2(\omega,s,\epsilon)|^{\bar{p}_3}\,ds \big] = 0,
\end{equation}
then  
it holds that
\begin{equation*}
    \lim_{\epsilon \rightarrow 0} \E \big[|\mathcal{S}_{\epsilon}(t)|^{\widetilde{p}_3} \big] \leq C_{\widetilde{p}_3} \lim_{\epsilon \rightarrow 0} \E \big[|\mathcal{S}_{\epsilon}(0)|^{\widetilde{p}_3}\big],\quad \widetilde{p}_3 \in [1, p_3 \wedge \bar{p_3} \wedge  \frac{p_0+1}{2}],
\end{equation*}
where $ C_{\widetilde{p}_3} $ is independent of $ \epsilon $ but dependent on $ \widetilde{p}_3 $ and $T$.
\end{cor}
\section{Derivatives of the solutions in the  initial condition}
\label{appen:deri-wrt-ini-val}
In this section, we focus on the SDEs that characterize the derivatives of $X^{0,x_0}$ with respect to the initial condition $x_0$.
We adopt the notation used in \cite{Cerrai2001second}. 
For random functions, we refer to Page 166 of \cite{khasminskii2011stochastic} and below for the definition of differentiability in the $L^{p}(\Omega) (p \geq 1)$ sense.
For simplicity, we present the proofs for the derivatives up to order two. 

\begin{defn}[$L^{p}$ differentiable]
\label{defi:mean_square_differential}
Let 
   $ \Psi \colon 
     \Omega \times \mathbb{R}^{d} \rightarrow
       \mathbb{R} $
   and 
   	$  \phi_{i} \colon \Omega \times \mathbb{R}^{d} 
   	 \rightarrow \mathbb{R} $
be random functions and suppose for each $i \in \{ 1,2,\cdots, d \}$
   \begin{equation}
     \lim_{\tau \rightarrow 0 }
        \E 
        \Big [ \Big|
         \frac{1}{\tau}
         \big [
          \Psi( x + \tau e_{i})
        -
          \Psi( x ) \big]
         -
          \phi_i(x) 
          \Big|^{p}
         \Big]
        = 0,
        \quad
    p \geq 1,
   \end{equation}
where $e_i \in \R^d $ is a unit vector in $\R^d$, with the $i$-th element being $1$. 
Then $ \Psi  $ is called to be differentiable, with
$ \phi = (\phi_{1}, \phi_{2}, \cdots, \phi_{d} ) $
being the derivative (in the $L^{p}(\Omega)$ sense) of $ \Psi  $  
and  we write  $ \partial_{(i)}\Psi = \phi_{i} $.
\end{defn}
This definition can be extended to the case of vector-valued functions by considering their components separately.
\par
We now apply conclusions in \ref{app:gen-SDE-well-pose} to analyze the derivatives of the solution to the SDE  \eqref{eq:Problem_SDE} concerning the initial conditions in the sense of $L^{p}(\Omega)$ differentiability. 
To do so, we first examine the limit behavior of solutions for the SDEs with perturbed initial conditions.
\begin{prop}
\label{lem:solution-wrt-inital-value-con}
%
Let Assumption \ref{ass:coefficient_function_assumption}  and the inequality \ref{assu:a3'-mono-condi} hold.
Then the solution $X^{0,x_0}(t)$ to \eqref{eq:Problem_SDE} satisfies, for any  $\epsilon>0$,
 \begin{equation}
 \label{eq:increment-solution-esti}
  \E \big[ |X^{0,x_0+\epsilon e_i}(t)-X^{0,x_0}(t)|^{p_4}\big] \leq 
  |\epsilon|^{ p_4} \exp{(C_{p_4}T)}, \quad
  p_4 \in [1, p_0'],
 \end{equation}
where $e_i \in \R^d $
is defined in Definition \ref{defi:mean_square_differential}
and $ C_{p_4} $ is independent of $ \epsilon $ but dependent on $ p_4 $.
For brevity, we denote
 \begin{align}
 \label{eq:short-notation-incre-sol-f}
  \chi_{x_0,\epsilon}(s) 
  &:= X^{0,x_0+\epsilon e_i}(s) -X^{0,x_0}(s),
  \,
  \Delta f^{x_0+\epsilon e_i,x_0}(s):=
  f \big( X^{0,x_0+\epsilon e_i}(s) \big)
  - f \big( X^{0,x_0}(s) \big), \\
 \label{eq:short-notation-incre-g}
  \Delta g^{x_0+\epsilon e_i,x_0}(s)
  &:= g \big( X^{0,x_0+\epsilon e_i}(s) \big)
  - g \big( X^{0,x_0}(s) \big).
 \end{align}
\end{prop}
\begin{proof}
By the definition of $ X^{0,x_0+\epsilon e_i}(t) $,  
 \begin{equation*}
    X^{0,x_0+\epsilon e_i}(t) = x_0+\epsilon e_i
 + \int_0^t f( X^{0,x_0+\epsilon e_i}(s) ) 
 \, \dd s 
+
\int_0^t g( X^{0,x_0+\epsilon e_i}(s) ) 
 \, \dd W (s),
\quad
0 < t \leq T.
 \end{equation*}
Using the It\^{o} formula on the functional $|x|^{p_4}$ for $x \in \R^d$, $p_4 \geq 2$ and the Schwarz inequality, we obtain that 
\begin{eqnarray*}
  |\chi_{x_0,\epsilon}(t)|^{p_4} 
   & \leq &
     |\epsilon e_i|^{p_4}  
     + p_4 \int_{0}^{t}
	  |\chi_{x_0,\epsilon}(s)|^{p_4-2} 
  \langle \chi_{x_0,\epsilon}(s), 
	\Delta f^{x_0+\epsilon e_i,x_0}(s)
  \rangle 
	   \, \dd s \\ 
	  &&+
	  \frac{p_4 (p_4-1)}{2} 
	   \int_{0}^{t}
	   |\chi_{x_0,\epsilon}(s)|^{p_4-2} 
	| \Delta g^{x_0+\epsilon e_i,x_0}(s) 
        |^{2} 
	   \, \dd s  \notag \\
    && +
      p_4
     \int_{0}^{t}
      |\chi_{x_0,\epsilon}(s)|^{p_4 -2} 
     \langle
      \chi_{x_0,\epsilon}(s),
      \Delta g^{x_0+\epsilon e_i,x_0}(s)
       \, \dd W(s)
       \rangle.
\end{eqnarray*}
Applying the inequality \ref{assu:a3'-mono-condi} and  the properties of It\^{o} integrals implies
 \begin{equation}
  \E [  |\chi_{x_0,\epsilon}(t)|^{p_4} ]  \leq |\epsilon|^{p_4}  
   + C_{p_4} \int_0^{t} 
   \E [  |\chi_{x_0,\epsilon}(s)|^{p_4} ]
   \dd s,  \quad 
   p_4 \in [2, p_0'],
 \end{equation}
where $ C_{p_4} $ is independent of $ \epsilon $ but dependent on $ p_4 $.
The Gronwall inequality shows that for all $t \in [0,T]$,
$\E \big[ |X^{0,x_0+\epsilon e_i}(t)-X^{0,x_0}(t)|^{p_4}\big] \leq 
  |\epsilon|^{ p_4} \exp{(C_{p_4}T)}, 
  p_4 \in [2, p_0'].$
For the case $p_4 \in [1,2)$, applying the H\"older inequality immediately suggests \eqref{eq:increment-solution-esti}. Thus we complete the proof.
\end{proof}
\begin{cor}
 \label{cor:well-posedness-deri-first}
 Under the same assumptions of Proposition \ref{lem:solution-wrt-inital-value-con} with 
 $p_0 \geq 2r + 1$
 ($r$ is defined in \eqref{ass:drift-f}),
  for each $ i \in  \{1,\dots, d\} $, $  \xi_{i}(t):= \frac{\partial X^{0,x_0}(t)}{ \partial {x^{(i)}_{0}}} \in \R^d $ exists and 
satisfies the following stochastic differential equation
\begin{equation}
	\label{eq:first-variation-equation}
	\left\{
	\begin{aligned}
   \dd \xi_{i}(t) & = 
     Df( X(t) ) \xi_{i}(t) \, \dd t 
   + Dg(X(t)) \xi_{i}(t) \, \dd W(t),
      \quad t \in (0,T], \\
      \xi_{i}(0) & = e_i,
	\end{aligned}\right.
\end{equation}
where $e_i \in \R^d $ comes from Definition \ref{defi:mean_square_differential}.
Furthermore, there exists a constant $ C > 0 $, such that
 \begin{align}
   \label{lem:mom-bound-first-deri}
    & \E [|\xi_{i}(t)|^{p_5}] 
      \leq C \exp(C_{p_5} t),
    \quad 
    1 \leq p_5 \leq p_0, 
    \quad \text{for all \,} t \in [0,T], \\
  \hbox{and}   \quad
   \label{lem:first-derivative}
   & \lim_{\epsilon \rightarrow 0} 
  \E \big[ |\frac{\chi_{x_0,\epsilon}(t)} {\epsilon} - \xi_{i}(t)|^{p_6} \big]   = 0, \quad 1 \leq p_6 \leq 
  \frac{p_0}{(2r \vee 0 )+1},
    \quad \text{for all \,} t \in [0, T],
 \end{align}
where $\chi_{x_0,\epsilon}(t) =: X^{0,x_0+\epsilon e_i}(t) -X^{0,x_0}(t)$.
\end{cor}
\begin{proof}
Following the lines similar to those in Lemma \ref{lem:general-SDE-form-mom-bou},
\eqref{lem:mom-bound-first-deri} can be derived by observing that $\xi_i(t)$ satisfies the equation \eqref{lem:general-SDE-form-mom-bou}  with $\mathcal{H}_1$ = $\mathcal{H}_2$ = 0, $Z(0) = e_i$.
Below, we apply Corollary \ref{cor:general-SDE-form-con-beha} to derive \eqref{lem:first-derivative}. First of all,
$\frac{\chi_{x_0,\epsilon}(t)} {\epsilon} - \xi_{i}(t)
$ satisfies the equation 
\eqref{eq:general-SDE-form}
with zero initial condition and $\mathcal{H}_1$ and $\mathcal{H}_2$ defined by 
\begin{eqnarray*}
\mathcal{H}_1(\omega,s,\epsilon) & = & 
\frac{f \big( X^{0,x_0+\epsilon e_i}(s) \big)
  - f \big( X^{0,x_0}(s) \big)}{\epsilon} - Df(X^{0,x_0}(s))
 \frac{\chi_{x_0,\epsilon}(s)} {\epsilon}, \notag \\
 \mathcal{H}_2(s)(\omega,s,\epsilon)  & = & \frac{g \big( X^{0,x_0+\epsilon e_i}(s) \big)
  - g \big( X^{0,x_0}(s) \big)}{\epsilon}  - Dg(X^{0,x_0}(s))
 \frac{\chi_{x_0,\epsilon}(s)} {\epsilon}. 
\end{eqnarray*}
It only requires to verify that the corresponding $\mathcal{H}_1$ and $\mathcal{H}_2$ satisfy \eqref{eq:ass-limit-behavior-H-fun}.  
In fact, using the Taylor expansion, the H\"older inequality with  $\frac{1}{\theta_1} + \frac{1}{\theta_2} = 1 $, for $\theta_2 = \frac{p_0}{p_6} $
and Proposition \ref{lem:solution-wrt-inital-value-con}, we have 
  \begin{eqnarray*}  
  & \E\big[\int_{0}^{t} 
  |\mathcal{H}_1|^{p_6} \, \dd s
  \big] & \leq \int_{0}^{t}
  \Big( \E \big[|\int_{0}^{1}
    D f \big( X^{0,x_0}(s) 
    + \lambda (X^{0,x_0+\epsilon e_i}(s)-X^{0,x_0}(s)) \big) - Df(X^{0,x_0}(s)) \, \dd \lambda |^{p_6\theta_1}\big]\Big)^{1/\theta_1}
    \notag \\
&&  \qquad \big( \E \big[|\frac{\chi_{x_0,\epsilon}(s)} {\epsilon} |^{p_6 \theta_2}\big] 
    \big)^{1/\theta_2} \, \dd s
\notag \\ 
  & & \leq C \int_{0}^{t}
  \Big( \E \big[|\int_{0}^{1}
    D f \big( X^{0,x_0}(s) 
    + \lambda (X^{0,x_0+\epsilon e_i}(s)-X^{0,x_0}(s)) \big) - Df(X^{0,x_0}(s)) \, \dd \lambda |^{p_6\theta_1}\big]\Big)^{1/\theta_1} \, \dd s.
  \end{eqnarray*}
By the Assumption \ref{ass:a1}, \eqref{eq:esti_sol_SDE_wrt_initial} and the dominated convergence theorem, we get
\begin{equation}
\label{eq:limit-first-derivative-H1}
    \lim_{\epsilon \rightarrow 0}
    \E\big[\int_{0}^{t}
  |\mathcal{H}_1|^{p_6} \, \dd s \big] = 0,
  \quad 1 \leq p_6 \leq 
  \frac{p_0}{(2r \vee 0 )+1},
  \quad \text{for all \,} t \in [0, T].
\end{equation}
Following a similar way as \eqref{eq:limit-first-derivative-H1} and Assumption \ref{ass:a2}, we have
\begin{equation*}
    \lim_{\epsilon \rightarrow 0} \E\big[\int_{0}^{t}
  |\mathcal{H}_2|^{p_6 \vee 2} \, \dd s \big] = 0,
  \quad 1 \leq p_6 \leq 
  \frac{p_0}{((\rho -1) \vee 0) +1},
  \quad \text{for all \,} t \in [0, T].
\end{equation*}
The proof is thus complete.
\end{proof}

We now extend the above result to higher-order derivatives. 
\begin{prop}
\label{prop:con-first-deri}
 Under the same assumptions of Proposition \ref{lem:solution-wrt-inital-value-con} with $p_0 \geq 2r + 1$, 
the solution $\xi_{i}(t)$ of \eqref{eq:first-variation-equation} satisfy
 \begin{equation}
 \label{eq:con-first-deri}
  \lim_{\epsilon_1 \rightarrow 0}
  \E \big[ |\xi_{i}^{0,x_0+\epsilon_1 e_j}(t)-\xi_{i}^{0,x_0}(t)|^{p_7}\big] = 0, \quad  p_7 \in [1, \frac{p_0}{(2r \vee 0 )+1}],
 \end{equation}
where $e_j \in \R^d $ comes from Definition \ref{defi:mean_square_differential}. 
\end{prop}
\begin{proof}
We apply Corollary \ref{cor:general-SDE-form-con-beha} to derive \eqref{eq:con-first-deri}, which requires us to verify that the corresponding $\mathcal{H}_1$ and $\mathcal{H}_2$ satisfy \eqref{eq:ass-limit-behavior-H-fun}. Here 
\begin{eqnarray*}
\mathcal{H}_1 &=& 
 Df(X^{0,x_0+\epsilon_1 e_j}(s))\xi_i^{0,x_0+\epsilon_1 e_j}(s) 
 -D f(X^{0,x_0}(s))\xi_i^{0,x_0+\epsilon_1 e_j}(s), \notag \\
 \mathcal{H}_2 &=& Dg(X^{0,x_0+\epsilon_1 e_j}(s))\xi_i^{0,x_0+\epsilon_1 e_j}(s) 
 -Dg(X^{0,x_0}(s))\xi_i^{0,x_0+\epsilon_1 e_j}(s), \quad  \quad
 \mathcal{S}_{\epsilon}(0)=0.
\end{eqnarray*}
Using the H\"older inequality with  $\frac{1}{\theta_1} + \frac{1}{\theta_2} = 1 $, for $\theta_2 =\frac{p_0}{p_7}$ and \eqref{lem:mom-bound-first-deri}, we have
  \begin{eqnarray*}  
  & \E \big[\int_{0}^{t}
  |\mathcal{H}_1|^{p_7} \, \dd s
  \big] & \leq  \int_{0}^{t}
  \Big( \E \big[|
    D f \big(X^{0,x_0+\epsilon e_i}(s) \big) - Df(X^{0,x_0}(s))|^{p_7 \theta_1}\big]\Big)^{1/\theta_1} \big( \E \big[|\xi_i^{0,x_0+\epsilon_1 e_j}(s) |^{p_7\theta_2}\big] 
    \big)^{1/\theta_2} \, \dd s
    \notag \\
  & & \leq 
 C  \int_{0}^{t}
  \Big( \E \big[|
    D f \big(X^{0,x_0+\epsilon e_i}(s) \big) - Df(X^{0,x_0}(s))|^{p_7 \theta_1}\big]\Big)^{1/\theta_1} \, \dd s.
  \end{eqnarray*}
Therefore, from the Assumption \ref{ass:a1}, \eqref{eq:esti_sol_SDE_wrt_initial} and dominated convergence theorem, we get
\begin{equation*}
    \lim_{\epsilon \rightarrow 0}
    \E\big[\int_{0}^{t}
  |\mathcal{H}_1|^{p_7} \, \dd s \big] = 0,
  \quad 1 \leq p_7 \leq 
  \frac{p_0}
  {(2r\vee 0)+1},
  \quad \text{for all \,} t \in [0, T].
\end{equation*}
Similarly, one can show that 
\begin{equation*}
\lim_{\epsilon \rightarrow 0}
    \E\big[\int_{0}^{t}
  |\mathcal{H}_2|^{p_7 \vee 2} \, \dd s \big] = 0,
  \quad 1 \leq p_7 \leq 
  \frac{p_0}
  {((\rho-1)\vee 0)+1},
  \quad \text{for all \,} t \in [0, T].
\end{equation*}
Using the Corollary \ref{cor:well-posedness-deri-first}, we complete the proof.
\end{proof}
\begin{cor}
 \label{cor:well-posedness-deri-second}
 Under the same assumptions of Proposition \ref{lem:solution-wrt-inital-value-con} with $p_0 \geq 4r+4$,
for each $ i,j \in  \{1,\dots, d\} $, $ \zeta_{i,j}(t):= \frac{\partial^2 X^{0,x_0}(t)}{ \partial {x^{(i)}_{0}}\partial {x^{(j)}_{0}}} $ exists and 
 satisfies the following  equation
\begin{equation}
	\label{eq:second-variation-equation}
	\left\{
	\begin{aligned}
   \dd \, \zeta_{i,j}(t) & = 
     Df( X(t) ) \zeta_{i,j}(t) \, \dd t 
   + Dg(X(t)) \zeta_{i,j}(t) \, \dd W(t)
   + \dd \, \mathfrak{S}^{i,j}(t,x_0),
      \quad t \in (0,T], \\
      \zeta_{i,j}(0) & = 0,
	\end{aligned}\right.
\end{equation}
where $ \mathfrak{S}^{i,j}(t,x_0) $ is the process defined by 
\begin{equation*}
 \mathfrak{S}^{i,j}(t,x_0) = \int_{0}^{t} D^2f(X(s))(\xi_i(s), \xi_j(s)) \, \dd s
 + \int_{0}^{t} D^2g(X(s))(\xi_i(s), \xi_j(s)) \, \dd W(s).
\end{equation*}
Furthermore, there exists a constant $ C > 0 $, such that
for all $t \in [0,T]$,
 \begin{align}
   \label{lem:mom-bound-second-deri}
    & \E [|\zeta_{i,j}(t)|^{p_8}] 
      \leq 
      C
      \exp(C p_8 T) 
      (1+|x_0|^{((2r-1) \vee 0) p_8}),
    \quad 
   1 \leq p_8 \leq \frac{p_0}{((2r-1)\vee 0)+2}, \\
  \hbox{and}   \quad
   \label{lem:second-derivative}
   & \lim_{\epsilon_1 \rightarrow 0} 
  \E \big[ |\frac{\xi_i^{0,x_0+\epsilon_1 e_j}(t) - \xi_i^{0,x_0}(t)} {\epsilon_1} - 
  \zeta_{i,j}(t)|^{p_9} \big]   = 0, 
   \quad 
   1 \leq p_9 \leq \frac{p_0}{((2r-1)\vee 0)+(2r\vee 0 )+2}.
 \end{align}
\end{cor}
\begin{proof}
We apply Lemma \ref{lem:general-SDE-form-mom-bou} to derive \eqref{lem:mom-bound-second-deri}, which requires us to verify that the corresponding $\mathcal{H}_1$ and $\mathcal{H}_2$ satisfy Assumption \ref{ass:general-fun-growth-con}. Here 
\begin{equation*}
 \mathcal{H}_1 = D^2 f(X(s))(\xi_i(s), \xi_j(s)), \quad
 \mathcal{H}_2 = D^2 g(X(s))(\xi_i(s), \xi_j(s)), \quad 
 Z(0) = 0.
\end{equation*}
It follows from the H\"older inequality, Assumption \ref{ass:coefficient_function_assumption}, \eqref{eq:esti_sol_SDE_wrt_initial} and \eqref{lem:mom-bound-first-deri}  that for $p_0 \geq (2r+1) \vee 2$
\begin{equation}
\label{eq:secone-deri-H1-bound}
\E \big[\int_0^t
|\mathcal{H}_1(\omega,s)|^{p_8} \, \dd s
\big]
\leq C (1+|x_0|^{((2r-1) \vee 0) p_8}), \quad
 1 \leq p_8 \leq \frac{p_0}{((2r-1) \vee 0) +2}.
\end{equation}
By similar arguments in the proof of \eqref{eq:secone-deri-H1-bound}, we obtain  
\[
\E \big[\int_0^t
|\mathcal{H}_2(\omega,s)|^{p_8 \vee 2} \, \dd s
\big]
\leq C (1+|x_0|^{((\rho-2) \vee 0) p_8}), \quad
 1 \leq p_8 \leq \frac{p_0}{((2r-1) \vee 0) +2}.\]
Next, based on the Corollary \ref{cor:general-SDE-form-con-beha}, we prove
\eqref{lem:second-derivative}, where we set  
\begin{eqnarray*}
\mathcal{H}_1 &=& 
\frac{ 
\big(Df(X^{0,x_0+\epsilon_1 e_j}(s)) 
 -D f(X^{0,x_0}(s))\big)
 \xi_i^{0,x_0+\epsilon_1 e_j}(s)}{\epsilon_1} - D^2f(X^{0,x_0}(s))(\xi_i(s), \xi_j(s)), \notag \\
 \mathcal{H}_2 &=& 
\frac{ 
\big(Dg(X^{0,x_0+\epsilon_1 e_j}(s)) 
 -Dg(X^{0,x_0}(s))\big)
 \xi_i^{0,x_0+\epsilon_1 e_j}(s)}{\epsilon_1} - D^2g(X^{0,x_0}(s))(\xi_i(s), \xi_j(s)), \quad
 \mathcal{S}_{\epsilon}(0)=0.
\end{eqnarray*}
Using the Taylor expansion and the elementary inequality, we have
  \begin{eqnarray*} 
  \E\big[\int_{0}^{t}
  |\mathcal{H}_1|^{p_9} \, \dd s
  \big] & \leq & 
  C \E \big[ \int_{0}^{t}\big| 
  \int_{0}^{1} D^2 f ( \lambda (s)) \, \dd \lambda 
  \big( \frac{
   X^{0,x_0+\epsilon_1 e_j}(s) - X^{0,x_0}(s) }{\epsilon_1} -
    \xi_i^{0,x_0+\epsilon_1 e_j}(s) \big) 
    \xi_i^{0,x_0+\epsilon_1 e_j}(s) \big|^{p_9}  \, \dd s 
   \big]  \notag \\
  & & +
C \E \big[ \int_{0}^{t}\big| 
   \int_{0}^{1} D^2 f ( \lambda (s)) \, \dd \lambda 
      \big(\xi_i^{0,x_0+\epsilon_1 e_j}(s)-\xi_i^{0,x_0}(s)\big) 
    \xi_i^{0,x_0}(s) \big|^{p_9}  \, \dd s 
   \big]  \notag \\
 & & +
C \E \big[ \int_{0}^{t}\big| \big(
    D^2 f (\lambda(s)) - D^2 f(X^{0,x_0}(s)) \big)
      \big(\xi_j^{0,x_0}(s), \xi_i^{0,x_0}(s)\big) 
     \big|^{p_9}  \, \dd s 
   \big]   \notag \\
   & =: & I_1 + I_2 + I_3,
  \end{eqnarray*}
where $\lambda(s)= \lambda X^{0,x_0}(s)+ (1-\lambda)X^{0,x_0+\epsilon_1 e_j}(s) $.
By the H\"older inequality, Assumption \ref{ass:coefficient_function_assumption}, Corollary \ref{cor:well-posedness-deri-first} and the dominated convergence theorem, we have that for $p_0 \geq 4r+4$
\begin{equation}
\label{eq:lim-beha-I1}
    \lim_{\epsilon_1 \rightarrow 0}
    I_1 = 0,
  \quad 1 \leq p_9 \leq \frac{p_0}{((2r-1)\vee 0)+(2r\vee 0 )+2},
  \quad \text{for all \,} t \in [0, T].
\end{equation}
Similar arguments lead to $\lim_{\epsilon_1 \rightarrow 0}
    I_2 = 0$ and $\lim_{\epsilon_1 \rightarrow 0}
    I_3 = 0$, which indicates that
$ \lim_{\epsilon_1 \rightarrow 0} \E\big[\int_{0}^{t}
  |\mathcal{H}_1|^{p_9} \, \dd s
  \big] =0 $.
Similarly, we have 
\begin{equation*}
    \E\big[\int_{0}^{t}
  |\mathcal{H}_2|^{p_9 \vee 2} \, \dd s \big] = 0,
  \quad 1 \leq p_9 \leq \frac{p_0}{((\rho-2)\vee 0)+(2r\vee 0 )+2},
  \quad \text{for all \,} t \in [0, T],
\end{equation*}
which, together with Corollary \ref{cor:well-posedness-deri-first}, completes the proof.
\end{proof}
\textit{Proof of Lemma \ref{lem:deri-mom-bound}.}
The proof is conducted by induction with the $j = 1$ case in Corollary \ref{cor:well-posedness-deri-first} as the base step, and $j = 2$ case in Corollary \ref{cor:well-posedness-deri-second} as the first induction step. We omit the details here.
\qed

\begin{rem}
As discussed in Remark \ref{rem:assumption-difference},  we prove Lemma \ref{lem:deri-mom-bound} under the relaxed condition compared to \cite{Cerrai2001second}.
\end{rem}
\section{Proof of the approximation theorem}
\label{append-proof-weak-conver-theorem}
To carry out the weak error analysis,
we introduce  the function $u \colon [0, T] \times \R^d \rightarrow \R$
  defined by $
u(t,x)
= \E  \big[ \varphi \big(X(t,x;T)\big) \big]$.
We first prove a useful result.
\begin{lem}[One-step error]
\label{lem:equivent_condition}
Under the same condition of Theorem \ref{thm:fundamental_weak_convergence},
function $u$ 
fulfills 
\begin{equation}
\label{eq:general-one-step-estimate-bound}
 \sup_{ t \in [0,T - h]} \big | \E \big[ u(t + h , X(t, x; t+h) ) \big]
- 
\E \big[ u ( t + h , Y(t, x; t+h) )  \big]\big| 
\leq 
C ( 1 + | x |^{\beta \kappa + \varkappa} ) 
h^{q+1},
\end{equation}
where $\beta $, $\varkappa$ come from assumption \ref{ass:ii} and \ref{ass:iii} in Theorem \ref{thm:fundamental_weak_convergence}, $h > 0$ and $ t + h \leq T$.
\end{lem}
\begin{proof}
We recall that $\varphi \in \mathbb{G}^{2q+2}$. By the
results in Corollary \ref{cor:well-posedness-deri-first} and Lemma \ref{lem:deri-mom-bound} together with chain rule we could show that 
$u (t, x) $ is $ 2q + 2 $ times differentiable with respect to $x$ and
\begin{equation}
\label{eq:kth-partial-u}
\big|
\tfrac{\partial ^{k} u ( t + h, x)}{\partial x_{i_{1}}
   \ldots \partial x_{i_{k}}}
\big|
\leq
C ( 1 + | x |^{ \kappa } ),
\quad
k = 1,2,\cdots, 2q+2.
\end{equation}
By the Taylor expansion, we get 
\begin{eqnarray}
\label{eq:u-Taylor-expansion}
 && \quad
     \E \big[ u(t + h , X(t, x; t+h) ) \big]
- 
\E \big[ u ( t + h , Y(t, x; t+h) )  \big] \notag \\
 & = &
    \sum^{2q+1}_{k=1}
    \sum^{d}_{i_{1},\ldots,i_{k} = 1}
   \frac{1}{k!} 
   \frac{\partial ^{k} u ( t + h, x)}{\partial x_{i_{1}}
   \ldots \partial x_{i_{k}}}
   \E \Big[ \prod^{k}_{j=1} (\delta_{X,x})^{i_{j}}
    - \prod^{k}_{j=1} (\delta_{Y,x})^{i_{j}} \Big ]
    + \E [\mathcal{R}_{2q+2}],
\end{eqnarray}
where we denote
$
\big( \delta_{X,x} \big)^{\alpha}
   := \prod_{j=1}^{2q+2}  (\delta_{X,x})^{i_j},\,
 \big( \delta_{Y,x} \big)^{\alpha}
 := \prod_{j=1}^{2q+2}  (\delta_{Y,x})^{i_j},
 \,
 \alpha = ( i_1, i_2, ..., i_{2q+2} )
$
and
\begin{eqnarray*}
  \mathcal{R}_{2q+2} &:= & \sum_{|\alpha|=2p+2}
  \frac{|\alpha|}{\alpha!}
  \bigg( 
  \int_0^1 (1-s)^{|\alpha|-1}D^{\alpha}
   u \big( t+ h, x + s \delta_{X,x} \big) \dd s
   \, ( \delta_{X,x} )^{\alpha}  
   \notag \\
& &\qquad
     - 
  \int_0^1 (1-s)^{|\alpha|-1}D^{\alpha}
   u \big( t +h, x + s \delta_{Y,x} \big) \dd s
    \, ( \delta_{Y,x} )^{\alpha
    } 
    \bigg).
\end{eqnarray*}
Combining the assumption \ref{ass:ii} with \eqref{eq:kth-partial-u} yields
\begin{equation} 
\label{eq:Taylor-expansion-estimate1}
\sum^{2q+1}_{k=1}\sum^{d}_{i_{1},\ldots,i_{k} = 1}
\frac{1}{k!} \Big | \frac{\partial ^{k} u ( t+h, x )}{\partial x_{i_{1}}
	\ldots \partial x_{i_{k}}} \Big|
	\cdot
\Big| \E \Big[ \prod^{k}_{j=1} (\delta_{X,x})^{i_{j}}
- \prod^{k}_{j=1} (\delta_{Y,x})^{i_{j}} \Big ] \Big| 
\leq
C ( 1 + | x |^{\kappa+\varkappa} )
h^{p+1} .
\end{equation}
Using the H\"older inequality,
\eqref{eq:esti_sol_SDE_wrt_initial}, assumption \ref{ass:iii}
and
\eqref{eq:kth-partial-u}, one can derive 
\begin{eqnarray}
 \label{eq:Taylor-expansion-estimate2}
\big | \E [ \mathcal{R}_{2q+2}] \big|
  & \leq &
  \sum_{|\alpha|=2q+2} 
  C  \int_0^1 \| D^{\alpha}
   u \big( t +h,  (1-s)x + sX(t,x;t+h)\big) \|_{L^2 (\Omega; \R)} \dd s
   \| ( \delta_{X,x} )^{\alpha}
   \|_{L^2 (\Omega; \R)}  \notag \\
&& \quad +
  \sum_{|\alpha|=2q+2} 
  C  \int_0^1 \| D^{\alpha}
   u \big( t+h, (1-s)x + sY(t,x;t+h)\big) \|_{L^2 (\Omega; \R)} 
   \dd s
    \| ( \delta_{Y,x} )^{\alpha} 
   \|_{L^2 (\Omega; \R)} \notag \\
& \leq &
C ( 1 + | x |^{ \beta  \kappa+\varkappa} ) h^{p+1},
\end{eqnarray}
where we use the assumption
	$ p_1 = 2 \kappa $, $ p = 2 \kappa$, 
	and $\mathds{P} = 4 q + 4$.
Plugging \eqref{eq:Taylor-expansion-estimate1} and 
\eqref{eq:Taylor-expansion-estimate2} into \eqref{eq:u-Taylor-expansion} , we get
\begin{equation*}
\big | \E \big[ u(t + h , X(t, x; t+h) ) \big]
- 
\E \big[ u ( t + h , Y(t, x; t+h) )  \big]\big| 
\leq 
C ( 1 + | x |^{\beta \kappa + \varkappa} ) 
h^{p+1}.
\end{equation*}
The proof of the lemma is thus complete.
\end{proof}
\textit{Proof of Theorem \ref{thm:fundamental_weak_convergence}.}
The proof follows the main ideas from \cite{Milstein1985Weak,Milstein2004stochastic,wang2021weak}. Here, we show the key steps of the proof. Fix $\varphi \in \mathbb{G}^{2q+2}$ and define $u(t,x) =\E  \big[ \varphi \big(X(t,x;T)\big) \big]$. Then, we find that
 \begin{equation*}
\E \big [ \varphi \big (X(t_0,X_0;T)
  	    \big ) \big ]
  	     -\E \big[
  	      \varphi \big (
     	Y(t_0,Y_0; t_N)
  	   \big )\big ]
= \E [ u ( t_0, X_0) ] - \E [u (t_N, Y_N )],
\end{equation*}
and 
\[\E \big[ u(t_0,X_0) \big]= 
		\E \big[ u(t_1 , X( t_{1} ) ) \big]
		- \E \big[ u(t_1 , Y_1)  \big]
		+ \E \big[ u( t_{2}, X(t_1,Y_1; t_2) )  \big]. \]
Continuing this process and subtracting $\E \big[ u(t_N,Y_N) \big] $, we get
\begin{equation}
  \label{eq:error_test_function}
	\E \big[ u(t_0,X_0) \big] - \E \big[ u(t_N,Y_N) \big] = 
		\sum_{i=0}^{N-1}
		\Big( \E \big[ u(t_{i+1} , X( t_{i}, Y_{i}; t_{i+1} ) ) \big]
		- \E \big[ u(t_{i+1} , Y( t_{i}, Y_{i}; t_{i+1} ) )  
		\big] \Big),
\end{equation}
and thus 
\begin{eqnarray}
\label{eq:expan-global-weak-error}
 & & \big | \E \big [
  	 \varphi \big (
    	X(t_0,X_0;T)
  	    \big ) \big ]
  	     -\E \big[
  	      \varphi \big (
     	Y(t_0,Y_0; t_N)
  	   \big )\big ] \big |  \notag \\
  & \leq &   
\sum_{i=0}^{N-1}
		\Big| \E \big[ u(t_{i+1} , X( t_{i}, Y_{i}; t_{i+1} ) ) \big]
		- \E \big[ u(t_{i+1} , Y( t_{i}, Y_{i}; t_{i+1} ) )  \big] \Big|
 \notag \\
& = & 
\sum_{i=0}^{N-1}
		\Big| \E \Big[  \E \big( u(t_{i+1} , X( t_{i}, Y_{i}; t_{i+1} ) ) - u(t_{i+1} , Y( t_{i}, Y_{i}; t_{i+1} ) ) | \sigma(Y_i) \big) \Big] \Big|.
\end{eqnarray}
Finally, combining a conditional version of \eqref{eq:general-one-step-estimate-bound} with \eqref{eq:expan-global-weak-error} yields
\begin{equation}
 \big | \E \big [
           \varphi \big ( X(t_0,X_0;T) \big ) 
                      \big ]
         - \E \big[
              \varphi \big ( Y(t_0,Y_0;t_N) \big )
                      \big ] \big |
                 \leq 
     C \big(1+|X_0|^{\beta (\beta \kappa + \varkappa)}\big)
             T h^{q},
\end{equation}
which validates the desired assertion \eqref{thm:convergence_order}.
\qed
\section{Proofs of moment bounds of modified Euler schemes}
\label{append-modified-moment-bound}
\textit{Proof of Lemma \ref{lem:Euler-scheme-moment-bound}.}
In the proof, we shall use the letter $ C $ to denote various constants independent of $ h $ and $ n $.
Let $\mathcal{R} > 0$ be sufficiently large and define a sequence of decreasing subevents
\begin{equation*}
\Omega_{\mathcal{R}, n}
  := 
  \left\{
  \omega \in \Omega:
 \sup_{0 \leq i \leq n}
  \left|Y_{i}(\omega)\right| \leq \mathcal{R} \right\}, 
  \quad  n=0,1, \ldots, N, \quad
   N \in \mathbb{N},
\end{equation*}
and their compliments $ \Omega_{ \mathcal{R}, n}^{c} $. At first, we show that the boundedness of high-order moments is valid within a family of appropriate subevents.
For integer $ \bar{p} \geq 1 $, We have 
\begin{eqnarray}
 \label{eq:decom-balanced-euler-sub}
	 \E \big[ \mathds{1}_{\Omega_{ \mathcal{R}, n+1}} |Y_{n+1}|^{\bar{p}}  \big] 
  & \leq &
   \E \big[ \mathds{1}_{\Omega_{ \mathcal{R}, n}} |Y_{n}|^{\bar{p}}  \big]
  +
   \E \big[ \mathds{1}_{\Omega_{ \mathcal{R}, n}} |Y_{n}|^{\bar{p}-2} 
    \big( \bar{p} \langle Y_n, Y_{n+1}-Y_n \rangle
      +\frac{\bar{p} (\bar{p}-1)}{2} | Y_{n+1}-Y_n |^2
    \big) \big]  \notag \\
 && + C \sum_{l=3}^{\bar{p}}
   \E \big[ \mathds{1}_{\Omega_{ \mathcal{R}, n}} 
   |Y_{n}|^{\bar{p}-l} 
   |Y_{n+1} - Y_n |^{l}  \big].
\end{eqnarray}
Consider the second term in the right-hand side of \eqref{eq:decom-balanced-euler-sub}:
\begin{eqnarray*}
    I_1 
    & = &
    \E \big[ \mathds{1}_{\Omega_{ \mathcal{R}, n}} |Y_{n}|^{\bar{p}-2} 
    \big(\bar{p} \langle Y_n, Y_{n+1}-Y_n \rangle
      + \frac{\bar{p} (\bar{p}-1)}{2} | Y_{n+1}-Y_n |^2
    \big) \big]  \notag \\
   & = &
  \bar{p} \E \big[ \mathds{1}_{\Omega_{ \mathcal{R}, n}} |Y_{n}|^{\bar{p}-2} 
      \langle Y_n, \mf{ Y_n } h-f(Y_n)h \rangle
      \big] 
  + \frac{\bar{p} (\bar{p}-1)}{2} \E \big[ \mathds{1}_{\Omega_{ \mathcal{R}, n}} 
   |Y_{n}|^{\bar{p}-2} 
   |\mf{ Y_n }h|^{2}  \big]   \notag \\
  && +
    \bar{p}
   \E \big[ \mathds{1}_{\Omega_{ \mathcal{R}, n}} |Y_{n}|^{\bar{p}-2} 
     \big( \langle Y_n, f(Y_n)h \rangle 
  + \frac{(\bar{p}-1)}{2} 
   \big|\sum_{r=1}^{m} 
      \mg{ Y_n }   \Delta W_r(n)
   \big|^{2} \big)  \big]. 
\end{eqnarray*}
Using the Schwarz inequality, \ref{ass:tame-h1}, \ref{ass:tame-h2} in Assumption \ref{ass:tame_function_assumption}, \eqref{eq:co_mono_condi} and \eqref{eq:f-growth} yields
\begin{eqnarray}
  \label{eq:the-result-I1}
    I_1    
  & \leq &
   C h + Ch
  \E \big[ \mathds{1}_{\Omega_{ \mathcal{R}(h), n}} |Y_{n}|^{\bar{p}} \big]  \notag \\
  && + 
  \bar{p} \E \big[ \mathds{1}_{\Omega_{ \mathcal{R}, n}} |Y_{n}|^{2 \bar{p}-1} 
      \big| \mf{ Y_n }-f(Y_n)\big| h \big] 
  + \frac{\bar{p} (\bar{p}-1)}{2} \E \big[ \mathds{1}_{\Omega_{ \mathcal{R}, n}} 
   |Y_{n}|^{ 2\bar{p}-2} 
   |\mf{ Y_n }h|^{2}  \big]   \notag \\
   & \leq & 
  C h 
 + Ch
  \E \big[ \mathds{1}_{\Omega_{ \mathcal{R}(h), n}} |Y_{n}|^{\bar{p}} \big] 
 + Ch^2 \E \big[ \mathds{1}_{\Omega_{ \mathcal{R}(h), n}} |Y_{n}|^{4r+\bar{p}} \big]  \notag \\
  && + 
  Ch^{\tau+1} \E \big[ \mathds{1}_{\Omega_{ \mathcal{R}(h), n}} |Y_{n}|^{(2r+1)l_1+\bar{p}-1} \big],
\end{eqnarray}
where $\bar{p} \leq p_0' - \varepsilon$.
Now consider the last term in \eqref{eq:decom-balanced-euler-sub}:
\begin{eqnarray}
  \label{eq:esti-I2}
   I_2 
   & = & 
   C \sum_{l=3}^{\bar{p}}
   \E \big[ \mathds{1}_{\Omega_{ \mathcal{R}, n}} 
   |Y_{n}|^{\bar{p}-l} 
   |Y_{n+1} - Y_n |^{l}  \big]   \notag \\
   & \leq &  
   C \sum_{l=3}^{\bar{p}}
   \E \big[ \mathds{1}_{\Omega_{ \mathcal{R}, n}} 
   |Y_{n}|^{\bar{p}-l} 
   \big( h^l |\mf{ Y_n }|^l + h^{\frac{l}{2}}
    \sum_{r=1}^{m}| \mg{ Y_n } |^l
   \big) \big]    \notag \\
& \leq & 
  Ch
   + 
   C \sum_{l=3}^{\bar{p}} h^l
   \E \big[ \mathds{1}_{\Omega_{ \mathcal{R}, n}} |Y_{n}|^{2rl+\bar{p}} \big]  
   + 
   C \sum_{l=3}^{\bar{p}} h^{\frac{l}{2}}
   \E \big[ \mathds{1}_{\Omega_{ \mathcal{R}, n}} |Y_{n}|^{rl+\bar{p}} \big].
\end{eqnarray}
Combining \eqref{eq:the-result-I1} and \eqref{eq:esti-I2}, we obtain
\begin{eqnarray}
\label{eq:result-Yn-sub}
   \E \big[ \mathds{1}_{\Omega_{ \mathcal{R}, n}} |Y_{n+1}|^{\bar{p}}  \big] 
  & \leq &
  Ch
 +
  ( 1 + Ch )  
  \E \big[ \mathds{1}_{\Omega_{ \mathcal{R}, n}} |Y_{n}|^{\bar{p}} \big] 
  +  Ch^2 \E \big[ \mathds{1}_{\Omega_{ \mathcal{R}, n}} |Y_{n}|^{4r+\bar{p}} \big]  \notag \\
  && + 
  Ch^{\tau+1} \E \big[ \mathds{1}_{\Omega_{ \mathcal{R}, n}} |Y_{n}|^{(2r+1) l_1+\bar{p}-1} \big]  \notag  \\
  && +
  C \sum_{l=3}^{\bar{p}} h^l
   \E \big[ \mathds{1}_{\Omega_{ \mathcal{R}, n}} |Y_{n}|^{2rl+\bar{p}} \big]  
   + 
   C \sum_{l=3}^{\bar{p}} h^{\frac{l}{2}}
   \E \big[ \mathds{1}_{\Omega_{ \mathcal{R}, n}} |Y_{n}|^{rl+\bar{p}} \big].
\end{eqnarray}
Choosing $\mathcal{R} = \mathcal{R}(h)=h^{-1/\mathds{G}_1}$ with 
$ \mathds{G}_1 = 6r \vee \frac{(2r+1)l_1-1}{\tau} $,
one can immediately show that for all $ l=3, \ldots, \bar{p} $,
 \begin{align*}
    \mathds{1}_{\Omega_{ \mathcal{R}, n}}
    \big| Y_n \big|^{(2r+1)l_1+\bar{p}-1} 
    h^{\tau}  
  & = 
  \mathds{1}_{\Omega_{ \mathcal{R}, n}}
    \big| Y_n \big|^{\bar{p}} 
    \big( 
    \mathds{1}_{\Omega_{ \mathcal{R}, n}}
     \big| Y_n \big|^{\frac{(2r+1)l_1-1}{\tau}} 
     h \big)^{\tau} 
    \leq 
    C 
    \mathds{1}_{\Omega_{ \mathcal{R}, n}}
     \big| Y_n \big|^{\bar{p}}, \\
\mathds{1}_{\Omega_{ \mathcal{R}, n}}
    \big| Y_n \big|^{4r+\bar{p}} h
  & = 
  \mathds{1}_{\Omega_{ \mathcal{R}, n}}
    \big| Y_n \big|^{\bar{p}} 
    \big( 
    \mathds{1}_{\Omega_{ \mathcal{R}, n}}
     \big| Y_n \big|^{4r} 
     h \big)
   \leq 
    C 
    \mathds{1}_{\Omega_{ \mathcal{R}, n}}
     \big| Y_n \big|^{\bar{p}}, \\ 
\mathds{1}_{\Omega_{ \mathcal{R}, n}}
    \big| Y_n \big|^{2rl+\bar{p}}
    h^{l-1}
  & = 
  \mathds{1}_{\Omega_{ \mathcal{R}, n}}
    \big| Y_n \big|^{\bar{p}} 
    \big( 
    \mathds{1}_{\Omega_{ \mathcal{R}, n}}
     \big| Y_n \big|^{\frac{2rl}{l-1}} 
     h \big)^{l-1} 
   \leq 
    C 
    \mathds{1}_{\Omega_{ \mathcal{R}, n}}
     \big| Y_n \big|^{\bar{p}}, \\ 
   \mathds{1}_{\Omega_{ \mathcal{R}, n}}
    \big| Y_n \big|^{rl+\bar{p}}
    h^{\frac{l-2}{2}} 
  & = 
  \mathds{1}_{\Omega_{ \mathcal{R}, n}}
    \big| Y_n \big|^{\bar{p}} 
    \big( 
    \mathds{1}_{\Omega_{ \mathcal{R}, n}}
     \big| Y_n \big|^{\frac{2rl}{l-2}} 
     h \big)^{\frac{l-2}{2}} 
   \leq 
    C 
    \mathds{1}_{\Omega_{ \mathcal{R}, n}}
     \big| Y_n \big|^{\bar{p}},
 \end{align*}
where the constants $ C $ are independent of the step size $ h $. Thus we have for \eqref{eq:result-Yn-sub}
\begin{equation*}
\E \big[ \mathds{1}_{\Omega_{ \mathcal{R}, n+1}} |Y_{n+1}|^{\bar{p}}  \big] 
  \leq 
  \E \big[ \mathds{1}_{\Omega_{ \mathcal{R}, n}} |Y_{n+1}|^{\bar{p}}  \big]   
   \leq 
  Ch
 +
  ( 1 + Ch )  
  \E \big[ \mathds{1}_{\Omega_{ \mathcal{R}, n}} |Y_{n}|^{\bar{p}} \big]. 
\end{equation*}
The discrete Gronwall inequality \cite[Lemma 1.6]{Zhang2017preserving} shows that 
\begin{equation}
  \label{eq:balanced-euler-result-moment-bound-sub}
     \E \big[ \mathds{1}_{\Omega_{ \mathcal{R}, n}} |Y_n|^{\bar{p}}  \big] 
   \leq 
 \exp(Cnh) \big( 1 + | Y_{0} | ^{\bar{p}} \big)
  \leq
     C 
  \big( 1 + | Y_{0} | ^{\bar{p}} 
  \big). 
\end{equation}
It remains to estimate $ \E \big[ \mathds{1}_{\Omega_{ \mathcal{R}, n}^{c}} |Y_n|^{p}  \big] $.
It follows from \eqref{eq:Euler-scheme}, \ref{ass:tame-h1} in Assumption \ref{ass:tame_function_assumption} that
\begin{eqnarray}
  \label{eq:decom_inequality_Yn}
    | Y_{n+1} | 
 &  \leq &
| Y_n | + Ch^{1-\gamma_1} + \sum_{r=1}^{m} 
  C h^{-\gamma_2}
\big|W_r(t_{n+1})-W_r(t_n)\big|
  \notag \\
 & \leq &
   | Y_0 | 
 + (n+1)Ch^{1-\gamma_1} 
 +  
  \sum_{k=0}^{n}  \sum_{r=1}^{m} 
 C h^{-\gamma_2}
 \big|W_r(t_{k+1})-W_r(t_k)\big|.
\end{eqnarray}
Note that
 $\mathds{1}_{\Omega_{ \mathcal{R}, n}^{c}}
   = 
  \mathds{1}_{\Omega_{ \mathcal{R}, n-1}^{c}}
 + 
 \mathds{1}_{\Omega_{ \mathcal{R}, n-1}}
 \mathds{1}_{|Y_n| > \mathcal{R}} 
 = 
  \sum_{i=0}^{n}
  \mathds{1}_{\Omega_{ \mathcal{R}, i-1}}
 \mathds{1}_{|Y_i| > \mathcal{R}},
$
where we set $ \mathds{1}_{\Omega_{ \mathcal{R}(h), -1} }= 1 $. This together with the H\"older inequality with 
$ \frac{1}{p'} + \frac{1}{q'} =1 $ 
for 
$ q' = \frac{\bar{p}}{(p\gamma_1+1)\mathds{G}_1} \vee \frac{\bar{p}}{(1+\frac12 p+ p\gamma_2)\mathds{G}_1}  > 1 $,
due to
$ p \leq \frac{ \bar{p}-  \mathds{G}_1}{1+\gamma_1 \mathds{G}_1} \wedge \frac{ \bar{p}-\mathds{G}_1}{1+(\frac12+\gamma_2)\mathds{G}_1}$, and the Chebyshev inequality give
\begin{eqnarray}
  \label{eq:resu_decom_Yn_subevent}
   \E 
      \big[ \mathds{1}_{\Omega_{ \mathcal{R}(h), n}^{c}}
   |Y_{n}|^{p}\big]  
  & = &
   \sum_{i=0}^{n}
   \E 
    \big[ |Y_{n}|^{p}
  \mathds{1}_{\Omega_{ \mathcal{R}, i-1}}
  \mathds{1}_{|Y_i|>\mathcal{R}}
   \big]    \notag \\
 & \leq &
  \sum_{i=0}^{n} \big(
   \E \big[ |Y_{n}|^{p p'} \big]
   \big)^{\frac{1}{p'}}
   \big(  \E \big[
   \mathds{1}_{\Omega_{ \mathcal{R}, i-1}}
  \mathds{1}_{|Y_i|>\mathcal{R}}
   \big] \big)^{\frac{1}{ q'}}   \notag \\
  & = &
  \E \big[ |Y_{n}|^{p p'} \big]
   \big)^{\frac{1}{p'}}
   \sum_{i=0}^{n}
   \big(  \P \big(
   \mathds{1}_{\Omega_{ \mathcal{R}, i-1}}
    |Y_i|>\mathcal{R}
   \big) \big)^{\frac{1}{ q'}}   \notag \\
 & \leq &
  \big( \E \big[ |Y_{n}|^{p p'} \big]
   \big)^{\frac{1}{p'}}
   \sum_{i=0}^{n}
   \frac{\big(  \E \big[
   \mathds{1}_{\Omega_{ \mathcal{R}, i-1}}
    |Y_i|^{\bar{p}}
   \big] \big)^{\frac{1}{ q'}}}
   {\mathcal{R}^{\bar{p}/q'}}.
\end{eqnarray}
Since $p \leq \frac{ \bar{p}-\mathds{G}_1}{1+\gamma_1 \mathds{G}_1} \wedge \frac{ \bar{p}-\mathds{G}_1}{1+(\frac12+\gamma_2)\mathds{G}_1}$ implies $ p p' \leq \bar{p} $, using the H\"older inequality, \eqref{eq:decom_inequality_Yn} and the elementary inequality implies
\begin{eqnarray} \label{eq:first_part_decom_Yn_subevent}
    \big( \E \big[ |Y_{n}|^{p p'} \big]
   \big)^{\frac{1}{p'}} 
  & \leq & 
   \big( \E \big[ |Y_{n}|^{\bar{p}} \big]
   \big)^{\frac{p}{\bar{p}}} 
  \leq 
  C \Big( |Y_0|^{\bar{p}} 
  + n^{\bar{p}}  h^{\bar{p}(1-\gamma_1)}
  + h^{-\bar{p}\gamma_2}
  \E \big[ \big| \sum_{k=0}^{n-1} |
  W_r(t_{k+1}) - W_r(t_k)|
  \big|^{\bar{p}} \big]  \Big)^{\frac{p}{\bar{p}}}\notag \\
 & \leq & 
   C (h^{-p \gamma_1} + h^{-\frac{p}{2}-p \gamma_2} ) 
 + C \big( 1 + |X_0|^{\bar{p}} \big)
   ^{\frac{p}{\bar{p}}}.
\end{eqnarray}
Inserting \eqref{eq:first_part_decom_Yn_subevent} into \eqref{eq:resu_decom_Yn_subevent} and exploiting the H\"older inequality, $\mathcal{R}=h^{-1/\mathds{G}_1}$ and \eqref{eq:balanced-euler-result-moment-bound-sub}, we deduce
\begin{align}
 \label{eq:EM-result-subevent-mom-bound}
    \E 
      \big[ \mathds{1}_{\Omega_{ \mathcal{R}, n}^{c}}
   |Y_{n}|^{p}\big]  
  & \leq 
 C (n+1) h^{\frac{\bar{p}} {\mathds{G}_1 q'}} 
  \big( h^{-p\gamma_1} + h^{-\frac{p}{2}-p \gamma_2} 
  + 
    \big( 1 + |X_0|^{\bar{p}} \big)
   ^{\frac{p}{\bar{p}}} \big)
  \big( 1+|X_0|^{\bar{p}} 
   \big)^{\frac{1}{q'}}   
%
 \leq C \big( 1 + |X_0|^{\bar{p}} \big)
   ^{\frac{p}{\bar{p}} + \frac{1}{q'}}.
\end{align}
This together with \eqref{eq:balanced-euler-result-moment-bound-sub} implies
\begin{align}
    \E  \big[ 
   \left|Y_{n}\right|^{p}
   \big] 
 & = 
 \E 
      \big[ \mathds{1}_{\Omega_{ \mathcal{R}, n}}
   |Y_{n}|^{p}\big]  
  +
     \E 
      \big[ \mathds{1}_{\Omega_{ \mathcal{R}, n}^{c}}
   |Y_{n}|^{p}\big]   
%
\leq 
  \big(
  \E \big[ \mathds{1}_{\Omega_{ \mathcal{R}, 
  n}} |Y_{n}|^{\bar{p}}  \big] 
   \big)^{\frac{p}{\bar{p}}}
 + \E 
      \big[ \mathds{1}_{\Omega_{ \mathcal{R}, n}^{c}}
   |Y_{n}|^{p}\big]   
 \leq  C 
  \big( 1+|X_0|^{p} \big)^{\beta},
\end{align}
for all integer $ p \in \big[ 1, \frac{ \bar{p}- \mathds{G}_1}{1+ \gamma_1 \mathds{G}_1} \wedge \frac{\bar{p}- \mathds{G}_1}{1+ (\frac12 + \gamma_2) \mathds{G}_1}\big]$. Here $\beta =1+ \frac{\bar{p}}{pq'} = 1+ \frac{(p\gamma_1+1)\mathds{G}_1}{p} \wedge \frac{(1+\frac12 p+ p\gamma_2)\mathds{G}_1}{p} $.
Then, by Jensen's inequality, \eqref{eq:balanced-euler-result-moment-bound} holds for real  $p$ as well.

The proof of bounded moments for the scheme \eqref{eq:Euler-scheme} under \ref{ass:tame-h2p} is nearly identical to that under \ref{ass:tame-h2} in Assumption \ref{ass:tame_function_assumption}, except for the second term $I_1$ on the right-hand side of \eqref{eq:decom-balanced-euler-sub}. By applying  \eqref{con:balanced-EM-sublinear-fghx} and assumption \ref{ass:tame-h2p}, we can obtain
\begin{align*}
    I_1 
   & = 
    \bar{p} h \E \big[ \mathds{1}_{\Omega_{ \mathcal{R}, n}} |Y_{n}|^{\bar{p}-2} 
     \big(
      \langle Y_n, \mathcal{T}_1(f(Y_n),h) \rangle
   + \frac{(\bar{p}-1)}{2} 
   | \sum_{r=1}^{m}\mg{Y_n}
   |^{2} \big) \big] \notag \\
&
 + \frac{\bar{p} (\bar{p}-1)}{2} \E \big[ \mathds{1}_{\Omega_{ \mathcal{R}, n}} 
   |Y_{n}|^{\bar{p}-2} 
   |\mathcal{T}_1(f(Y_n),h)h|^{2}  \big]   \notag \\
  & \leq 
  C h 
 + Ch
  \E \big[ \mathds{1}_{\Omega_{ \mathcal{R}, n}} |Y_{n}|^{\bar{p}} \big] 
 + Ch^2 \E \big[ \mathds{1}_{\Omega_{ \mathcal{R}, n}} |Y_{n}|^{4r+\bar{p}} \big],
\end{align*}
where $ \bar{p} \leq p_{\mathcal{T}} $ with $p_{\mathcal{T}}$ from \ref{ass:tame-h2p} or \eqref{con:one-sided-lip-con-balanced}.
Similar to the above proof, we choose $\mathcal{R} = h^{-1/\mathds{G}_1}$, where $\mathds{G}_1 = 6r$ and we obtain the desired conclusion.
\qed
\end{document}